\documentclass{conm-p-l}

\usepackage{amssymb}

\usepackage[cmtip,all]{xy}

\newtheorem{theorem}{Theorem}[section]
\newtheorem{lemma}[theorem]{Lemma}
\newtheorem{cor}[theorem]{Corollary}
\newtheorem{prop}[theorem]{Proposition}
\theoremstyle{definition}
\newtheorem{definition}[theorem]{Definition}
\newtheorem{example}[theorem]{Example}

\theoremstyle{remark}
\newtheorem{remark}[theorem]{Remark}

\numberwithin{equation}{section}

\newcommand{\K}{{k}}
\newcommand{\Sym}{{\mathrm{Sym}}}
\newcommand{\scriptEnd}{\mathcal E\!\mathit{nd}}
\newcommand{\scriptHom}{\mathcal H\!\mathit{om}}
\newcommand{\dow}{\downarrow}
\newcommand{\upa}{\uparrow}
\begin{document}
\title{The classical master equation}
\author{Giovanni Felder}
\address{Department of mathematics,
ETH Zurich, 8092 Zurich, Switzerland}
\email{felder@math.ethz.ch}
\author{David Kazhdan}
\address{Einstein Institute of Mathematics,
The Hebrew University of Jerusalem,
Jerusalem 91904, Israel}
\email{kazhdan@math.huji.ac.il}
\contrib[with an appendix by]{Tomer M. Schlank} 
\address{Department of mathematics, 
MIT,
Cambridge, MA 02139-4307}
\email{schlank@math.mit.edu}


\date{20 December 2012}

\begin{abstract}
We formalize the construction by Batalin and Vilkovisky of a solution of the
classical master equation associated with a regular
function on a nonsingular affine variety (the classical action). 
We introduce the notion of stable
equivalence of solutions and prove that a solution exists and is unique up to
stable equivalence. A consequence is that the associated BRST cohomology, with
its structure of Poisson${}_0$-algebra, is independent of choices and is
uniquely determined up to unique isomorphism by the classical action.
We give a geometric interpretation of the BRST cohomology sheaf in degree
0 and 1 as the cohomology of a Lie--Rinehart algebra associated with the
critical locus of the classical action. Finally we consider
the case of a quasi-projective varieties and show that
the BRST sheaves defined on an open affine cover can be glued 
to a sheaf of differential Poisson${}_0$-algebras.
\end{abstract}
\maketitle

\tableofcontents
\section{Introduction}\label{s-1}
Batalin and Vilkovisky
{}\cite{BatalinVilkovisky1981,BatalinVilkovisky1983}, in their study
of generalized gauge symmetries in quantum field theory, proposed to
associate to a function $S_0$, called the classical action, on the
space of fields $X$, taken here to be an affine variety, a solution
$S$ of the classical master equation $[S,S]=0$ in a completed space of
functions on the $(-1)$-shifted cotangent bundle $T^*[-1]V$ of a
certain graded manifold $V$ containing $X$. The Poisson bracket
$[\;\,,\;]$ has degree 1 and the classical master equation implies
that $[S,\;\,]$ is a differential on functions on $T^*[-1]V$.  The
corresponding cohomology is called (classical) BRST cohomology and
comes with an induced product of degree $0$ and Poisson bracket of
degree $1$. The BRST cohomology in degree $0$ consists of regular
functions on the critical locus of $S_0$ that are annihilated by the
vector fields that annihilate $S_0$.  The ultimate aim is to study (or
make sense of) the asymptotic expansion of oscillatory integrals $\int
\exp( i S_0(x)/\hbar)f(x)dx$ as $\hbar\to 0$ in cases where the
critical points of $S_0$ are not isolated, particularly in the
infinite dimensional case. The classical master equation, considered
here, appears in the study of the critical locus, which is preliminary
to the study of the oscillatory integrals where the quantum master
equation arises. We plan to extend our approach to the quantum case in
a future publication.

The starting point of this paper is the remark that the Batalin--Vilkovisky
construction depends on several choices, so that the invariant meaning of the
BRST cohomology remained unclear.  Our aim is to formalize the construction by
introducing the notion of {\em BV variety} associated with a regular function
$S_0$ on a nonsingular affine variety, comprising a solution $S$ of the
classical master equation. We show that, given $S_0$, such a BV variety exists
and is unique up to a natural notion of stable equivalence and that
automorphisms act trivially on cohomology. A consequence is that the BRST
cohomology is uniquely determined, up to unique isomorphism of Poisson
algebras, by $S_0$.

Let us describe the result in more detail. Let $\K$ be a field of
characteristic $0$ and $X$ a nonsingular affine variety. We use the
language of $\mathbb Z$-graded varieties. A $\mathbb Z$-graded variety
with support $X$ is a $\mathbb Z$-graded commutative ringed space
$V=(X,\mathcal O_V)$ with structure sheaf $\mathcal
O_V=\bigoplus_{i\in\mathbb Z}\mathcal O_V^i$ locally isomorphic to the
completed symmetric algebra of a free graded $\mathcal O_X$-module
with homogeneous components of finite rank.  The completion is defined
by the filtration $F^p\mathcal O_V$ of ideals generated by the elements of
degree $\geq p$.  The (-1)-shifted cotangent bundle of a $\mathbb
Z$-graded variety $V$ such that $\mathcal O_V^i=0$ for $i<0$ is
$M=T^*[-1]V=(X,\mathcal O_M)$ with $\mathcal
O_M=\widehat\Sym_{\mathcal O_V}T_{V}[1]$, the completed graded
symmetric algebra of the tangent sheaf of $V$, with degree shift
$T_{V}[1]^i=T_V^{i+1}$.  Then $\mathcal O_M$ is, in the terminology of
\cite{CostelloGwilliam}, a sheaf of Poisson${}_0$ or $P_0$-algebras;
namely, it comes with a graded commutative product of degree zero and
a Poisson bracket of degree 1.\footnote{In general a $P_j$-algebra is
  a graded commutative algebra with a Poisson bracket of degree
  $1-j$.}  The Poisson bracket comes from the canonical symplectic
structure of degree $-1$ on $T^*[-1]V$.  More generally we define a
$(-1)$-symplectic variety with support $X$ to be a graded manifold
$(X,\mathcal O_M)$ such that $\mathcal O_M$ is locally isomorphic as a
$P_0$-algebra to a shifted cotangent bundle.  The classical master
equation for a function $S\in \Gamma(X,\mathcal O_M)$ of degree $0$ on
a $(-1)$-symplectic manifold $M$ is $[S,S]=0$. A solution $S$ defines
a differential $d_S=[S,\;\,]$ on the sheaf of $P_0$-algebras $\mathcal
O_M$. Let $I_M=F^1\mathcal O_M$ be the ideal of $\mathcal O_M$
generated by elements of positive degree. Then $d_S$ preserves $I_M$
and induces a differential on the non-positively graded complex of
sheaves $\mathcal O_M/I_M$.
\begin{definition}
  Let $S_0\in\Gamma(X,\mathcal O_X)$ be a regular function on
  $X\in\mathcal C$.  A {\em BV variety with support} $(X,S_0)$ is a
  pair $(M,S)$ consisting of a $(-1)$-symplectic variety $M$
  with support $X$ and a function
  $S\in\Gamma(X,\mathcal O_{M}^0)$ such that
  \begin{enumerate}
    \item[(i)] $S|_X=S_0$.
    \item[(ii)] $S$ is a solution of the classical master equation
      $[S,S]=0$.
    \item[(iii)] The cohomology sheaf of the complex $(\mathcal
      O_{M}/I_M,d_S)$ vanishes in non-zero degree.
  \end{enumerate}
\end{definition}
The complex of sheaves $(\mathcal O_M,d_S)$ is called BRST complex and
its cohomology is called BRST cohomology (after Becchi, Rouet, Stora
and Tyutin, who introduced it in the case of ordinary gauge theory
\cite{BRS}).  The BRST complex is a sheaf of differential
$P_0$-algebras, namely a sheaf of $P_0$-algebras with a differential
that is a derivation for both the product of and the bracket, so that
the BRST cohomology is a $P_0$-algebra.

Before stating our results we now add a few comments on the origin and
meaning of our axiomatic setting.  We refer to
\cite{HenneauxTeitelboim} for the physical background of this
construction and to \cite{Stasheff} for its mathematical context.  The
origin of the story is in the Faddeev--Popov description
\cite{FaddeevPopov} of path integrals over the quotient of the space
of fields by the action of the gauge group. With the work of Becchi,
Rouet, Stora and Tyutin, see \cite{BRS}, who identified gauge
invariant observables as cocycles in a differentialg graded algebra,
the BRST complex, and of Zinn-Justin \cite{Zinn} who introduced a
version of the master equation, it became clear that $-1$-sympectic
manifolds and the master equation are the structure underlying
perturbative gauge theory and renormalization, see \cite{Costello} for
a mathematical approach to this subject. In our finite-dimensional
classical setting the solutions of the master equation covered by the
approach of Faddeev and Popov arise in the case of a classical action
$S_0$ invariant under the action of a connected Lie group, see Example
\ref{exa-3} in Section \ref{s-5andahalf}.  To these data one
associates a solution \eqref{e-FP} of the master equation, the
Faddeev--Popov action.  This solution obeys (i) and (ii) but in
general not (iii). Condition (iii) is satisfied under additional
conditions of freeness of the action, see Section \ref{s-5andahalf},
Example \ref{exa-3}.  In the physics literature it was noticed that
the Faddeev--Popov construction had to be extended if the gauge group
does not act freely, or in more general situations in which one would
like to quotient by symmetries of the classical actions that are not
described by a group action.  For example, in certain quantum field
theories, such as the Poisson sigma model underlying Kontevich's
deformation quantization {}\cite{Kontsevich,CattaneoFelder}, the
classical action $S_0$ is invariant under (i.e., annihilated by) a
distribution of tangent planes which is integrable only when
restricted to the critical locus of $S_0$.  A finite-dimensinal model
for this phenomenon is given by Example \ref{exa-6} in Section
\ref{s-5andahalf}, where $S_0$ is the square of the norm on a
Euclidean vector bundle with orthogonal connection.  The contribution
of Batalin and Vilkovisky was to introduce a general construction of
solutions of the master equation replacing previous ad hoc attempts to
generalize the Faddeev--Popov solution. Their idea was to start from a
classical action without assuming a priori the existence of a group of
symmetries. Our observation is that the Batalin--Vilkovisky approach
amounts to add axiom (iii) to the wish list for solutions of the
master equation.  From the point of view of this paper, axiom (iii) is
important as it implies existence and uniqueness results: the
existence and uniqueness up to stable equivalence of a solution $S$
obeying (i)--(iii) given $S_0$ and the existence and uniqueness of the
corresponding differential $P_0$-algebra of observables (the BRST
complex) up to a contractible space of isomorphisms, see
Prop.~\ref{p-contractible} in Section \ref{s-7}.

Finally let us remark that there are
other interesting solutions of the master equation, that do not
necessarily obey property (iii) of the definition of BV varieties. They
are effective actions obtained by reduction from solutions in infinite dimensional
spaces of fields of a quantum field theory, 
see {}\cite{Mnev, BonechiMnevZabzine,  BonechiCattaneoMnev, CattaneoMnev}.

We now turn to the description of our results.

The main local result of this paper is that a BV variety $(M,S)$ with
given support $(X,S_0)$ such that $X$ is affine 
exists and is essentially unique. Two BV varieties
$(M_1,S_1)$, $(M_2,S_2)$ with support $(X,S_0)$ are called equivalent
if there is a Poisson isomorphism $M_1\to M_2$ inducing
the identity on $X$ whose pull-back sends $S_2$ is $S_1$.  They are
stably equivalent if they become equivalent after taking the product
with solutions of BV varieties with support
$(X=\{\mathrm{pt}\},S_0=0)$, see Section \ref{ss-products} for the
precise definition. The main property is that stably equivalent
solutions give rise to BRST complexes that are quasi-isomorphic as
sheaves of differential $P_0$-algebras.
\begin{theorem}\label{t-00} Let $S_0$ be a regular function on a
nonsingular affine variety $X$ over a field $\K$ of characteristic zero.
\begin{enumerate}
\item[(i)] 
There exists a BV variety $(M,S)$ with support $(X,S_0)$ such
that $M\cong T^*[-1]V$ for some non-negatively graded variety $V$. 
It is unique up to stable equivalence.
\item[(ii)] Poisson automorphisms of $(M,S)$ act as the identity on the
cohomology of the BRST complex.
Thus the BRST cohomology $\mathcal H^\bullet(\mathcal O_M,d_S)$ is determined by
$(X,S_0)$ up to unique isomorphism. 
\end{enumerate}
\end{theorem}
The existence proof is based on the construction described in
{}\cite{BatalinVilkovisky1981,BatalinVilkovisky1983} and is in two steps. In the
first step one extends the map $dS_0\colon T_X\to \mathcal O_X$ sending a
vector field $\xi$ to $\xi(S_0)$ to a semi-free resolution of the Jacobian
ring, namely a quasi-isomorphism $(R,\delta)\to (J(S_0),0)$ of differential
graded commutative $\mathcal O_X$-algebras, where $R$ is the symmetric algebra
of a negatively graded locally free $\mathcal O_X$-module with homogeneous
components of finite rank.  The existence of such resolutions is due to Tate
\cite{Tate1957} and $R$ is called Tate (or Koszul--Tate) resolution.
Geometrically $\delta$ is a cohomological vector field on the coisotropic
subvariety of a shifted cotangent bundle $M=T^*[-1]V$ determined by the ideal
$I_M$. In the second step one extends this vector field to a Hamiltonian
cohomological vector field $[S,\ ]$ on $T^*[-1]V$. This existence proof is
basically adapted from \cite{HenneauxTeitelboim}, {Chapter 17}, but we avoid
using the ``regularity condition'' on the smoothness of the critical locus
assumed there.

To show uniqueness up to stable equivalence we remark that all BV varieties
with given support are isomorphic to BV varieties obtained from some Tate
resolution and the question reduces to comparing different Tate resolutions. It
is a standard result that different Tate resolutions of the same algebra are
related by a quasi-isomorphism that is unique up to homotopy. We prove in the
Appendix the stronger result that any two such resolutions become isomorphic as
differential graded commutative algebras after taking the tensor product with
the symmetric algebra of an acyclic complex.

The existence part of Theorem \ref{t-00} (i) is proved in Section
\ref{s-existence} (Theorem \ref{t-T1}); the uniqueness up to stable equivalence
is Theorem \ref{t-33} in Section \ref{s-relating}.  Part (ii) is proved in
Section \ref{s-auto} (Theorem \ref{t-automorphisms} and Corollary \ref{c-unique}).

The next result is a partial description of the cohomology of the BRST complex.
The cokernel of the map $dS_0\colon T_X\to \mathcal O_X$ is the {\em Jacobian
ring}, the quotient of $\mathcal O_X$ by the ideal generated by partial
derivatives of $S_0$.  Vector fields in the kernel $L(S_0)$ of $dS_0$ are {\em
infinitesimal symmetries} of $S_0$. They form a sheaf of Lie subalgebras of
$T_X$.
\begin{theorem}\label{t-01}
\ 
\begin{enumerate}
\item[(i)] The cohomology sheaf of the BRST complex is supported on the critical
locus of $S_0$ and vanishes in negative degree.
\item[(ii)] 
The zeroth BRST cohomology algebra is isomorphic to the algebra of invariants
    \[
    J(S_0)^{L(S_0)}=\{f\in J(S_0)\,|\,\xi(f)=0, \forall \xi\in L(S_0)\},
    \]
    of the Jacobian ring for the Lie algebra of infinitesimal symmetries of $S_0$.
\end{enumerate}
\end{theorem}
The computation of the BRST cohomology is based on the spectral sequence
associated with the filtration $F^\bullet\mathcal O_M$ and is presented in
Section \ref{s-5}.  Theorem \ref{t-01} follows from the description of the
$E_2$-term in Theorem \ref{t-co2}, see Corollary \ref{c-neg} and Proposition
\ref{p-35}.

Can one describe the BRST cohomology in terms of the geometry of 
the critical locus? We give a conjectural description of this kind, which
we prove in degree 0 and 1: the BRST cohomology for an affine variety
is isomorphic to the cohomology of a Lie--Rinehart algebra naturally
associated to the critical locus, see Section \ref{s-6}. One encouraging
fact is that the bracket $H^0\otimes H^0\to H^1$ induced by the Poisson bracket
has a very natural geometric description in terms of this Lie--Rinehart
algebra.

Theorem \ref{t-01} refers to affine varieties and it is natural to ask whether
affine BV varieties glue well to build global objects defined on general
nonsingular varieties. We have a partial existence result in this direction.
We show in Corollary \ref{c-quasipro} that if $S_0$ is a (possibly multivalued)
function on a quasi-projective variety, then there is a sheaf of differential
$P_0$-algebras which is locally quasi-isomorphic to the BRST complex of
a BV variety associated to $S_0$. The necessary homotopy gluing technique
is explained in Appendix B by Tomer Schlank.

Apart from the extension of our results to the quantum case, namely the theory
of the quantum master equation and Batalin--Vilkovisky integration, see
{}\cite{BatalinVilkovisky1981, HenneauxTeitelboim,Schwarz,
Khudaverdian,AKSZ,Severa,Albertetal}, it is important to study the higher
dimensional case of local functionals in field theory, 
see {}\cite{BatalinVilkovisky1981,
HenneauxTeitelboim,CostelloGwilliam,Paugam}.  It would also be interesting to
compare our approach to the derived geometry approach of \cite{CostelloGwilliam},
developed in {}\cite{Vezzosi,
Pantevetal}, and consider, as these authors do, the more general situation of an
intersection of Lagrangian submanifolds in a symplectic manifold (the case
studied in this paper is the intersection of the zero section with the graph of
$dS_0$ in the cotangent bundle).

In most of the paper we formulate our results for a nonsingular affine
variety $X$ over a field $\K$ of characteristic zero for consistency of
language, but our results hold also, with the same proofs, for smooth manifolds
(with $\K=\mathbb R$) or complex Stein manifolds (with $\K=\mathbb C$).

Another straightforward generalization to which our results apply
with the same proofs is the case where $S_0$ is
a {\em multivalued function defined modulo constants} (alias a closed
one-form). 
By this we mean a formal indefinite integral $S_0 = \int \lambda$,
where $\lambda$ is a closed 1-form. The point is that it is not
$S_0$ that matters but the differential $[S_0,\;]$, 
which depends on $S_0$ through $dS_0$, see \ref{ss-mBV} for a
more formal treatment.

The paper is organized as follows. In Section \ref{s-2} we introduce a
notion of graded variety suitable for our problem. It is patterned on Manin's
definition of supermanifolds and supervarieties \cite{Manin}.  Shifted
cotangent bundles are also introduced there.  BV varieties
and their cohomology are introduced in Section \ref{s-3}. In Section \ref{s-4} 
we prove the existence and
uniqueness result for BV varieties, and the computation of the BRST cohomology
is contained in Section \ref{s-5}.  We then discuss several examples in
Section \ref{s-5andahalf} and in
Section \ref{s-6} we give a geometric description of the
cohomology in degree $0$ and $1$ and of the induced bracket $H^0\otimes H^0\to
H^1$. We conclude our paper with Section \ref{s-7} where we extend our
existence result to the case of quasi-projective varieties.

\medskip\noindent{\bf Conventions.} We work over a field $\K$ of
characteristic zero. The homogeneous component of degree $i$ of a graded
object $E$ is denoted by $E^i$ and $E[i]$ is $E$ with degrees shifted by $i$:
$E[i]^j=E^{i+j}$. Differentials have degree 1.
To avoid conflicts of notation we denote by $I^{(j)}$ the
$j$-th power $I\cdots I$ of an ideal $I$.   

\medskip\noindent{\bf Acknowledgements.} 
We thank Damien Calaque, Kevin Costello, 
Vladimir Hinich and Dmitry Roytenberg for instructive discussions,
Pavel Etingof for kindly providing 
Example \ref{exa-8} in Section \ref{s-5andahalf},
Bernhard Keller and Jim Stasheff for useful correspondence.
We also thank Tomer Schlank for discussions and explanations on homotopy limits and for
providing Appendix \ref{appB}.
We are grateful to Ricardo Campos, Ran Tessler and Amitai Zernik for comments and 
corrections to the manuscript. 
G.F. thans the Hebrew Univerisity
of Jerusalem and D.K. thanks the Forschungsinstitut f\"ur Mathematik at ETH
for hospitality. 
G.F. was supported in part by the Swiss National Science Foundation (Grant
200020-105450). D. K. was supported in part by an ERC Advanced Grant.

\section{Graded varieties}\label{s-2}
\subsection{Symmetric algebras of graded modules}\label{s21}
Let $V=\oplus_{i\in\mathbb Z}V^i$ be a $\mathbb Z$-graded module over a
commutative unital ring $B$ with free homogeneous components $V^i$ 
of finite rank such that $V^0=0$.
If $a\in V^i$ is homogeneous, we set $\deg a=i$. The symmetric algebra
$\Sym(V)=\Sym_B(V)$ is the quotient of the tensor algebra of $V$ by the
relations $ab=(-1)^{\deg a\deg b} ba$. It is a graded commutative algebra with
grading induced by the grading in $V$.  Let
$F^p\Sym(V)$ be the ideal generated by elements of degree $\geq p$. These ideal
form a descending filtration $\Sym(V)\supset F^1\Sym(V)\supset
F^2\Sym(V)\supset\cdots$.
\begin{definition}
  The {\em completion} $\widehat{\Sym}(V)$ of the graded algebra
  $\Sym(V)$ is the inverse limit of $\Sym(V)/F^p\Sym(V)$ in the
  category of graded modules. Namely,
  $\widehat{\Sym}(V)=\oplus_{i\in\mathbb Z} \widehat{\Sym}(V)^i$ with
  \[
  \widehat{\Sym}(V)^i=\lim_{\leftarrow p}\Sym(V)^i/(F^p\Sym(V)\cap
  \Sym(V)^i).
  \]
\end{definition}
Then $\widehat{\Sym}(V)$ is a graded commutative algebra and comes
with the induced filtration $F^p\widehat{\Sym}(V)$.  Note that the
completion has no effect if $V$ is $\mathbb Z_{\geq 0}$-graded or
$\mathbb Z_{\leq 0}$-graded, namely if $V^i=0$ for all $i<0$ or for
all $i>0$.
\begin{remark} 
  The assumption that $V^0=0$ is not essential. It can be achieved by
  replacing $B$ by $\Sym_B(V^0)$.
\end{remark}
\subsection{Graded manifolds and graded algebraic varieties}
We adapt the constructions and definitions of Manin \cite{Manin}, who
introduced a general notion of $\mathbb Z/2\mathbb Z$-graded spaces
(or superspaces), to the $\mathbb Z$-graded case.
\begin{definition} 
  Let $M_0$ be a topological space. A {\em graded space} with {\em
  support}\footnote{or body; we use the terminology of \cite{Berezin}}
  $M_0$ is a ringed space $M=(M_0,\mathcal O_M)$ where
  $\mathcal O_M$ (the {\em structure sheaf} of $M$) is a sheaf of
  $\mathbb Z$-graded commutative rings on $M_0$ such that the stalk
  $\mathcal O_{M,x}$ at every $x\in M_0$ is a local graded ring (namely it has a unique
  maximal proper graded ideal). Morphisms
  are morphisms of locally ringed spaces: a morphism $M\to N$ is a
  pair $(f,f^*)$ where $f:M_0\to N_0$ is a homeomorphism and
  $f^*\colon \mathcal O_N\to f_*\mathcal O_M$ is a grading preserving
  morphism of sheaves of rings, such that, for all $x\in M_0$, $f^*$
  maps the maximal ideal of $\mathcal O_{N,f(x)}$ to the maximal ideal
  of $\mathcal O_{M,x}$.
\end{definition}
\begin{definition} An {\em open subspace} of a graded space $M=(M_0,\mathcal O_M)$
is a graded space of the form $(U_0,\mathcal O_M|_{U_0})$ for some open
subset $U_0\subset M_0$. A {\em closed subspace} is a graded space of the form
$(N_0,(\mathcal O_M/J)|_{N_0})$ for some sheaf of ideals $J\subset
\mathcal O_M$ such that $\mathcal O_M/J$ has support on $N_0$.
\end{definition}
Both open and closed subspaces come with inclusion morphisms 
to $M$. 
\begin{definition} 
  Let $M$ be a graded variety.  Let $J_M$ be the ideal sheaf of
  $\mathcal O_M$ generated by sections of non-zero degree.  The
  {\em reduced space} is the locally ringed space
  $M_{\mathrm{rd}}=(M_0,\mathcal O_{M}/J_M)$.
\end{definition}
Thus $M_{\mathrm{rd}}$ is a closed subspace of $M$ and it is a
locally ringed space in the classical sense, with strictly commutative
structure sheaf sitting in degree 0. 

An important class of graded spaces is obtained from graded locally free
sheaves. Let $(X,\mathcal O_X)$ be a commutative locally ringed space (no
grading). If $\mathcal E$ is a graded module with homogeneous components
$\mathcal E^i$ of finite rank and $\mathcal
E^0=0$, then $U\mapsto \widehat\Sym_{\mathcal O_X(U)}(\mathcal E(U))$ is a
sheaf of rings on $X$ whose stalks are local rings. Thus $M=(X,\widehat
\Sym_{\mathcal O_X}(\mathcal E))$ is a graded space with
$M_\mathrm{rd}=(X,\mathcal O_X)$.

\begin{definition} 
  Let $\mathcal C$ be a subcategory of the category of locally ringed
  space, such as algebraic varieties, smooth manifolds or complex
  manifolds. A {\em graded $\mathcal C$-variety} with support $M_0\in
  \mathcal C$ is a graded space $M=(M_0,\mathcal O_M)$ such that every
  point $x\in M_0$ has an open neighborhood $U$ such that $(U,\mathcal
  O_M|_U)$ is isomorphic to $(U,\widehat \Sym_{\mathcal O_X}(\mathcal
  E))$ for some free graded $\mathcal O_X$-module $\mathcal E$ with 
  homogeneous components of finite rank and
  $\mathcal E^0=0$. Morphisms are morphisms of locally ringed spaces
  restricting to morphisms in $\mathcal C$ on their supports, namely
  such that there is a commutative diagram
  \[
\begin{array}{rcccl} M &\to & N\\
    \uparrow & & \uparrow\\
    M_0&\to& N_0
  \end{array}
  \]
 with lower arrow in $\mathcal C$.
\end{definition}

Depending on $\mathcal C$, we call graded $\mathcal C$-varieties
graded smooth manifolds, graded algebraic varieties, graded affine
varieties, and so on.
\begin{definition} A graded $\mathcal C$-variety $M$ is called $\mathbb Z_{\geq
0}$-{\em graded} ($\mathbb Z_{\leq0}$-{\em graded}) if $\mathcal O^j_V=0$
for $j<0$ ($j>0$).  
\end{definition}
Examples of graded $\mathcal C$-varieties are
$(X,\widehat\Sym_{\mathcal O_X}(\mathcal E))$ for some locally free
graded $\mathcal O_X$-module $\mathcal E$ with finite rank homogeneous
components such that $\mathcal
E^0=0$. Conversely, if $(X,\mathcal O_M)$ is a graded $\mathcal
C$-variety with support $(X,\mathcal O_X=\mathcal O_M/J_M)\in\mathcal
C$, then $\mathcal E=J_M/J_M^2$ is a locally free $\mathcal O_X$
module and $(X,\mathcal O_M)$ is locally isomorphic to
$(X,\widehat\Sym_{\mathcal O_X}(\mathcal E))$. The obstructions to
patch local isomorphisms to a global isomorphism lie in $H^1$ of a
certain vector bundle on $X$. Thus if $X$ is a smooth manifold or an
affine algebraic variety, then the obstruction vanish and Batchelor's
Theorem holds: every graded variety with support $X$ is isomorphic to
$(X,\widehat\Sym_{\mathcal O_X}(\mathcal E))$ for some locally free
$\mathcal E$ with homogeneous components of finite rank.
\subsection{$P_0$-algebras}
A $P_0$-{\em algebra} over a field $\K$ is a graded commutative algebra $A= 
\oplus_{d\in\mathbb Z}A^d$ over $\K$ with a Poisson bracket $[\;\,,\;]\colon
A\otimes A\to A$ of degree $1$. A {\em differential $P_0$-algebra} is a 
$P_0$-algebra together with a differential of degree $1$ which is a derivation
for both the product and the bracket. For completeness (and to fix sign 
conventions) let us write out the axioms for the bracket $[\;\,,\;]$ and the 
differential $d$.
For any homogeneous elements
$a,b,c$,
\begin{enumerate}
\item[(i)] The bracket is a bilinear map sending $A^r\otimes A^s$ to $A^{r+s+1}$
\item[(ii)] $[a,b]=-(-1)^{(\deg a-1)(\deg b-1)}[b,a]$.
\item[(iii)] $[ab,c]=a[b,c]+(-1)^{\deg a\deg b}b[a,c]$.
\item[(iv)] $d(ab)=(da) b+(-1)^{\deg a}a\,db$. 
\item[(v)] $d[a,b]=[da,b]+(-1)^{\deg a-1}[a,db]$.
\item[(vi)] $(-1)^{(\deg a-1)(\deg c-1)}[[a,b],c]+\text{cyclic permutations} =0.$
\end{enumerate}
If $S\in A^0$ obeys the classical master equation
$[S,S]=0$ then the map $d_S\colon b\to [S,b]$ is a
differential. We call $d_S$ a {\em hamiltonian differential} with {\em 
hamiltonian} $S$.

\subsection{Shifted cotangent bundles and $(-1)$-symplectic varieties}\label{s-scb}
We consider $\mathcal C$-varieties, where $\mathcal C$ is the category
of nonsingular algebraic
varieties over a field $\K$ of characteristic zero or of smooth manifolds (with $\K=\mathbb R$) or 
of complex manifolds (with $\K=\mathbb C$). Let
$V$ be such a graded variety with support $X$ and suppose that 
$\mathcal O_V^i=0$
for all $i<0$ ($V$ is $\mathbb Z_{\geq0}$-graded). To $V$ we associate
its shifted cotangent bundle $M=T^*[-1]V$ by the following construction. 
A (left) derivation
of $\mathcal O_V$ of degree $d$ is a section $\xi$ of the sheaf
$\Pi_{j=0}^\infty\scriptHom(\mathcal O_V^j,\mathcal O_V^{j+d})$ of degree $d$
endomorphisms of $\mathcal O_V$ such that
$\xi(ab)=\xi(a)b+(-1)^{d\deg\,a}a\xi(b)$. Derivations of $\mathcal
O_V$ of degree $d$ form a sheaf $T_V^d$ and $T_V=\oplus_d T^d_V$ is a
sheaf of graded Lie algebras acting on $\mathcal O_V$ by derivation. Then the
bracket extends to a Poisson bracket of degree 1 on 
\[
\tilde{\mathcal O}_M=\Sym_{\mathcal
  O_V}T_V[1],
\]
which thus becomes a sheaf of $P_0$-algebras. 
We will need a completion $\mathcal O_M$ of $\tilde{\mathcal O}_M$.  Let
$
F^p\tilde{\mathcal O}_M$ be the ideal in 
$\Sym_{\mathcal O_V}(T_V[1])$ generated by elements of degree at least $p$.
These sheaves of ideals form a descending filtration of $\tilde{\mathcal O}_M=F^0\tilde{\mathcal O}_M$.
\begin{definition}
  The $(-1)$-{\em shifted cotangent bundle} of $V$ is the graded variety
  $M=T^*[-1]V=(X,\mathcal O_M)$, where
  \[
  \mathcal O_{M}=\lim_{\leftarrow}\tilde{\mathcal O}_M/F^p\tilde{\mathcal O}_M.
  \]
  The inverse limit is taken in the category of $\mathbb Z$-graded
  sheaves, i.e., degree by degree.
\end{definition}
We denote by $F^p\mathcal O_M=\lim_{\leftarrow}F^p\tilde{\mathcal
  O}_M/F^{p+q}\tilde{\mathcal O}_M$ the filtration by ideals
topologically generated by elements of degree $\geq p$ in $\mathcal
O_M$.
\begin{prop} Let $V$ be a $\mathbb Z_{\geq0}$-graded variety and
  $M=T^*[-1]V$.  Then the Poisson bracket on $\tilde{\mathcal O}_M$
  extends to the completion making $\mathcal O_{M}$ a sheaf of
  $P_0$-algebras over $\K$.
\end{prop}
Let $\tilde{\mathcal O}_M=\Sym_{\mathcal O_V}T_V[1]$. The product
on $\tilde{\mathcal O}_M$ is compatible with the filtration 
in the sense that $F^p\tilde{\mathcal O}_M\cdot 
F^q\tilde{\mathcal O}_M\subset F^{p+q}\tilde{\mathcal O}_M$, 
and thus passes to the completion $\mathcal O_M$
but this is not true for the bracket, making the statement not completely trivial. 
However, we have the
following observation, which suffices to show that the bracket is defined
on the completion.
\begin{lemma}\label{l-Maud}
  Let $p\geq0$.
  \begin{enumerate}
  \item[(i)] If $d\geq-1$ then $[\tilde{\mathcal O}_M^d,F^p\tilde{\mathcal O}_M]
    \subset F^p\tilde{\mathcal O}_M$.
  \item[(ii)] If $d<-1$ and 
    $p+d\geq0$ then $[\tilde{\mathcal O}_M^d,F^p\tilde{\mathcal O}_M]
    \subset F^{p+d+1}\tilde{\mathcal O}_M$.
  \end{enumerate}
  Thus the bracket passes to the completion $\mathcal O_M=\oplus_i\lim_\leftarrow 
  \tilde{\mathcal O}^i_M/F^p\tilde{\mathcal O}^i_M$ 
  and (i), (ii) hold for $\tilde{\mathcal O}_M$ replaced by $\mathcal O_M$.
\end{lemma} 
\begin{proof}
Let us adopt the convention that $a_j,b_j,\dots$ denote local sections of $\tilde{\mathcal O}_M$ 
of degree $j$ and $a,b,\dots$ sections of unspecified degree. 
Then, if $d+p+1\geq0$ (which is trivially true in (i)), and $p'\geq p$,
\[
[a_d,b c_{p'}]=[a_d,b]c_{p'}\pm[a_d,c_{p'}]b
\in F^p\tilde{\mathcal O}_M+F^{d+p+1}\tilde{\mathcal O}_M
\subset F^{\mathrm{min}(p,d+p+1)}\tilde{\mathcal O}_M,
\]
since $c_{p'}$ has degree $\geq p$ and $[a_d,c_{p'}]$ has degree $d+p'+1\geq d+p+1$.
\end{proof}
We note that with the same construction we can define $T^*[-1]N$ for a $\mathbb Z_{\leq0}$-graded variety $N$.
\begin{definition}
  A {\em Poisson morphism} $T^*[-1]V \to T^*[-1]W$ is a map of graded
  varieties respecting the Poisson bracket.
\end{definition}
An \'etale morphism $\varphi\colon V\to W$ of graded varieties (namely one
which is \'etale on supports and for which $\varphi^*$ is locally invertible)
induces a morphism $T_V\to T_W$ of sheaves of graded Lie algebras, defined by
$\theta\mapsto (\varphi^*)^{-1}\circ \theta\circ\varphi^*$, and thus induces a
Poisson isomorphism $\Phi\colon T^*[-1]V\to T^*[-1]W$, called the {\em
symplectic lift} of $\varphi$. 

\begin{definition}
  A {\em $(-1)$-symplectic variety} is a graded variety $M=(X,\mathcal
  O_M)$ locally Poisson isomorphic to a $(-1)$-shifted cotangent bundle:
  every point $x\in X$ has an open neighborhood $U$ so that $\mathcal
  O_M|_U$ is Poisson isomorphic to $\mathcal O_{T^*[-1]V}$ for some
  non-negatively graded variety $V=(U,\mathcal O_V)$.
\end{definition}
The structure sheaf $\mathcal O_M$ of a $(-1)$-symplectic variety comes
with a filtration $\mathcal O_M=F^0\mathcal O_M\supset\cdots\supset 
F^p\mathcal O_M\supset F^{p+1}\mathcal O_M\supset\cdots$. This filtration can be used
to uniquely reconstruct $V$ from $M$ up to isomorphism:
\begin{prop} \label{p-PoI}
  If $M=(X,\mathcal O_M)$ is Poisson isomorphic to $T^*[-1]V$ for a
  $\mathbb Z_{\geq 0}$-graded variety $V$ then $V$ is isomorphic to
  $(X,\Sym_{\mathcal O_X}\mathcal E)$, where the graded $\mathcal
  O_X$-module $\mathcal E$ has homogeneous components
  \[
  \mathcal E^p=\mathcal O_M^p/F^{p+1}\mathcal O^p_M+I_M\cdot
  I_M\cap\mathcal O^p_M,\quad I_M=F^1\mathcal O_M, \quad p\geq 1.
  \]
\end{prop}
\begin{proof}
  Suppose $\mathcal O_V=(X,\Sym_{\mathcal O_X}\tilde{\mathcal
    E})$. Then we have a monomorphism $\tilde{\mathcal E}^p\to
  \mathcal O_M^p$ given by the composition
  \[
  \tilde{\mathcal E}\hookrightarrow\mathcal O_V\hookrightarrow\mathcal
  O_{M}\to \mathcal E
  \]
  But $\mathcal O_M^p$ is spanned over $\mathcal O_X$ by the image of
  $\mathcal E^p$ and products of sections of non-zero degree.
  Products with at least two factors of positive degree are in
  $I_M\cdot I_M$ and products with at least a factor of negative
  degree have a factor of degree $\geq p$ and lie therefore in
  $F^{p+1}\mathcal O_M$.
\end{proof}

\subsection{Local description}\label{s-ld}
Let $X$ be an $n$-dimensional nonsingular algebraic variety over a
field $\K$ of characteristic zero. Then every point $p\in X$ has an
affine open neighborhood $U$ with an \'etale map $U\to \mathbb A^n$ to
the affine $n$-space. Thus there are functions
$x^1,\dots,x^n\in\mathcal O_X(U)$ generating the maximal ideal at $p$
and commuting vector fields $\partial_1,\dots,\partial_n\in T_X(U)$
with $\partial_ix^j=\delta_{ij}$. Similarly if $V$ is a graded variety
with support $X$, then $U$ can be chosen so that $\mathcal O_V|_U\cong
\widehat\Sym_{\mathcal O_X}\mathcal E(U)$ where $\mathcal E$ is a free
$\mathcal O_X$-module with homogeneous components $\mathcal E^i$ of
finite rank and $\mathcal E^0=0$.  Let us assume that $\mathcal
O_V^i=0$ for $i<0$ (the case $i>0$ is treated similarly). Then there
are sections $\beta^1,\beta^2,\dots$ of $\mathcal E(U)$, such that,
for all $i>0$, those of degree $i$ are a basis of the free $\mathcal
O_X(U)$-module $\mathcal E^i(U)$.  Let $\beta^*_j$ be the dual basis
of $\mathcal E^*[1]$, so that
\[
\deg\beta_j^*=-\deg\beta^j-1\leq-2.  
\]
Then $\mathcal O_{T^*[-1]V}(U)\cong \widehat\Sym(T_X[1]\oplus\mathcal
E\oplus\mathcal E^*[1])(U)$. 
Its homogeneous component of degree $j$ consists
of formal power series with
coefficients in $\mathcal O_X(U)$ whose terms are monomials of degree $j$ in
generators $x_i^*=-\partial_i\in T_X[1]$, $i=1,\dots,n$ of degree
-1 and $\beta^j,\beta_j^*$, $j=1,2,\dots$. The Poisson bracket 
is
\[
[f,x^*_j]=\partial_jf,\quad f\in\mathcal O_X,\quad [\beta^i,\beta_j^*]=\delta_{ij},
\]
and vanishes on other pairs of generators.

\subsection{Associated graded}
It will be useful to have a description of the associated graded of the
structure sheaf $\mathcal O_M$ of the cotangent bundle $M=T^*[-1]V$ of a
$\mathbb Z_{\geq0}$-graded variety $V$. Let as above $F^p\mathcal O_M$ be the
ideal generated by sections of degree $\geq p$ and $\mathrm{gr}\,\mathcal
O_M=\oplus_{p\geq0}F^p\mathcal O_M/F^{p+1}\mathcal O_M$. Then $R_M=\mathcal
O_M/F^1\mathcal O_M$ is an algebra over $\mathcal O_X=\mathcal O_V/F^1\mathcal
O_V$ and for each $p$, $\mathrm{gr}^p\mathcal O_M$ is naturally an
$R_M$-module. Also there is a natural $\mathcal O_X$-linear map
$\iota\colon\mathcal O_V\to\mathrm{gr}\,\mathcal O_M$ induced from the
inclusion $\mathcal O_V^p\subset F^p\mathcal O_M$.
\begin{lemma}\label{l-assgr} 
The  composition of $\iota$ with the structure map of the $R_M$-module
\[
R_M\otimes \mathcal O_V\to R_M\otimes \mathrm{gr}\,\mathcal O_M\to\mathrm{gr}\,\mathcal O_M
\]
factors through $R_M\otimes_{\mathcal O_X}\mathcal O_V$ and induces
an isomorphism of graded algebras
\[
R_M\otimes_{\mathcal O_X}\mathcal O_V \cong\mathrm{gr}\,\mathcal O_M.
\] 
\end{lemma}
\begin{proof}
Both the map $\iota$ and the action of $R_M$ are $\mathcal O_X$-linear so the map
factors through the tensor product over $\mathcal O_X$. In the local description
of the preceding section $\mathcal O_V=\Sym_{\mathcal O_X}{\mathcal E}$ for some
locally free sheaf $\mathcal E$. On the other hand, 
$F^{p}\mathcal O_M/F^{p+1}\mathcal O_M$ is a free $R_M$-module spanned by monomials
in $\mathcal O_V$ of total degree $p$. Thus $\mathrm{gr}\,\mathcal O_M$ is a free module
over $R_M$ generated by $\mathcal O_V$.
\end{proof}

\subsection{Poisson center and Poisson derivations}
Let $M=T^*[-1]V$ be the shifted cotangent bundle
of a $\mathbb Z_{\geq0}$-graded variety $V$ with support $X$. The Poisson center $Z_M$ is 
the subalgebra of $\mathcal O_M$ of sections $z$ such that $[z,g]=0$ for
all sections $g\in\mathcal O_M$. Let $\K_X$ be the locally constant sheaf
with fiber $\K$. 

\begin{prop}\label{p-center}
$Z_M=\K_X$
\end{prop}
\begin{proof}
Using the local description we see that every point of $X$
has an open neighborhood $U$
such that $\mathcal O_M(U)$ is a completion of a free $\mathcal O_X(U)$-%
module generated by monomials in $x_i^*,\beta^j,\beta_j^*$. The condition
that the bracket of $z\in Z_M(U)$ with $\beta_i$ vanishes implies that $\beta_i^*$ cannot
appear in the monomials contributing to $z$ with nontrivial coefficients. Similarly $\beta^j$ cannot appear
and the vanishing of bracket with $\mathcal O_X(U)$ and with $x_i^*$ implies
that $z$ is locally constant. 
\end{proof}

A {\em Poisson derivation} of degree zero of $\mathcal O_M$ is a vector 
field $D\in\Gamma(X,T_M^0)$ of degree zero such that $D[a,b]=
[D(a),b]+[a,D(b)]$ for all local sections $a,b\in\mathcal O_M$. 
If $h\in\mathcal O_M(X)^{-1}$ is a global section of degree $-1$
then $a\mapsto[h,a]$ is a Poisson derivation by the Jacobi identity.
Derivations of this form are called hamiltonian and $h$ is called a
hamiltonian of the derivation.

\begin{prop}\label{p-ham}
All Poisson derivations of degree zero of $\mathcal O_M$ 
are hamiltonian with unique hamiltonian. 
\end{prop}

\begin{proof}
The basic case is the one dimensional case: let $B$ be a graded commutative
ring and $A=B[z,z^*]=B[z]+B[z]z^*$ where $z$ is an even variable of degree
$j\in 2\mathbb Z$ and $z^*$ has degree $-j-1$. Let $[\;\,,\;]$ be a
$P{}_0$-bracket such that $[z,z^*]=1$ and $[z,B]=0=[z^*,B]$. Then
the action of a Poisson derivation on generators has the form 
$D(z)=f_0+z^*f_1$, $D(z^*)=g_0+z^*g_1$ with $f_i,g_i\in
B[z]$. The conditions for $D(z)$ and $D(z^*)$ to define a derivation are
\[
[D(z),z]+[z,D(z)]=2 [D(z),z]=0,\quad [D(z),z^*]+[z,D(z^*)]=0.
\]
They imply that $f_1=0$ and $g_1=-[f_0,z^*]$ and it follows that 
\begin{equation}\label{e-ham}
h=G-z^*f_0,
\end{equation}
where $G=\int g_0\,dz\in B[z]$ with $\int bz^jdz=bz^{j+1}/(j+1)$, defines a
hamiltonian derivation with the same action as $D$ on $z$ and $z^*$. Thus after
subtracting a hamiltonian derivation we get a derivation vanishing on $z$ and
$z^*$. Moreover, a Poisson derivations obeys $[D(b),z]=0=[D(b),z^*]$, for all
$b\in B$. It follows that $D(B)\subset B$.

Thus adding to any Poisson derivation of $B[z,z^*]$ a suitable hamiltonian
derivation we obtain the $\K[z,z^*]$-linear extension of a derivation of $B$.

Now let $D$ be a Poisson derivation of $\mathcal O_M$ and use the
local description to study the action of $D$ on $\mathcal O_M(U)$ for
some open neighborhood $U$ of a point of $X$. By the repeating the
above reasoning for each pair $\beta^j,\beta_j^*$ of variables, we may
subtract a hamiltonian derivation to get a derivation vanishing on
$\beta^j,\beta_j^*$.  Note that this works even if $V$ is infinite
dimensional as the hamiltonian for each pair lies in $F^p\mathcal O_M$
with $p$ increasing as the degree of $\beta^j$ increases (in
Eq.~\eqref{e-ham} $h$ is of degree $-1$ and $G$ is divisible by $z$ of
degree $j$ so both terms are a product of an element of degree $j$ and
one of degree $-j-1$).  We are left with a derivation of $\mathcal
O_{T^*[-1]X}(U)=\Sym_{\mathcal O_X}T_X[1](U)$ extended to $\mathcal
O_M(U)$ by linearity over $\K[\beta^j,\beta^*_j]$.  The restriction to
$\mathcal O_X(U)$ is a vector field $\xi$ on $U$. By subtracting a
hamiltonian derivation with hamiltonian $\xi\in T_X[1]$ we may assume
that $D$ vanishes on $\mathcal O_X$.  We claim that $D$ also vanishes
on $T_X[1]$ and is thus zero. Indeed $D$ maps $T_X[1]$ to itself and
being a derivation obeys $[D(\xi),f]+[\xi, D(f)]=0$ for all $\xi\in
T_X[1], f\in\mathcal O_X$; since $D(f)=0$ and $[D(\xi),f]$ is the
action of the vector field $D(\xi)$ on the function $f$, we see that
$D(\xi)=0$.  Then $D$ vanishes on generators. Since it preserves the
filtration, it is well-defined and vanishes on each $\mathcal
O_M/F^p\mathcal O_M$, and thus vanishes on the inverse limit $\mathcal
O_M$.

Now let $D\in\Gamma(X,T_V)$  be a global derivation of degree zero. Then we have
an open cover $(U_i)$ of $X$ such that $D|_{U_i}=[h_i,\;\,]$ for some hamiltonian
$h_i\in\mathcal O_M(U_i)$. On intersections $U_i\cap U_j$, $h_i-h_j$ is a
Poisson central element of degree ${-1}$ and thus vanishes by
Prop.~\ref{p-center}. Hence the hamiltonians $h_i$ agree on intersections and
are restrictions of a globally defined hamiltonian $h$ of degree $-1$ with
$D=[h,\;\,]$. Since the Poisson center is trivial in degree $-1$, $h$ is unique.
\end{proof}
\subsection{Duality}\label{ss-Duality}
The following is an extension of a result of Roytenberg \cite{Roytenberg},
who considered the case of smooth graded manifolds.
\begin{prop} 
  Let $V=(X,\Sym_{\mathcal O_X}\mathcal E)$ for some positively graded
  locally free $\mathcal O_X$-module $\mathcal E$ with homogeneous components
of finite rank on a nonsingular 
  algebraic variety. Let
  $V^\vee=(X,\Sym_{\mathcal O_X}\mathcal E^*[1])$. Then $T^*[-1]V$ is
  Poisson isomorphic to $T^*[-1]V^\vee$.
\end{prop}
\begin{proof}
  Let us first assume that $X$ is affine. Then $\mathcal E$ admits an
  algebraic connection $\nabla\colon T_X\to \scriptEnd(\mathcal E)$
  (the obstruction to the existence of connection lies in
  $H^1(\Omega^1\otimes\scriptEnd(\mathcal E))$ and thus vanishes for
  affine varieties).  Such a connection extends to a map of $\mathcal
  O_X$-modules $T_X\to \mathrm{Der}(\Sym_{\mathcal O_X}\mathcal
  E)$. Also the pairing between $\mathcal E$ and $\mathcal E^*$
  defines an inner multiplication $\iota\colon \mathcal E^*\to
  \mathrm{Der}(\Sym_{\mathcal O_X}\mathcal E)$.  Then the sheaf of
  derivations of $\Sym_{\mathcal O_X}\mathcal E$ is isomorphic to
  $T_X\oplus \mathcal E^*$, where $T_X$ acts via $\nabla$ and
  $\mathcal E^*$ via $\iota$. Thus
  \begin{equation}
  \label{e-002}
    \mathcal O_{T^*[-1]V}\cong 
    \widehat\Sym_{\mathcal O_X}
    (T_X[1] \oplus \mathcal E \oplus \mathcal E^*[1]).
  \end{equation}
  By using the dual connection on $\mathcal E^*$, the same result is
  obtained for $V^\vee$ as $\mathcal E$ and $\mathcal E^*[1]$ are
  interchanged. It remains to show that the resulting isomorphism
  respects the bracket and that it is independent of the choice of
  connection, implying that the isomorphisms on affine subsets glue to
  a global isomorphism. For this we work locally on an open subset $U$
  and use the local description of the previous section. Then
  $\mathcal O_V(U)=\mathcal O_X(U)[\beta^j, j\in I]$, $\mathcal
  O_{V^\vee}(U)=\mathcal O_X(U)[\beta_j^*,j\in I]$ and the right-hand
  side of \eqref{e-002} is $A(U)=\mathcal
  O_X(U)[[x_i^*,\beta^j,\beta_j^*,i=1,\dots,n,j\in I]]$.  The degree
  $d$ component of $A(U)$ consists of formal power series with
  coefficients in $\mathcal O_X(U)$ such that all terms have degree
  $d$. Let $\nabla \beta^i=\sum_j a^i_j\beta^j$ for some
  $a^i_j\in\Omega^1(X)$.  Then the isomorphism $\phi\colon A(U)\to
  \mathcal O_{T^*[1]V}(U)$ is
  \[
    x_i^*\mapsto
    -\partial_i-\sum_{j,k}\beta^ja_j^k(\partial_i)
    \frac\partial{\partial{\beta^k}},\qquad
    \beta^j\mapsto \beta^j,\qquad \beta_j^*\mapsto
    -\frac\partial{\partial \beta^j},
  \]
  and the isomorphism $\phi^\vee\colon A(U)\to \mathcal
  O_{T^*[1]V^\vee}(U)$ is
  \[
    x_i^*\mapsto
    -\partial_i+\sum_{k,j}\beta^*_ka^k_j(\partial_i)\frac\partial{\partial
    \beta^*_j},\qquad \beta^j\mapsto \frac{\partial}{\partial
    \beta^*_j},\qquad \beta_j^*\mapsto\beta_j^*.
  \]
  The composition $\phi^\vee\circ \phi^{-1}$ sends $\partial_i$ to
  $\partial_i$, $\beta^j$ to $\partial/\partial \beta^*_j$ and
  $\partial/\partial\beta^j$ to $-\beta_j^*$, therefore it is
  independent of the choice of connection and it is easy to check that
  it is an isomorphism of Poisson algebras.
\end{proof}
\subsection{The group of gauge equivalences}\label{s-gauge}
Let $M=T^*[-1]V$ then 
$\mathcal O_M^{-1}$, the homogeneous component of $\mathcal O_M$ 
of degree $-1$, is a sheaf of Lie algebras
acting on the sheaf $\mathcal O_M$ by derivations (for both the
product and the bracket) of degree $0$. Thus the action preserves the
filtration $F^\bullet\mathcal O_M$.  Let $I_M=F^1\mathcal O_M$ be the
graded ideal generated by elements of positive degree and $I_M^{(j)}$
be the $j$-th power of $I_M$: $I_M^{(0)}=\mathcal O_M$;
$I_M^{(j+1)}=I_M\cdot I_M^{(j)}$. Clearly $I_M^{(p)}\subset
F^p\mathcal O_M$.
\begin{lemma}\label{l-pro}
  \
\begin{enumerate}
  \item[(i)] $[I_M^{(2)}, \mathcal O_M]\subset I_M$.
  \item[(ii)] If $d\geq -1$, and $j\geq0$, $[I_M^{(j)},\mathcal O_M^d]\subset
    I_M^{(j)}$
  \item[(iii)] Let $p\geq 1, d\in\mathbb Z$. Then $[\mathcal
    O_M^{-1}\cap I_M^{(2)},F^p\mathcal O^d_M]\subset I_M^{(2)}\cap
    F^p\mathcal O^d_M$.
  \end{enumerate}
\end{lemma}
\begin{proof}
  (i) This follows from the Leibniz rule: $[I_M^{(2)},\mathcal
  O_M]\subset I_M[I_M,\mathcal O_M]\subset I_M$. The same argument
  proves (ii) by induction, by taking into account that $[I_M,\mathcal O_M^d]\subset
  I_M$ for $d\geq -1$, as the bracket of an element of positive degree
  with one of degree $-1$ has positive degree.  (iii) We have
  $[\mathcal O_M^{-1},F^p\mathcal O^d_M]\subset F^p\mathcal O^d_M$ for
  degree reasons as the bracket has degree 1. Similarly, $I_M$ is
  closed under the Poisson bracket and therefore
  $[I_M^{(2)},I_M]\subset I_M[I_M,I_M]\subset I_M^{(2)}$. Since
  $F^p\mathcal O_M\subset I_M$ for $p\geq1$ the claim follows.
\end{proof}
\begin{cor}\label{c-pro}
  Let $j\geq0$.  Then $[I_M^{(2)}\cap \mathcal
  O_M^{-1},I_{M}^{(j)}]\subset I_M^{(j+1)}$. In particular $\mathcal
  O_M^{-1}$ is a pronilpotent Lie algebra acting nilpotently on 
  $\mathcal O_M/F^p\mathcal O_M$.
\end{cor}
\begin{proof}
  By Lemma \ref{l-pro} (iii) for $p=1$, we have
  $[I_M^{(2)}\cap\mathcal O_M^{-1},I_M]\subset I_M^{(2)}$ and the
  claim follows with the Leibniz rule by induction on $j$.
\end{proof}
Thus the adjoint action of the Lie algebra $\mathfrak g_M=\mathcal
O_M^{-1}\cap I^{(2)}_M$ exponentiates to a sheaf of groups
$G_M=\exp(\mathrm{ad}\,\mathfrak g_M)$.
\begin{definition}\label{d-gaugegroup} Let $\mathfrak g(M)=\Gamma(X,\mathfrak g_M)$. 
  The group of Poisson automorphisms $G(M)=\exp(\mathrm{ad}\,\mathfrak g(M))$ 
  is called group of {\em gauge equivalences}.
\end{definition}
\section{Solutions of the classical master equation}\label{s-3} 
Let $\mathcal C$ be as in Section \ref{s-scb}. We formulate most of
the statements for $\mathcal C$ consisting of smooth algebraic varieties,
but they are valid with slight change of vocabulary to the other cases. 
\subsection{The classical master equation} 
Let $M$ be a $(-1)$-symplectic variety with
support $X\in\mathcal C$. The classical master equation is the
equation $[S,S]=0$ for a function $S\in\Gamma(X,\mathcal O_M^0)$ of
degree $0$ on $M$.  If $S$ is a solution of the master
equation then the operator $d_S=[S,\; ]$ 
is a differential on the
sheaf of $P_0$-algebras $\mathcal O_{M}$. Moreover, being a derivation
of degree 1, it preserves $I_M=F^1\mathcal O_{M}$ and thus defines a
differential on the sheaf of $\mathbb Z_{\leq0}$-graded algebras
$\mathcal O_M/I_M$.
\subsection{BV varieties}
\begin{definition}\label{def-BV} 
  Let $S_0$ be a regular function on $X\in\mathcal C$.
  A {\em BV variety} with support $(X,S_0)$ is a pair $(M,S)$
  consisting of a $(-1)$-symplectic variety $M$ with support $X$
  and a
  function $S\in\Gamma(X,\mathcal O_{M}^0)$ such that
  \begin{enumerate}
    \item[(i)] $S|_X=S_0$.
    \item[(ii)] $S$ is a solution of the classical master equation
      $[S,S]=0$.
    \item[(iii)] The cohomology sheaf of the complex $(\mathcal
      O_{M}/I_M,d_S)$ vanishes in non-zero degree.
  \end{enumerate}
\end{definition}
\begin{remark}
The inclusion $X$ in $M$ is described by the map $\mathcal O^0_M\to \mathcal O_X$ which has kernel $I_M^0=I_M\cap \mathcal O_M^0$. Thus we can write (i) as $S_0\equiv S\mod I_M^0$.
\end{remark}

\subsection{Multivalued BV varieties}\label{ss-mBV}
As the bracket of two functions depends only on their differential it is natural to consider a slight generalization
where we allow $S$ to be multivalued. Let $\tilde O_X$ be the $\mathcal O_X$ 
module $\Omega^{\mathrm{cl}}_X$ of closed differentials. We think of it as 
the space of ``regular multivalued functions modulo constants'', namely as
formal expressions $S_0=\int\lambda$ where $\lambda\in \Omega^{\mathrm{cl}}_X$.
\begin{definition}\label{def-mBV} 
  Let $S_0\in\Gamma(X,\tilde O_X)$ be a multivalued
  function on $X\in\mathcal C$.
  A {\em BV variety} with support $(X,S_0)$ is a pair $(M,S_0+S')$
  consisting of a $\mathbb Z_{\geq 0}$-graded variety $V$ with support $X$
  and a
  function $S'\in\Gamma(X,\mathcal O_{M}^0)$ such that
  \begin{enumerate}
    \item[(i)] $S'|_X=0$.
    \item[(ii)] $S=S_0+S'$ is a solution of the classical master equation
      $[S,S]=0$.
    \item[(iii)] The cohomology sheaf of the complex $(\mathcal
      O_{M}/I_M,d_S)$ vanishes in non-zero degree.
  \end{enumerate}
\end{definition}
We notice that $d_{S_0}=[S_0,\; ]$  is well defined on $\mathcal O_M$ 
as the bracket involves only the
derivative of $S_0$ and
it is easy to check that it is a differential. Then condition (ii)
means $d_{S_0}S'+\frac12[S',S']=0$.

Our results hold in the more general setting of multivalued BV varieties
with the same proofs, except for notational adjustments.
\subsection{The resolution of the Jacobian ring associated with a BV variety}
\label{ss-Jac}
Let $(M,S)$ be a BV variety with support $(X,S_0)$.
Then $S_0=S|_{X}$ is a regular function on $X$. The complex $R_M=(\mathcal
O_M/F^1\mathcal O_M,d_S)$ looks like 
\[ 
  \cdots \to T_X\to \mathcal O_X, 
\] 
and the last map is $\xi\mapsto [S_0,\xi]=\xi(S_0)$, so the cohomology of $R_M$ is
the Jacobian ring $J(S_0)$, the quotient of $\mathcal O_X$ by the
ideal generated by partial derivatives of $S_0$. Thus $R_M$ is a resolution of
$J(S_0)$.

\subsection{The BRST complex of a BV variety}
\begin{definition} 
  The {\em BRST complex} of a BV variety $(M,S)$ is the sheaf of
  differential $P_0$-algebras $(\mathcal O_{M},d_S)$.
\end{definition}
\begin{prop}\label{p-gaugeaction}
  Let $(M,S)$ be a BV variety.
  The subgroup $G(M,S)$ of gauge equivalences of $M$
  fixing $S$ acts as the identity on the cohomology sheaf of the BRST
  complex.
\end{prop}
\begin{proof} 
  Let $\mathrm{ad}_a$ be the operator $b\mapsto [a,b]$. If
  $\exp(\mathrm{ad}_a) S=S$ and $a\in \mathfrak g(M)$ then
  \[
  \left(\frac{\exp(\mathrm{ad}_a)-\mathrm{id}}{\mathrm{ad}_a}\right)[a,S]=0.
  \]
  The expression in parentheses is a power series
  $\mathrm{id}+\frac12\mathrm{ad}_a+\cdots$ starting with the identity
  and is thus an invertible operator acting on $\mathcal
  O_M(X)$. Thus $[a,S]=0$, i.e., $a$ is a cocycle of degree
  $-1$. Since the cohomology vanishes in this degree, there is a
  $b\in\mathcal O_M(X)^{-2}$ such that $a=[S,b]$. It follows that
  \[
  \mathrm{ad}_a=d_S\circ\mathrm{ad}_b+\mathrm{ad}_b\circ d_S,
  \]
  and thus $\mathrm{ad_a}$ is homotopic to the zero map.
\end{proof}

\subsection{Products, trivial BV varieties and stable equivalence}\label{ss-products}
If $(M',S')$ and $(M'',S'')$ are BV varieties with supports
$(X',S_0')$, $(X'',S_0'')$ then $M'\times M''$ with $S=S'\otimes 1
+1\otimes S''\in\Gamma(X'\times X'',\mathcal O_{M'\times M''})=
\Gamma(X'\times X'',\mathcal O_{M'}\hat\otimes \mathcal O_{M''})$ is
also a BV variety, called the {\em product} of the BV varieties
$(M',S')$, $(M'',S'')$. It is denoted by slight abuse of notation
$(M'\times M'', S'+S'')$.

Let $W=\oplus_{i<0} W^{i}$ be a negatively graded vector space over
$\K$ with finite dimensional homogeneous components $W^i$ and set
$W^*=\oplus_{i>0}(W^*)^i$ with $(W^*)^i=(W^{-i})^*$. Then $W$ may be
considered as a graded variety $W=(\{\mathrm{pt}\},\Sym\,W^*)$
supported at a point.  Its shifted cotangent bundle is
$T^*[-1]W=(\{\mathrm{pt}\},\widehat\Sym(W^*\oplus W[1]))$.  Let
$d_{W^*}$ be a differential on $W^*$ with trivial cohomology and set
$S_W$ be the element of the completion of $W^*\otimes W[1]$ in
$\mathcal O_{T^*[-1]W}=\widehat\Sym(W^*\oplus W[1])$ corresponding to
$d_{W^*}$: for any homogeneous basis $(\beta^j)$ of $W^*$ and dual
basis $(\beta^*_j)$ of $W[1]$,
\[
S_W=\sum_jd_{W^*}(\beta_j^*)\beta^j\in\mathcal O_{T^*[-1]W}.
\]
Then $(T^*[-1]W,S_W)$
is a BV variety with support $(\{\mathrm{pt}\},S_0=0)$. Its BRST
cohomology is $\K$ and the BRST complex is   $\widehat\Sym(W^*\oplus
W[1])$ with differential induced from $d_{W^*}$ on $W^*$ and its
dual on $W$.

\begin{definition} A BV variety of the form $(T^*[-1]W,S_W)$ for a
  negatively graded acyclic complex of vector spaces $W$ is called
  {\em trivial}.
\end{definition}

The K\"unneth formula implies:
\begin{lemma}\label{l-55}
If $(M,S)$ is a BV variety with support $(X,S_0)$ and $(T^*[-1]W,S_W)$ is a trivial 
BV variety then their product is a BV variety with support $(X,S_0)$ and
the canonical map 
\[
\mathcal O_{M}\to\mathcal O_{M\times T^*[-1]W}
\]
is a quasi-isomorphism of sheaves of $P_0$-algebras between the corresponding 
BRST complexes.
\end{lemma}
\begin{definition} 
  Two BV varieties $(M,S)$, $(M',S')$ with support $(X,S_0)$
  are called {\em equivalent} it there is a Poisson isomorphism $\Phi\colon M\to M'$ such that $\Phi^*S'=S$. They are called {\em stably equivalent} if they
  become equivalent after taking products with trivial BV varieties.
\end{definition}
\begin{remark}\label{rTate} 
  If  $R_M=(\mathcal O_M/I_M,\delta)$ is the resolution of the Jacobian ring
  associated with $(M,S)$, see \ref{ss-Jac}, then the resolution associated with
  the product with the trivial BV variety $(W,S_W)$ is $R_M\otimes\Sym(W[1])$ with
  differential $\delta\otimes \mathrm{id}+\mathrm{id}\otimes d_W$, where $d_W$ is
  induced from the differential dual to $d_{W^*}$.
\end{remark}
\section{Existence and uniqueness for affine varieties}\label{s-4} Let
$S_0$ be a regular function on a nonsingular affine algebraic variety
over a field $\K$ of characteristic zero. In this section we prove the
existence and uniqueness up to stable equivalence of a BV variety with
support $(X,S_0)$ and thus prove Theorem \ref{t-00}, (i).
The existence result occupies Sections
\ref{ss-TateResolutions}--\ref{ss-43} and is contained in Theorem \ref{t-T1}. The uniqueness up to stable equivalence is Theorem \ref{t-33} and is discussed in Section \ref{s-relating}. A variant of the argument, discussed in 
Section \ref{ss-4}, shows that adding a square of a linear function to
$S_0$ does not influence the BRST cohomology.
Finally, in \ref{s-auto} we show that
Poisson automorphism act trivially on cohomology thereby proving
Theorem \ref{t-00}, (ii). 

 The existence proof is an adaptation of the
construction proposed in {}\cite{BatalinVilkovisky1981,
  BatalinVilkovisky1983, HenneauxTeitelboim} to our context.  The idea
is to start from a resolution of the Jacobian ring of $S_0$ and show
that there exists a BV variety $(M,S)$ with $M$ a shifted cotangent bundle
$M=T^*[-1]V$ such that $\mathcal O_M/I_M$ with the differential
induced from $d_S=[S,\;\,]$ is the given resolution.

\subsection{Tate resolutions}\label{ss-TateResolutions}
Let $S_0\in\mathcal O_X(X)$.  The Jacobian ring $J(S_0)$ of $S_0$ is the cokernel of the
map $\delta\colon T_X\to \mathcal O_X$ sending $\xi$ to $\xi(S_0)$.  The first
step of the existence proof is the extension of $\delta$ to a {\em Tate
resolution} $R$ of the Jacobian ring, namely a quasi-isomorphism of 
differential graded commutative $\mathcal O_X$-algebras $(R,\delta)\to
(J(S_0),0)$ such that $R=\mathrm{Sym}_{\mathcal O_X}(\mathcal W)$ for
some graded $\mathcal O_X$-module $\mathcal W=\oplus_{j\leq-1}\mathcal W^j$
with locally free $\mathcal O_X$-modules $\mathcal W^j$ such that $\mathcal 
W^{-1}=T_X$. See the Appendix for more properties of Tate resolutions.

By construction, $\mathcal W=T_X[1]\oplus \mathcal E^*[1]$, where $\mathcal E$ is
concentrated in positive degree. Tate resolutions exist by the
classical recursive construction of Tate \cite{Tate1957}: assume that 
$\mathcal W$ is constructed down to degree $-d$ in such a way that the negative degree 
cohomology of $\Sym_{\mathcal O_X}(\mathcal W)$ is
zero down to degree $-d$. Then the cohomology of degree $-d$ is a finitely
generated $\mathcal O_X$-module so that it can be killed
by adding to $\mathcal W$ a direct summand of degree $-d-1$.

Given a Tate resolution,  let $V=(X,\Sym_{\mathcal O_X}(\mathcal E))$
the corresponding graded variety and $M=T^*[-1]V$ its (-1)-shifted cotangent
bundle. We have
\[
\mathcal O_M=\widehat
\Sym_{\mathcal O_X}(T_X[1]\oplus \mathcal E\oplus\mathcal E^*[1]),
\]
and $R$ is identified with the quotient $R_M=\mathcal O_M/F^1\mathcal O_M$ by the
ideal generated by elements of positive degree.

The first approximation to a solution of the master equation is $S_0+\delta|_{\mathcal E^*[1]}$,
where the restriction $\delta|_{\mathcal E^*[1]}\in \mathrm{Hom}(\mathcal E^*[1],(\widehat\Sym(T_X[1]\oplus \mathcal E^*[1]))[1])$ 
of $\delta$ to $\mathcal E^*$
is viewed as an element of 
$\mathcal E\hat\otimes\Sym(T^*[1]\oplus \mathcal E^*[1])\subset\mathcal O_M(X)$. More explicitly, 
\[
S_\mathrm{lin}=S_0+\sum_i\delta\beta_i^*\beta^i,
\] 
for any basis $\beta^i$ of $\mathcal E$ and dual basis $\beta^*_i$ of
$\mathcal E^*$.  Although $S_{\mathrm{lin}}$ does
not obey the master equation, its hamiltonian vector field $[S_\mathrm{lin},\
]$ does define a differential on the associated graded of $\mathcal O_M$.
\begin{prop}\label{p-assgr}
The operator $[S_\mathrm{lin},\;\,]$ preserves the filtration
$F^\bullet\mathcal O_M$ and thus induces a differential on the graded
$\mathcal O_X$-algebra $\mathrm{gr}\,\mathcal O_M$. The canonical
isomorphism $\mathrm{gr}\,\mathcal O_M\cong R_M\otimes_{\mathcal
O_X}\mathcal O_V$ (see Lemma \ref{l-assgr}) identifies the differential with
$\delta\otimes \mathrm{id}$.
\end{prop}
\begin{proof}
The operator preserves the filtration because it is of degree one.
 Let us compute the action of
$[S_\mathrm{lin},\;\,]$ on the associated graded starting with 
$R_M=\mathrm{gr}^0\mathcal O_M$: $R_M$ is a
locally free graded $\mathcal O_X$-module 
whose generators are the classes of
$x_i^*$, a local basis of $T_X[-1]$ and of $\beta_j^*$, a local basis of
$\mathcal E^*[-1]$. The non-zero brackets are $[f,x_i^*]=\partial_if$,
$f\in\mathcal O_X$ and $[\beta^j,\beta_l^*]=\delta_{jl}$. Thus
$[S_\mathrm{lin},x^*_i]\equiv\partial_iS_0 \mod F^1\mathcal O_M$ and
$[S_\mathrm{lin},\beta_j^*]\equiv\delta \beta^*_j\mod F^1\mathcal O_M$. Also
$[S_\mathrm{lin},\mathcal O_X]\equiv 0\mod F^1\mathcal O_M$. Thus
$[S_\mathrm{lin},\;\,]$ coincides with $\delta$ on $R_M=\mathrm{gr}^0\mathcal O_M$.
The claim in the general case follows from the fact that if $a\in\mathcal O_V=
\Sym_{\mathcal O_X}(\mathcal E)$ has positive degree $p$ then
$[S_\mathrm{lin},a]$ has degree $p+1$ and thus vanishes modulo $F^{p+1}\mathcal
O_M$.
\end{proof}
\begin{remark}\label{r-bigraded}
 Note that $\mathrm{gr}\,\mathcal O_M$ is bigraded, with 
\[
\mathrm{gr}^{p,q}\mathcal O_M=F^p\mathcal O_M^q/F^{p+1}\mathcal O_M^q,
\]
and nonzero components for $p\geq \max(q,0)$.
\end{remark}
\subsection{BV varieties associated to a Tate resolution}\label{s-existence}

Let $I_M=F^1\mathcal O_M$ be the ideal generated by elements of positive degrees
and $I^{(j)}_M$ its $j$-th power, cf.~Section \ref{s-gauge}.
\begin{lemma} \label{l-Fitzherbert}
  $[I_M^{(2)}\cap \mathcal O_M^{0},F^p\mathcal O_M]\subset F^{p+1}\mathcal O_M$
\end{lemma}
\begin{proof}
  Let us  again adopt the convention that $a_j,b_j,\dots$ denote arbitrary
  local sections of degree $j$ and $a,b,\dots$ sections of unspecified degree.
  Then an element of $F^p\mathcal O_M$ is a sum of terms of the form $a b_{p'}$
  with $p'\geq p$ and if $c_0\in I_M^{(2)}\cap\mathcal O_M^0$, $[c_0,a
  b_{p'}]=[c_0,a]b_{p'}\pm[c_0,b_{p'}]a$.  The bracket in the second term has
  degree $p'+1$ so  the second term is in $F^{p+1}$.  By Lemma \ref{l-pro}, the
  first term belongs to $[I_M^{(2)},\mathcal O_M]F^{p}\mathcal O_M \subset
  I_MF^p\mathcal O_M \subset F^{p+1}\mathcal O_M$.
\end{proof}
  Thus $I^{(2)}_M\cap\mathcal O_M$ acts trivially on the associated graded. In
  particular any $S\in\mathcal O_M^0$ such that $S\equiv S_{\mathrm{lin}}\mod
  I_M^{(2)}$ will induce the same differential as $S_\mathrm{lin}$ on
  $\mathrm{gr}\,\mathcal O_M$. The idea is to construct a solution of
  $S\in\Gamma(X,\mathcal O_M)$ of the classical master equation 
  such that $S\equiv S_{\mathrm{lin}}\mod I_M^{(2)}$. This is done recursively:
\begin{definition}\label{d-asso} Let $R=(\mathcal O_M/I_M,\delta)$ be a Tate resolution
  of the Jacobian ring $J(S_0)$ with $M=T^*[-1]V$ and $S_{\mathrm{lin}}$ be the
  corresponding hamiltonian function.  We say that a solution
  $S\in\Gamma(X,{\mathcal O}_M)$ of the master equation is {\em associated with $R$} if
  $S\equiv S_{\mathrm{lin}}\mod I_M^{(2)}$.
\end{definition}
\begin{theorem}\label{t-T1} Let $S_0\in{\mathcal O}(X)$.
  Let $R=(\mathcal O_M/I_M,\delta)$
  be a Tate resolution of the Jacobian
  ring $J(S_0)$. Then there exists a solution $S\in\mathcal O_M(X)$ 
  of the classical master equation
  \begin{equation}\label{e-T1}
    [S,S]=0
  \end{equation}
  associated with $R$.
  If $S'$ is another solution with this property then $S'=g\cdot S$
  for some gauge  equivalence $g\in G(M)$.
\end{theorem}
\subsection{Proof of Theorem \ref{t-T1}}\label{ss-43}
\subsubsection*{(a) Filtration and bracket}
  The following Lemma gives a compatibility condition between bracket
  and filtration needed for the recursive construction of the solution
  of the classical master equation.
\begin{lemma}\label{l-ooo}
  Let $p\geq0$.
\begin{enumerate}
  \item[(i)] $[F^p\mathcal O_M^0,\mathcal O_M^0]\subset F^p\mathcal
    O_M^1$.
  \item[(ii)] $[F^p\mathcal O_M^0,F^p\mathcal O_M^0]\subset
    F^{p+1}\mathcal O_M^1$.
  \end{enumerate}
\end{lemma}
\begin{proof} 
  Clearly $[\mathcal O_M^0,\mathcal O_M^0]\subset \mathcal O_M^1$ since the
  bracket has degree 1. Also $F^1\mathcal O_M^1=\mathcal O_M^1$ so if $p=0$
  there is nothing to prove. So let us assume that $p\geq 0$. Let
  $a_i,b_i,\dots $ denote general local sections of $\mathcal O_M$ of degree
  $i$. Then, for $j\geq p$, the bracket
\begin{equation}\label{e-001} [a_{-j}b_j,c_0]=a_{-j}[b_j,c_0]\pm
    b_j[a_{-j},c_0]
  \end{equation}
  lies in $F^p\mathcal O_M$ since $\deg\,[b_j,c_0]=j+1\geq p+1\geq p$ and
  $\deg b_j=j\geq p$. This proves (i). Now let us assume that $c_0\in
  F^p\mathcal O_M^0$. If $j>p$, or if $c_0\in F^{p+1}\mathcal O_M^0$
  (ii) follows from (i), so let $j=p$ and $c_0=d_{-p}e_p$.  Then the
  first term in \eqref{e-001} is in $F^{p+1}\mathcal O_M$ and the
  second term is
  \[
  b_p[a_{-p},c_0]=b_p([a_{-p},d_{-p}]e_p\pm [a_{-p},e_p]d_{-p})\in
  F^{p+1}\mathcal O_M,
  \]
  since $\deg b_p e_p=2p>p$ and $\deg[a_{-p},e_p]=1$.
\end{proof}
  The construction of $S$ and the construction of the gauge equivalence in the
  uniqueness proof both rely on the vanishing of the cohomology of a complex of
  sheaves, that we now introduce.  Let $0\leq q\leq p$. Let 
\[
  \pi_p\colon F^p\mathcal O_M\to \mathrm{gr}^p\mathcal O_M,
\] 
be the canonical projection and consider the subcomplex of $\mathrm{gr}^p\,O_M$ 
\[
\mathcal G^\bullet_{p,q}=\pi_p(F^p\mathcal
O^\bullet_M\cap I_M^{(q)}),
\]
\begin{lemma}\label{l-vanish}
\[
 H^j(\Gamma(X,\mathcal G^\bullet_{p,q}))=0,\quad \text{if $j<p$.}
\]
\end{lemma}
\begin{proof}
The canonical isomorphism $\mathrm{gr}\,\mathcal O_M\cong
(R_M\otimes_{\mathcal O_X}\mathcal O_V,\delta\otimes \mathrm{id})$ of Lemma
\ref{l-assgr} identifies $\mathcal G^\bullet_{p,q}$ with with
$R_M\otimes_{\mathcal O_X} (I_V^{(q)}\cap \mathcal O^p_V)$. The cohomology
sheaf $\mathcal H^j(X,\mathcal G^\bullet_{p,q})$ is zero in degree $j< p$ 
because the cohomology groups of $R_M$ are trivial in
negative degree and $I_V^{(q)}$ is a locally free $\mathcal O_X$-module. 
Since $X$ is affine and $\mathcal G_{p,q}$ is
quasi-coherent the same holds for the complex of global section.
\end{proof}

\subsubsection*{(b) Existence proof}\label{ss-existence} 
We prove by induction that for each $p\geq 1$ there is an $S_{\leq p}\in\Gamma(X,\mathcal
O_M^0)$ such that
\begin{enumerate}
\item[(i)]
$S_{\leq p}\equiv S_\mathrm{lin}\mod I_M^{(2)}$,
\item[(ii)]
$[S_{\leq p},S_{\leq p}]\in I_M^{(2)}\cap F^{p+1}\mathcal O_M$,
\item[(iii)]
$S_{\leq p+1}\equiv S_{\leq {p}}\mod F^{p+1}\mathcal O_M$.
\end{enumerate}
We set $S_{\leq 1}=S_{\mathrm{lin}}$
Then obviously (i) holds for $p=1$ and (iii) does not apply. To prove
(ii) for $p=1$ we use the local description and choose a local basis of sections
$\beta^i$ of generators of $\mathcal O_V$ with dual sections $\beta_i^*$, see \ref{s-ld}. 
Then $S_\mathrm{lin}=S_0+\sum_{i}\delta\beta_i^*\beta^i$. 
Using the fact that $[S_0,S_0]=[S_0,\beta^i]=0$, we obtain
\begin{equation}\label{e-T3a}
  [S_{\leq 1},S_{\leq 1}]= 2\sum_j [S_0,\delta
  \beta_j^*]{}\beta^j+\sum_{i,j}[\delta {}\beta_i^*{}\beta^i,\delta
  {}\beta_j^*{}\beta^j].  
\end{equation} The summand in the second term modulo
$I^{(2)}_M$ can be written as
\begin{eqnarray*}
[\delta {}\beta_i^*{}\beta^i,\delta {}\beta_j^*{}\beta^j]
&\equiv& \delta {}\beta_i^*[{}\beta^i,\delta {}\beta_j^*]{}\beta^j+
\delta {}\beta_j^*[{}\beta^j,\delta {}\beta_i^*]{}\beta^i\mod  I_M^{(2)}\\
 &\equiv& [\delta {}\beta_i^*{}\beta^i,\delta {}\beta_j^*]{}\beta^j+
[\delta {}\beta_j^*{}\beta^j,\delta {}\beta_i^*]{}\beta^i\mod I_M^{(2)}.
\end{eqnarray*}
Summing over $i,j$ and inserting in \eqref{e-T3a} yields
\begin{eqnarray*}
 [S_{\leq 1},S_{\leq 1}]&\equiv&
2\sum_j[S_{\leq 1},\delta\beta^*_j]\beta^j\mod I_M^{(2)}
\\
&\equiv& 2\sum_j\delta^2(\beta^*_j)\beta^j\mod I_M^{(2)}
\\
&\equiv& 0\mod I_M^{(2)}.
\end{eqnarray*}

For the induction step we write
\[
S_{\leq p+1}=S_{\leq p}+v,
\]
with $v\in \Gamma(X,I_M^{(2)}\cap F^{p+1}\mathcal O_M^0)$ to be determined.
Then $S_{\leq p+1}$ obeys (i) and (iii). As for (ii) we notice that by
Proposition \ref{p-assgr}, $[S_{\leq p},v]\equiv [S_\mathrm{lin}, v]\equiv
\delta v\mod F^{p+2}\mathcal O_M$ where $\delta$ is the differential of
$\mathrm{gr}\,\mathcal O_M$, and, by Lemma \ref{l-ooo} (ii), $[v,v]\equiv 0\mod
F^{p+2}\mathcal O_M$. Thus
\[
  [S_{\leq p+1},S_{\leq p+1}]\equiv[S_{\leq p},S_{\leq p}]+2\delta
  v\mod F^{p+2}\mathcal O_M.
\]
On the other hand, by the Jacobi identity, Lemma \ref{l-Fitzherbert}  
and Proposition \ref{p-assgr},
\[
  0=[S_{\leq p},[S_{\leq p},S_{\leq p}]]\equiv
  \delta[S_{\leq p},S_{\leq p}]\mod F^{p+2}{\mathcal O_M}.
\]
Then $[S_{\leq p},S_{\leq p}]\mod F^{p+2}\mathcal O_M \in \Gamma(X,\mathcal
G^1_{p+1,2})$ is a cocycle of degree 1. Since by Lemma \ref{l-vanish} the
cohomology vanishes in degree $p+1\geq2$ there exists a $\bar
v\in\Gamma(X,\mathcal G^0_{p+1,2})$ with $2\delta \bar v+[S_{\leq p},S_{\leq
p}]\equiv 0\mod F^{p+2}\mathcal O_M$. Let $v\in\Gamma(X,F^{p+1}\mathcal
O_M^0\cap I_M^{(2)})$ such that $\pi_p v=\bar v$. Such a lift certainly exists
locally and, because $X$ is affine, also globally.  Then $S_{\leq p+1}=S_{\leq
p}+v$ is a solution of the master equation modulo $F^{p+2}\mathcal O_M$.

It remains to show that $[S_{\leq p+1},S_{\leq p+1}]\in I_M^{(2)}$.  It is
clear that $I_M$ is a Lie subalgebra. Thus $[v,v]\in I_M^{(2)}$ and $[S_{\leq
p},v]\equiv [S_0,v] \mod I_M^{(2)}$ since $S_{\leq p}\equiv S_0\mod I_M$. But
clearly $[S_0,I_M]\subset I_M$ and therefore $[S_0,v]\in I^{(2)}_M$ for $v\in
I^{(2)}_M$.

This completes the induction step.

\subsubsection*{(c) Uniqueness up to gauge equivalence}
Next we prove the transitivity of the action of the group of gauge equivalences
on the space of solutions $S$ of the master equation
associated with a given Tate resolution $(R_M,\delta)$.
Assume that $S,S'$ are two such solutions.
Since both $S$ and $S'$
are congruent to $S_\mathrm{lin}$ modulo $I_M^{(2)}$, we know that
\begin{equation}\label{e-Werner} 
  S-S'\equiv 0\mod I_M^{(2)}\subset F^2\mathcal O_M.  
\end{equation}
The proof is by induction: we show that if $S-S'\in F^p\mathcal O_M(X)$, with
$p\geq2$, we can find a gauge equivalence $g$ such that $g\cdot S-S'\in
\Gamma(X,F^{p+1}\mathcal O_M)$.  This induction step is done by a second induction:
for fixed $p$ let us suppose inductively over $q$ that the difference between
the two solutions is a section of $I_M^{(q)}\cap F^p\mathcal
O_M+F^{p+1}\mathcal O_M$ with $p\geq q\geq2$ and show that we can find a gauge
equivalence $g$ so that 
\[
  g\cdot S-S'\in\Gamma(X,I_M^{(q+1)}\cap F^{p}\mathcal O_M+F^{p+1}\mathcal O_M).  
\]
Since $I_M^{(p+1)}\subset F^{p+1}\mathcal O_M$ this shows by induction that we
may achieve that $S-S'\equiv 0 \mod F^{p+1}\mathcal O_M$ and in fact, because
of \eqref{e-Werner}, in $I_M^{(2)}\cap F^{p+1}\mathcal O_M$, completing the
induction step in $p$.  So let us assume that $v=S-S'$ is a section of
$I_M^{(q)}\cap F^p\mathcal O_M+F^{p+1}\mathcal O_M$ with $p\geq q\geq2$ so that
it defines a section
\[
  \bar v\in\Gamma(X,\mathcal G_{p,q}^0).
\]
Since $S$ and $S'$ both obey the master equation, we see that
$0=[S+S',S-S']=[S+S',v]=2\delta v\mod F^{p+1}\mathcal O_M$ and thus $\bar v$ is
a cocycle. As the cohomology vanishes in degree $0$ (it starts in degree $p\geq
2$), $\bar v$ is exact and there exists a $\bar u\in\Gamma(X,\mathcal
G_{p,q}^{-1})$ such that $\delta\bar u=\bar v$. Let $u\in \Gamma(X,I_M^{(q)}\cap
F^p\mathcal O_M^{-1})$ be a lift of $u$, so that
$v\equiv [S,u]\mod F^{p+1}\mathcal
O_M$. As in the existence proof, it is clear that such a lift exists locally and
we use the fact that $X$ is affine to show that it exists globally on $X$.
Since $[u,S]=-[S,u]$, we have 
\[
v+[u,S]\in  \Gamma(X,F^{p+1}\mathcal O_M).
\]
Let $g=\exp(\mathrm{ad}_u)$. Then
\begin{eqnarray*}
  g\cdot S - S' &=& g\cdot S - S + v\\
  &=& v+[u,S]+\frac12[u,[u,S]]+\cdots\\
  &\equiv& \frac12[u,[u,S]]+\cdots\mod F^{p+1}\mathcal O_M.
\end{eqnarray*}
By Lemma \ref{l-Maud} (i) and Lemma \ref{l-pro} (ii), $[u,S]\in F^{p}\mathcal
O_M^{0}\cap I_M^{(q)}$; also $u\in I_M^{(2)}\cap O_M^{-1}$ so by Lemma
\ref{l-Maud} (i) and the fact that $[I_M^{(2)},I_M^{(q)}]\subset I_M^{(q+1)}$,
we conclude that $[u,[u,S]]$, and by the same argument any of the higher
brackets in the sum, is a section of $I_M^{(q+1)}\cap F^p\mathcal O_M^0$.
Therefore $g\cdot S-S'\in \Gamma(X,I_M^{(q+1)}\cap F^p\mathcal
O_M+F^{p+1}\mathcal O_M)$, as required.

The proof of Theorem \ref{t-T1} is complete.

\subsection{Relating Tate resolutions}\label{s-relating}
Let $S_0$ be a regular function on a nonsingular affince variety 
$X$ and suppose $(M=T^*[-1]V,S), (M'=T^*[-1]V',S')$ are two BV
varieties with support $(X,S_0)$.  Then the quotients
$R_M=\mathcal O_M/I_M$, $R_{M'}=\mathcal O_{M'}/I_{M'}$ by the ideals
generated by positive elements are both resolutions of the Jacobian
sheaf of rings $J(S_0)$.  Let us first consider the case where $R_{M}$
is isomorphic to $R_{M'}$ as a differential graded algebra by an
isomorphism that is the identity in degrees $-1$ and $0$ (so that in
particular $\varphi$ is a morphism of $\mathcal O_X$-modules).
\begin{prop}\label{p-lift}
  Suppose $\varphi\colon R_M\to R_{M'}$ is an isomorphism of sheaves of
  differential graded algebras which is the identity in degree $0$ and $-1$. Then
  $\varphi$ is induced by a Poisson isomorphism 
  $M'\to M$ sending $S$ to $S'$.
\end{prop}
\begin{proof}
  We use the duality, see \ref{ss-Duality}, to represent $R_M$ as $\mathcal
  O_{M^\vee}/I_{M^\vee}$ with $M^\vee=T^*[-1]V^\vee$. By construction $\mathcal
  O_{V^\vee}=\mathcal O_X\oplus\bigoplus_{j\leq -2}\mathcal O_{V^\vee}^j$. It
  follows that derivations of $\mathcal O_{V^\vee}$ of degree $\geq2$ vanish on
  $\mathcal O_X$.  Conversely, any derivation vanishing on $\mathcal O_X$ is in
  the $\mathcal O_V$-submodule generated by derivations of degree $\geq 2$. Such
  derivations correspond to elements of positive degree in $T_{V^\vee}[1]$.
  Therefore $I_{M^\vee}$ is the ideal of
  $
  \mathcal O_{M^\vee}=\widehat\Sym_{\mathcal O_{V^\vee}}T_{V^\vee}[1]
  $
  generated by derivations acting trivially
  on $\mathcal O_X$.  Thus $\mathcal O_{M^\vee}/I_{M^\vee}$  is the symmetric
  algebra over $\mathcal O_{V^\vee}$ of the quotient of $T_{V^\vee}[1]$ by the
  derivations acting trivially on $\mathcal O_X$. The latter quotient is
  canonically $\mathcal O_{V^\vee}\otimes_{\mathcal O_X} T_X[1]$. Thus 
  \[
  \mathcal O_{M^\vee}/I_{M^\vee}\cong 
  \mathcal O_{V^\vee}\otimes_{\mathcal O_X}\Sym_{\mathcal O_X} T_X[1].
  \]
  By assumption, $\varphi$ acts trivially on the second factor and the
  restriction to $\mathcal O_{V^\vee}$ lifts to a Poisson isomorphism $\Phi\colon
  \mathcal O_{M^\vee}\to \mathcal O_{(M')^\vee}$ which coincides with $\varphi$
  on $\mathcal O_{V^\vee}$.  To show that $\Phi$ induces $\varphi$ on $\mathcal
  O_{M^\vee}/I_{M^\vee}$ it remains to show that $\Phi$ induces the identity on
  $T_X[1]=(T_{V^\vee}/\mathrm{Ann}\,\mathcal O_X)[1]$.  But this follows from the
  fact that $\varphi$ is the identity on $\mathcal O_X$: the restriction of the
  image $\varphi^{-1}\circ\xi\circ\varphi$ of a derivation $\xi$ to $\mathcal
  O_X$ is $\xi$. 

  Thus $\Phi(S)$ is a solution of the master equation in $\mathcal O_{M'}(X)$
  associated with the same Tate resolution $R'$ as $S'$. By Theorem
  \ref{t-T1}, $S'=g\circ\Phi (S)$ for some Poisson automorphism $g\in G(M')$.
\end{proof}

Let $R\to J(S_0)$ be a Tate resolution of the Jacobian
ring and let $(W,\delta_W)$ be an acyclic negatively graded complex of
$\K$-vector spaces with finite dimensional homogeneous components. Then the
differential $\delta_{W[1]}$ extends uniquely as a derivation of $\Sym(W[1])$
and $(\Sym(W[1]),\delta_{W[1]})$ is a differential graded algebra with
cohomology $\cong \K$. The tensor product $R\otimes
\Sym(W[1])$ is another resolution of $J(S_0)$.
\begin{prop}\label{p-3} 
  Let $R_j\to J(S_0)$, $j=1,2$ be Tate resolutions extending $\delta\colon T_X\to\mathcal O_X$. 
  Then there exist acyclic
  negatively graded complexes  $W_1,W_2$ of $\K$-vector spaces and an isomorphism
  of differential graded commutative algebras 
  \[ 
  R_1\otimes \Sym(W_1[1])\cong R_2\otimes \Sym(W_2[1]),  
  \] 
  that is the identity in degree $-1$ and $0$.
\end{prop}
We prove this Proposition in a slightly more general context in Appendix \ref{a-2}.

\begin{theorem}\label{t-33}
  Let $X$ be affine. 
  Any two BV varieties $(M_1,S_1),(M_2,S_2)$ 
  with support $(X,S_0)$ such that $M_1,M_2$ are
shifted cotangent bundle are stably equivalent.
\end{theorem}

\begin{proof}
  Let $(T^*[-1]V_1,S_1)$, $(T^*[-1]V_2,S_2)$ be BV varieties 
with the same support $(X,S_0)$ and
corresponding Tate resolutions
  $(R_1,\delta_1)$, $(R_2,\delta_2)$.  By
  Prop.~\ref{p-3} there are acyclic negatively graded complexes $W_1$, $W_2$ and
  an isomorphism of differential graded commutative algebras $R_1\otimes
  \Sym(W_2[1])\cong R_2\otimes \Sym(W_2[1])$, which is the identity in degree
  $0$ and $-1$. 
  Then the product of $(M_i,S_i)$ with the trivial BV varieties $(T^*[-1]W_i,S_{W_i})$,
  see Section~\ref{ss-products}, give stably equivalent BV varieties whose Tate
  resolutions are isomorphic by Remark \ref{rTate}. By Prop.~\ref{p-lift},
  the BV varieties $V_1\times W_1$, $V_2\times W_2$ are equivalent.
\end{proof} 
Together with Prop.~\ref{l-55} this implies:
\begin{cor} If $(T^*[-1]V_1,S_1)$ and $(T^*[-1]V_2,S_2)$ are BV varieties with support $(X,S_0)$ then
  there exists a BV variety $(T^*[-1]V,S)$ with support $(X,S_0)$ and morphisms of sheaves of differential
  $P_0$-algebras 
  \[
  (\mathcal O_{T^*[-1]V_1},d_{S_1})\to 
  (\mathcal O_{T^*[-1]V},d_{S})\leftarrow 
  (\mathcal O_{T^*[-1]V_2},d_{S_2}),
  \]
  between the corresponding BRST complexes, inducing isomorphisms on the cohomology.
\end{cor}
\subsection{Adding a square}\label{ss-4}
\begin{prop}\label{p-2}
  Let $(M,S)$ be a BV variety with support $(X,S_0)$. Let $X'=X\times \mathbb
  A^1$ and $S'_0=S_0+a t^2\in{\mathcal O}(X')$, where $t$ is a coordinate on
  $\mathbb A^1$ and $a\in\K^\times$. 
  Then $(M'=M\times T^*[-1]\mathbb A^1,S'=S+a t^2)$ is a BV variety with support
  $(X',S'_0)$ and the corresponding BRST complexes are quasi-isomorphic
  differential $P_0$-algebras.  
\end{prop}

\begin{proof}
  The argument is similar to the one for trivial solutions: we are taking the
  product with the BV variety $(T^*[-1]\mathbb A^1,a t^2)$. The shifted cotangent bundle
  $T^*[-1]\mathbb A^1$ has coordinates $t,t^*$ and differential such that
  $t\mapsto 0, t^*\mapsto 2at$ which is clearly acyclic. Thus the natural map
  $\mathcal O_{M}\to\mathcal O_{M'}$ is a quasi-isomorphism of
  sheaves of differential $P_0$-algebras by the K\"unneth formula.
\end{proof}
\subsection{Automorphisms of a BV variety}\label{s-auto}
\begin{theorem}\label{t-automorphisms} 
  Let $(M,S)$ be a BV variety and $\phi\colon M\to M$ a
  Poisson automorphism preserving $X$ such that $\phi^*S=S$. Then $\phi^*$
  induces the identity on the cohomology sheaf of the BRST complex.
\end{theorem}
\begin{proof} Since the statement is local we can assume that $M=T^*[-1]V$
is a cotangent bundle.
 Let $F=\phi^*\colon \mathcal O_M\to\mathcal O_M$.
  The automorphism $F$ induces an automorphism $f$ of the sheaf of
  differential graded algebras $\mathcal O_M/I_M$.  Then $f$ and
  $\mathrm{id}$ are both automorphisms of the Tate resolution
  $R_M=\mathcal O_M/I_M$ of the Jacobian ring that are the identity
in degree $0$ and $-1$ and are thus
  related by a homotopy $H$, namely a morphism of differential graded
  algebras $R_M\to R_M[t,dt]$ such that $ev_0\circ
  H=\mathrm{id}$ and $ev_1\circ H=f$, see Lemma \ref{lemma-1} in the
  Appendix. In more detail,
  \[
  H=f_t+dt\, h_t,
  \]
  and $f_t$ is a morphism of differential graded algebras $ R_M\to
  R_M[t]=R_M\otimes_\K\K[t]$ which is the identity in degree $0$ and
  $-1$. Its $\K[t]$-linear extension $R_M[t]\to R_M[t]$ is not
  invertible in general, but since it is at $t=0$ and $t=1$, we can
  invert it at the generic point. More precisely, in each degree i we
  have a rational map $(f_t^{-1})^i\in
  \mathrm{End}(R_M^i)\otimes_\K\K(t)$ regular at $t=0$ and $t=1$, and
  such that $(f_0^{-1})^i=\mathrm{id}$ and $(f^{-1}_1)^i=
  f^{-1}|_{R_M^i}$. This map is the inverse of $f_t|{R_M^i}$ for $t$
  in a Zariski open subset $U_i$ of $\mathbb A^1$ containing $0,1$.
  We next use Prop.~\ref{p-lift} to lift $f_t$ to a family of Poisson
  automorphisms of $\mathcal O_{M}$. Recall that the lift is
  constructed using the duality isomorphism
  $T^*[-1]V=T^*[-1]V^\vee$. Under this identification a morphism of
  Tate resolutions is the same as an automorphism of $\mathcal
  O_{V^\vee}$ and the lift is the canonical symplectic lift of an
  automorphism of the base of a cotangent bundle. The latter is given
  in terms of $f_t$ and $f_t^{-1}$ and thus the lift is defined for
  all $t$ for which $f_t^{-1}$ is defined. Moreover the action of the
  lift on generators of degree bounded by $n$ is defined by the
  restriction of $f_t, f_t^{-1}$ to $R_M^i$ for finitely many $i$
  depending on $n$.  The result is that there is a family $F_t$ of
  automorphisms of $\mathcal O_{M}$ given by a sequence of compatible
  maps
  \[
  \mathcal O_M/F^p\mathcal O_M\to \mathcal O_M/F^p\mathcal O_M,
  \]
  parametrized by $t$ in a Zariski open ($p$-dependent) subset $V_p$
  (an intersection of finitely many $U_i$) of $\mathbb A^1$ containing
  $0$ and $1$. By construction $F_0$ is the identity and we may assume
  (by possibly composing $F_t$ with a gauge equivalence of the form
  $\exp(t a)$, $a\in \mathfrak g(M)$) that $F_1=F$.  Thus
  $S_t:=F_t(S)\in \lim_\leftarrow\mathcal O(M)\otimes\K(t)/F^p
  O(M)\otimes\K(t)$ is a family of solutions of the master equation
  for the Tate resolution $R_M$. It is given by a compatible sequence
  $S_t^{(p)}=S_t\mod F^p\mathcal O_M$ that is defined for $t$ in the
  open set $V_p\subset \mathbb A^1$ containg $0$ and $1$ and such that
  $S_0=S_1=S$.

The next step is to replace $F_t$ by $G_t\circ F_t$ for some gauge
equivalence $G_t$ such that $G_t(S_t)=S$ and such that
$G_0=G_1=\mathrm{id}$.  We need to check that $G_t$ can be chosen this
way and that it is defined when $F_t$ is. The construction of a gauge
equivalence relating two solutions of the master equation associated
with the same Tate resolution is done recursively in the filtration
degrees, see part (c) of the previous section, and the induction step
relies on the vanishing of the cohomology of the complexes $\mathcal
G_{p,q}$ of locally free $\mathcal O_X$-modules of finite rank.  Now
$S_t-S \mod F^p\mathcal O(M)$ vanishes at $t=0$ and $t=1$ and is
defined for $t\in V_p$. In other words $S_t-S\in \mathcal
O(M)/F^p\mathcal O(M) \otimes t(1-t)\K[V_p]$ where $\K[V_p]\subset
\K(t)$ is the space of rational functions that are regular on
$V_p$. Thus the construction of the previous section applies to the
complex $\mathcal G_{p,q}\otimes_\K t(1-t)\K[V_p]$ (the tensor product
with the free and thus flat $\mathcal O_X$-module $\mathcal O_X\otimes_\K t(1-t)\K[V_p]$)
gives recursively a gauge equivalence $G_t$ such that
$G_0=G_1=\mathrm{id}$ and $G_t(S_t)=S$, as required.

We thus have a compatible family of morphisms
$F^{(p)}_t\in\mathrm{End}(\mathcal O_M/F^p\mathcal O_M)\otimes \K[V_p]$ whose value
at every $t\in V_p$ is an automorphism such that $F_t^{(p)}(S)=S\mod F^p\mathcal O(M)$,
so that $F_t^{(p)}$ commutes with $d_S$.
The inverse limit $F_t$ is defined for $t$ in a countable intersection of Zariski
open subsets and is given by a sequence $F_t^{(p)}$ each parametrized by $t$ in
a Zariski open subset of $\mathbb A^1$.

{\it Claim}: $F^{(p)}_t$ acts trivially on the cohomology 
$H^j(\mathcal O_M/F^p\mathcal O_M,d_S)$ for $p$ large enough depending on $j$.
To prove this claim we use the following result.

\begin{lemma}
For any $j$ there exist a $p_0(j)$ such that
$H^j(\mathcal O_M/F^p\mathcal O_M,d_S)\simeq H^j(\mathcal O_M,d_S)$ 
for all $p\geq p_0(j)$.
\end{lemma}

This Lemma is proved in Section \ref{s-5}, see Theorem \ref{t-co2}. Given the
claim, it also implies that $F=F_1$ acts trivially on the cohomology.
To prove the claim, we take the derivative $\dot F_t$ of $F_t$ with respect to $t$.

The endomorphism $F_t^{-1}\circ\dot F_t$ is a Poisson derivation of
degree $0$ of $\mathcal O_{T^*[-1]V}$. By Prop.~\ref{p-ham} all
Poisson derivations of degree zero are uniquely hamiltonian.  Thus
there exists an element $K_t$ of degree $-1$ such that
    \[
      F_t^{-1}\circ \dot F_t= [K_t,\;\,].
    \]
As above, these expression have to be understood as sequences of families
of endomorphisms of $\mathcal O_M/F^p\mathcal O_M$  parametrized by $t\in V_p$.
By the uniqueness of the Hamiltonian, $K_t$ is defined whenever $F_t$ is. 
    As $F_t(S)=S$, we see that $K_t$ is a cocycle: $[S,K_t]=0$. But
    by Theorem \ref{t-01} (i) (proved in Section \ref{s-5})
    the cohomology in degree -1 is trivial, so there exists a family of elements
    $L^{(p)}_t\in\Gamma(X,\mathcal O^j_{M}/F^p\mathcal O^j_M)\otimes \K[V_p]$ for $p$ large
    of degree $-2$ such
    that $K_t=[S,L^{(p)}_t]\mod F^p\mathcal O_M$.  With the Jacobi identity we obtain the
    homotopy formula
    \[
      F_t^{-1}\circ\dot F_t(a)=[S,[L^{(p)}_t,a]]+[L^{(p)}_t,[S,a]], a\in\mathcal O_M^i/F^p\mathcal O_M^i,
    \]
from which it follows that the action on the cohomology is trivial.
\end{proof}
Theorem \ref{t-automorphisms} implies that the BRST cohomology is
canonically associated with $S_0$ in the affine case.
\begin{cor}\label{c-unique} Let $S_0$ be a function on a nonsingular affine
variety $X$ over $\K$. Then the BRST cohomology sheaf $\mathcal
H^\bullet(M,S)$ is determined by $(X,S_0)$ up to unique isomorphism.
\end{cor} 
\begin{proof} By Lemma \ref{l-55} the BRST complexes of
BV varieties differing by taking products with trivial BV varieties are
canonically quasi-isomorphic.  By Theorem \ref{t-33} any two BV varieties with
support $(X,S_0)$ become equivalent after taking such products. The
equivalence induces a quasi-isomorphism of the corresponding BRST-complexes
as differential $P_0$-algebras. By Theorem \ref{t-automorphisms} any two
equivalences differ by an automorphism, so they induce the same map on BRST
cohomology.
\end{proof}

\section{Computing the BRST cohomology}\label{s-5}
The BRST complex of a BV variety $(M,S)$ with support $(X,S_0)$
is a sheaf of differential graded $P_0$-algebras
over $X$. We consider here the cohomology sheaf $\mathcal H(M,S)$ of the BRST complex for $X$ affine and $M=T^*[-1]V$. We use the local description of $\mathcal O_M$ of \ref{s-ld}.
\subsection{The spectral sequence}\label{sub-43}
The main tool for the computation of the BRST cohomology is the spectral sequence of the
filtered complex $(\mathcal O_M,d_S)$. Recall that the sheaf of Jacobian rings $J(S_0)$ is by definition
the cokernel of the map $dS_0\colon T_X\to \mathcal O_X$.
\begin{theorem}\label{t-co2} Let $(M,S)$ be a BV variety with support $(X,S_0)$ such that $X$ is affine. 
\begin{enumerate}
\item[(i)]
There is a fourth quadrant spectral sequence
degenerating at $E_2$ such that
\[
 E^{p,q}_2=\begin{cases} \mathcal H^p(J(S_0)\otimes_{\mathcal O_X}\mathcal O_V,d_1),& \text{ if $q=0$,}
\\
0, & \text{ if $q\neq 0$.}
\end{cases}
\]
The differential $d_1$ is described as follows.
Let $S^{(1)}$  the component in $\Gamma(X,T_V[1])$ of 
$S\in\Gamma(X,\widehat{\Sym}_{\mathcal O_V}T_V[1])$. Then the derivation $S^{(1)}$ vanishes on
the kernel of the canonical projection $\mathcal O_V\to J(S_0)\otimes_{\mathcal O_X}\mathcal O_V$.
The induced differential on the image is $d_1$.
\item[(ii)]
This spectral sequence converges to the cohomology $\mathcal H(M,S)$. More precisely,
the edge homomorphism
\[
\mathcal H^\bullet(M,S)\to E_2^{\bullet,0}
\]
is an isomorphism of graded commutative algebras.
\item[(iii)]
The natural map $\mathcal H^j({\mathcal O_M},d_S)\to \mathcal 
H^j({\mathcal O_M}/F^{p+1}{\mathcal O_M})$ 
is an isomorphism for $j<p$ and a monomorphism for $j=p$.
\end{enumerate}
\end{theorem}

\begin{proof} The first term in the spectral sequence of the
filtered complex is (see, e.g., \cite{CartanEilenberg}, {Chap.~XV, \S 4})
\[
E_0^{p,q}=F^p\mathcal O_M^{p+q}/F^{p+1}\mathcal O_M^{p+q}.
\]
By Remark \ref{r-bigraded}, $E_0$ lives in the fourth
quadrant $p\geq 0, q\leq 0$ and, as a sheaf of $\K$-algebras,
\[
E_0^{p,q}\cong R_M^q\otimes_{\mathcal O_X}\mathcal O_V^p,
\]
where $(R_M=\mathcal O_M/I_M,\delta)$ is the resolution of $J(S_0)$ associated with $S$.
By Prop.~\ref{p-assgr}, the differential is $\delta\otimes \mathrm{id}$.
It follows that
\[
E_1^{p,q}\cong\left\{
\begin{array}{rl}
J(S_0)\otimes_{\mathcal O_X} \mathcal O_V^p,& \text{if $q=0$,}\\
0,&\text{if $q\neq 0$.}
\end{array}\right.
\]
Let us compute the differential $d_1\colon E_1^{p,0}\to E_1^{p+1,0}$ by decomposing
$S$ according to the power in the symmetric algebra:
\[
S=S^{(0)}+S^{(1)}+S^{(2)}+\cdots,\quad \text{with $S^{(j)}\in\Gamma(X,\Sym^j_{\mathcal O_V}T_V[1])$.}
\]
Since the natural map $\mathcal O_V^p\to F^p\mathcal O_M^p/F^{p+1}\mathcal O_M^p=E^{p,0}$
is an isomorphism, we may compute $d_1$ by acting with $d_S$ on 
representatives in $\mathcal O_V^p$.  
We have $S{(0)}=S_0$ and thus  $[S^{(0)},\mathcal
O_V^p]=0$.  Let $I_M^-$ be the ideal of $\mathcal O_M$ generated by elements of
{\em negative} degree. Then, if $j\geq 2$, $S^{(j)}\in I_M^{-}\cdot I_M^{-}$
and thus $[S^{(j)},\mathcal O_V^p]\subset I_M^{-}\cap \mathcal O_M^{p+1}\subset F^{p+2}\mathcal O_M$.
Hence for $a\in\mathcal O_V^p$ and $j\geq 2$ 
the class of $[S^{(j)},a]$ in $E^{p+1,0}=F^{p+1}\mathcal O_M^{p+1}/F^{p+2}\mathcal O_M^{p+1}$ vanishes.
By definition of the Poisson structure on $\mathcal O_M$, the bracket of 
$\mathcal S^{(1)}$ with $\mathcal O_V$ 
is the action of $S^{(1)}$ viewed as a derivation.

Since $E_1^{p,q}=0$ for $q\neq0$, all higher differentials vanish
for degree reasons and the spectral sequence degenerates.

It remains to show that the spectral sequence converges to the cohomology
$\mathcal H(M,S)$.  Recall from \cite{CartanEilenberg}, {p.~324} that a descending filtration
$\cdots \supset F^pA\supset F^{p+1}A\supset\cdots$ of a cochain complex $A$ is
called {\em regular} if for each $m$ there exists a $p_0=p_0(m)$ such that the
cohomology $H^m(F^pA)$ vanishes for $p\geq p_0$. By
\cite{CartanEilenberg}, {Chap.~XV, Prop.~4.1} the spectral sequence of a regular
filtration converges to the cohomology.

\begin{lemma}
Suppose $\cdots\supset F^pA\supset F^{p+1}A\supset\cdots $ is a filtration of a
cochain complex $(A,d)$ such that, for each $j$, the natural map
\begin{equation}\label{e-1}
A^j\to\lim_{\leftarrow}A^j/F^pA^j
\end{equation}
is an isomorphism and assume that $E_1^{p,m-p}=H^m(F^pA/F^{p+1}A)=0$ for $p\geq
p_0(m)$ sufficiently large. Then $H^m(F^pA)=0$ for $p\geq p_0(m)$ and thus the
filtration is regular.
\end{lemma}
\begin{proof}
  Let $p\geq p_0$ and $z\in F^pA^m$ be a cocycle representing a class in
  $H^m(F^pA)$. By the assumption on $E_1$, $z\equiv d y_0 \mod F^{p+1}A$ for some
  $y_0\in F^pA^{m-1}$. Thus $z-dy_0\in F^{p+1}A$ and by the same argument we find
  $y_1\in F^{p+1}A$ such that $z-dy_0-dy_1\in F^{p+2}A$. Iterating we conclude
  that $z\equiv d (y_0+\cdots +y_r)\mod F^{p+r+1}A^{m-1}$, for some $y_j\in
  F^{p+j}A^{m-1}$. The sequence $(y_0+\dots+y_r)_{r\geq 0}$ defines an element of
  $\lim_{\leftarrow}F^pA^{m-1}/F^{p+r} A^{m-1}$. Let $y\in F^pA^{m-1}$ be its
  inverse image by the isomorphism \eqref{e-1}. Then $z=dy$ and it follows that
  $H^m(F^pA)=0$.  
\end{proof}
Since the completed complex ${\mathcal O_M}$ obeys \eqref{e-1} and the
assumption on $E_1$ holds with $p_0(m)=m+1$, we have
\begin{equation}\label{e-4} 
  \mathcal H^m(F^p{\mathcal O_M},d_S)=0, \qquad \text{for $p>m$}
\end{equation} 
and the proof of convergence is complete. For the statement about the map
$\mathcal H^p(M,S)\to E_2^{p,0}$ see \cite{CartanEilenberg}, {Chap.~XV, Theorem~5.12}.
Since the product is compatible with the filtration, the edge homomorphism is
an algebra homomorphism.

The statement (iii) follows from the long exact sequence associated with the
short exact sequence
\[
0
\to F^{p+1}{\mathcal O_M}
\to{\mathcal O_M}
\to {\mathcal O_M}/F^{p+1}{\mathcal O_M}
\to 0,
\]
and \eqref{e-4}.
\end{proof}

\begin{remark}
Theorem \ref{t-co2} shows that although in general Tate resolutions require
in general an infinite dimensional $V$, computing the BRST cohomology in a
given degree is a finite process: computing $\mathcal H^j(M,S)$ requires
knowing $S$ modulo $F^{j+1}$, which in turn can be computed from the Tate
resolution down to degree $-j-1$.  
\end{remark}
\begin{cor}\label{c-neg}
  The BRST cohomology sheaf vanishes in negative degree.
\end{cor}
\begin{proof}
This is obvious as $E_2\cong \mathcal H(M,S)$ is the cohomology of a
complex $E_1$ concentrated in non negative degree.  
\end{proof}
\subsection{BRST cohomology in degree $0$}\label{sub-35}
Let us compute $\mathcal H^0(M,S)$ from the spectral sequence:
\[
\mathcal H^0(M,S)\cong\mathrm{Ker}(d_1\colon J(S_0)\to
\oplus_{i=1}^rJ(S_0)\beta^i), 
\]
where $\beta^1,\dots, \beta^r$  is a local basis of the locally free
$\mathcal O_X$-module $\mathcal O_V^1$.  The differential $d_1$ is induced
from the bracket with the terms in $S$ linear in the dual variables
$x_i^*\in \mathcal O_M^{-1}$. To compute it we need to construct a Tate
resolution down to degree $-2$.

The Lie algebra $L(S_0)=\{\xi\in T_X\,|\, \xi(S_0)=0\}$ acts on the algebra
$J(S_0)$ by derivations. The vector fields $\xi(S_0)\eta-\eta(S_0)\xi$, for
$\xi,\eta\in T$, act by zero and span a Lie ideal $L_0(S_0)$. Let
$L^\mathrm{eff}(S_0)=L(S_0)/L_0(S_0)$. It is a Lie algebra and an $\mathcal
O_X$-module and comes with an $\mathcal O_X$-linear Lie algebra homomorphism
$L^{\mathrm{eff}}(S_0)\to\mathrm{Der}(J(S_0))$. As $L_0(S_0)$ acts trivially on
$J(S_0)$, the action of $L^\mathrm{eff}(S_0)$ on $J(S_0)$ is defined and
we have
\begin{equation}\label{e-3}
J(S_0)^{L(S_0)}=J(S_0)^{L^{\mathrm{eff}}(S_0)}.
\end{equation}
\begin{prop}\label{p-35}
$\mathcal H^0(M,S)\cong J(S_0)^{L(S_0)}$
\end{prop}
\begin{proof}
The Tate resolution down to degree $-2$ looks like
\[
\cdots \to \wedge^2T_X\oplus \bigoplus_{i=1}^r\mathcal O_X\beta_i^*\to T_X
\to {\mathcal O_X}.
\]
The map $\wedge^2T_X\to T_X$ sends $\xi\wedge\eta$ for vector fields
$\xi,\eta\in T_X$ to $\xi(S_0)\eta-\eta(S_0)\xi$. Its image is $L_0(S_0)$.
Since the cohomology must vanish in degree $-1$, $\delta$ must map generators
$\beta_i^*$ of degree $-2$ to vector fields $\xi_i=\delta(\beta_i^*)$ which
together with $L_0(S_0)$ span the kernel $L(S_0)$ of $dS_0\colon T_X\to\mathcal
O_X$.  In other words the classes of $\xi_i$ generate the ${\mathcal
O}_X$-module $L^\mathrm{eff}(S_0)$.  Next we use the fact that $d_1$ is given
by the induced action of the component in $\Gamma(X,T_V[1])$ of $S$. Now
$S=S_\mathrm{lin}\mod I_M^{(2)}$ The only term contributing to $d_1\colon
J(S_0)\to\oplus_i J(S_0)\beta^i$ is then
$\sum\delta(\beta_i^*)\beta^i=\sum\xi_i\beta^i$ appearing in $S_\mathrm{lin}$:
indeed, by degree reasons, the terms linear in $T_V[1]$ in $I_M^{(2)}$ cannot
have a component in $T_X[1]$ and thus vanish when acting on $J(S_0)$.
Therefore $d_1f=\sum[\xi_i\beta^i,f]=-\xi_i(f)\beta^i$. Thus the kernel
consists of elements of $J(S_0)$ annihilated by vector fields spanning
$L(S_0)^{\mathrm{eff}}$. By \eqref{e-3} this proves the claim.  
\end{proof}

\begin{cor}\label{c-0}
Let $(M,S)$ be a BV variety with support $(X,S_0)$ and suppose that
$S_0$ has no critical points. Then $\mathcal H(M,S)=0$.
\end{cor}

\begin{proof}
In this case $J(S_0)=0$ and thus $E_2=0$.
\end{proof}

Thus the sheaf $\mathcal H(M,S)$ has support on the critical locus
of $S_0$.

\subsection{Hypercohomology}
The {\em BRST cohomology} $H(M,S)$ of a BV variety with support $(X,S_0)$ is
the hypercohomology of the BRST complex of sheaves. Then there is a
hypercohomology spectral sequence converging to $H^{p+q}(M,S)$ and whose
$E_2$-term is 
\[
E_2^{p,q}=H^p(X,\mathcal H^q(M,S)).
\]
The results on $\mathcal H$ in non-positive degree imply:
\begin{cor}\label{c-1}
The BRST cohomology of a BV variety $(M,S)$ with support $(X,S_0)$ vanishes
in negative degree and 
\[
H^0(M,S)=\Gamma(X,J(S_0)^{L(S_0)}).
\]
\end{cor} 
Also, there is a second hypercohomology spectral sequence
whose $E^{p,q}_2$ term is the $p$-th cohomology of $[S,\;\,]$ on $H^q(X,\mathcal O_M)$.
If $X$ is affine, the hypercohomology coincides with the cohomology
of global sections and we obtain:
\begin{cor}\label{c-2}
Let $X$ be an affine variety. Then
\[
H^\bullet(M,S)=H^\bullet(\Gamma(X,\mathcal O_{M}),d_S).
\]
\end{cor}

If $X$ is affine, the BRST cohomology is determined up to unique isomophism by
$(X,S_0)$, see Corollary \ref{c-unique}. It then makes sense to define the BRST
cohomology of $(X,S_0)$ as 
\[
H^\bullet (X,S_0)=H^\bullet(M,S),
\]
for any choice of $(M,S)$ with support $(X,S_0)$ as in Theorem \ref{t-00}.
From the Mayer--Vietoris sequence we then obtain:

\begin{cor}\label{c-MV}
Let $X$ be an affine variety.  Suppose that the critical locus of
$S_0\in\Gamma(X,\mathcal O_X)$ has two disjoint components $C_1\subset U_1$,
$C_2\subset U_2$ contained in open sets $U_1$, $U_2$ such that $C_1\cap
U_2=\varnothing=C_2\cap U_1$. Then 
\[
H^\bullet(X,S_0)=H^\bullet(U_1,S_0|_{U_1})\oplus
H^\bullet(U_2,S_0|_{U_2}). 
\]
\end{cor}
\section{Examples}\label{s-5andahalf}
In this section we discuss some examples of functions $S_0$ and
corresponding BV varieties $(M,S)$, see Definition \ref{def-BV}.

The simplest non-trivial examples are the quadratic forms.
\begin{example}\label{exa-1}
Suppose that $S_0=a_1(x^1)^2+\cdots+a_j(x^j)^2\in
\K[x^1,\dots,x^n]$ for some $a_i\in\K^\times$ 
with $0\leq j\leq n$. Then we can choose
the Tate resolution 
\[
R=\K[x_1^*,\dots,x_n^*,\beta^*_{j+1},\dots,\beta_n^*], \quad
\deg(x_i^*)=-1,\quad \deg(\beta_i^*)=-2,
\]
and $\delta(x_i^*)=2a_ix^i$, $\delta(\beta^*_i)=x_i^*$ ($i>j$). Then
\[
S=S_0+\sum_{i=j+1}^n x_i^*\beta^i
\]
is a solution of the classical master equation. The BRST cohomology is
one-dimensional concentrated in degree $0$. Indeed this example is obtained
from the zero example $(X=\{\mathrm{pt}\},S_0=0)$ by adding squares
(Prop.~\ref{p-2}) and taking the product with trivial BV varieties, see
\ref{ss-products}.  
\end{example}

\begin{example}\label{exa-2} 
Let $S_0$ be a regular function on a nonsingular affine variety $X$.
Suppose that $\Gamma(X,T_X)$ is spanned by vector fields $\xi_1,\dots,\xi_n$
with the property that $\xi_1(S_0),\dots,\xi_n(S_0)$ form a regular
sequence. Then we can choose $\Sym_{\mathcal O_X}(T_X[1])=\wedge T_X$, the
exterior algebra of the ${\mathcal O_X}$-module $T_X$, as a Tate resolution:
it is the Koszul resolution associated with the regular sequence. Then
$S=S_0$ and the BRST complex is concentrated in non-positive degree and has
cohomology $H^0(M,S)\cong J(S_0)$, $H^j(M,S)=0$, $j\neq 0$. This class
includes the case of isolated critical points.  
\end{example}
\begin{example}(Faddeev--Popov action)\label{exa-3}
Let $\mathfrak g$ be a finite dimensional Lie algebra over $\K$ acting on a
nonsingular algebraic variety $X$. We assume that the action of $\mathfrak
g$ is {\em infinitesimally free and transitive} on the fibers of a flat morphism
$\pi\colon X\to Y$ with $Y$ smooth and irreducible smooth fibers, namely
that it is given by a Lie algebra homomorphism $\theta\colon\mathfrak
g\to\Gamma(X,T_X)$ such that (a) $\theta(\mathfrak g)\subset
\Gamma(X,T_{X|Y})$ where $T_{X|Y}$ is the sheaf of vector fields tangent to
the fibers and (b) the action $\theta\colon\mathcal O_X\otimes\mathfrak g\to
T_{X|Y}$ is an isomorphism of $\mathcal O_X$-modules.  For example
$\pi\colon X\to Y$ could be a principal $G$-bundle where $G$ is a connected
algebraic group whose Lie algebra is $\mathfrak g$.

Suppose that $S_0\in\Gamma(X,\mathcal O_X)^{\mathfrak g}$ is the pull-back of a function
$S_{0,Y}$ on $Y$ with isolated critical points. Then a BV variety with
support $(X,S_0)$ is given by the classical construction of
Faddeev and Popov in the context of gauge theory \cite{FaddeevPopov}.  
The graded variety $V$ is $(X,\mathcal O_V)$ where
$\mathcal O_V=\mathcal O_X\otimes\wedge\mathfrak g^*$ with $\mathfrak g^*$ in
degree $1$. Thus 
\[
\mathcal O_{T^*[-1]V}=\wedge_{\mathcal O_X} 
T_X\otimes 
\Sym\,\mathfrak g
\otimes
\wedge \mathfrak g^*,
\]
with $\mathfrak g$ in degree $-2$ and $T_X$ in degree $-1$.  The solution of
the master equation is $S=S_0+a+b$, where $a\in T_X\otimes \mathfrak
g^*\cong\mathrm{Hom}_\K(\mathfrak g,T_X)$ is the infinitesimal action and
$b\in\wedge^2\mathfrak g^*\otimes \mathfrak
g\cong\mathrm{Hom}_\K(\wedge^2\mathfrak g,\mathfrak g)$ is (-1/2) times the
bracket. More explicitly, 
let $\beta^1,\dots,\beta^m$ (the ``Faddeev--Popov ghosts'') 
be a basis of $\mathfrak
g^*$ with dual basis $\beta_1^*,\dots,\beta_m^*$ of $\mathfrak g$ (the
``antighosts''), with commutation
relations $[\beta_i^*,\beta_j^*]_{\mathfrak g}
=\sum_\ell c_{ij}^\ell \beta_\ell^*$ and
fundamental vector fields $\theta_i=\theta(\beta_i^*)$.
Then
\begin{equation}\label{e-FP}
S=S_0+\sum_{i=1}^m\theta_i\beta^i-\frac12\sum_{i,j=1}^mc^\ell_{ij}\beta^*_\ell\beta^i\beta^j.
\end{equation}
It is a
well-known exercise to check that $[S,S]=0$, and it follows from the
more general result of the next example.  
Since $S|_X=S_0$ axioms (i), (ii) of BV varieties, see Definition \ref{def-BV}, are fulfilled

Let us check that $(M,S)$ obeys axiom (iii).  We have $R_M=\mathcal
O_M/I_M=\Sym\,\mathfrak g\otimes \wedge T_X$ with $\mathfrak g$ in degree $-2$
and $T_X$ in degree $-1$. The induced differential is the $\mathcal O_X$-linear
derivation sending $a\in\mathfrak g$ to $\theta_a\in T_X$ and $\xi\in T_X$ to
$\xi(S_0)\in\mathcal O_X$.  Consider the ascending filtration $0\subset
F_0R_M\subset F_1R_M\subset F_2R_M\subset\cdots$ where $F_pR_M$ is spanned by
$a\otimes \xi_1\wedge\cdots\wedge\xi_m\in\Sym\,\mathfrak g\otimes \wedge\mathfrak
T_X$, where at most $p$ among $\xi_1,\dots,\xi_m$ do not belong to $T_{X|Y}$. The
associated graded is then 
\begin{eqnarray*}
E^0_{p,q}&=&F_pR_M^{-p-q}/F_{p-1}R_M^{-p-q}\\
&=&\bigoplus_{2j+\ell=q} \Sym^j\mathfrak
g\otimes \wedge{}^\ell T_{X|Y}\otimes_{\mathcal O_X}\wedge{}^p(T_X/T_{X|Y}),
\quad p,q\geq0.
\end{eqnarray*}
Here we switched to chain complex conventions to get a first quadrant homology
spectral sequence instead of the third quadrant cohomology spectral sequence
$E_{r}^{-p,-q}=E^r_{p,q}$ corresponding to the descending filtration
$F^{-p}=F_p$. The differential $d_0$ is the derivation vanishing on $T_{X|Y}$
and on $T_X/T_{X|Y}$, and coinciding with $\theta\colon\mathfrak g\to T_{X|Y}$
on $\mathfrak g$. Since $d_0\colon \mathcal O_X\otimes \mathfrak g\to T_{X|Y}$
is an isomorphism, the cohomology of $(E^0_{p,\bullet},d_0)$ is concentrated in
degree $0$ and equal to $E^1_{p,0}=\wedge^pT_X/T_{X|Y}=\pi^*\wedge^pT_Y$.  The
next differential is $d_1=\pi^*\delta$, where $\delta(\xi)=\xi(S_{0,Y})$. For
isolated singularities $(\wedge T_Y,\delta)$ is a Koszul resolution and thus
has cohomology concentrated in degree $0$. Since $\pi$ is flat, $\pi^*$ is
exact and $E^2$ is concentrated in bidegree $(0,0)$. The spectral sequence
degenerates and the cohomology of $\mathcal O_M/I_M$ is trivial in nonnegative
degrees. Thus $(M,S)$ obeys (iii) and is indeed a BV variety.

The BRST complex is quasi-isomorphic to 
$(J(S_0)\otimes\wedge\mathfrak g^*,d_S)$.  The induced
differential $d_S$ is the Chevalley--Eilenberg differential of the $\mathfrak
g$-module $J(S_0)$.
\end{example}
\begin{example} (Lie algebroids)\label{exa-4}
The previous example is a special case of a more general construction: let $X$
be a nonsingular variety, $\mathcal L$ a Lie algebroid, i.e. a locally free
sheaf $\mathcal L$ of $\mathcal O_X$-modules with a Lie bracket
$[\;,\;]_{\mathcal L}\colon\mathcal L\otimes \mathcal L\to \mathcal L$ and a
Lie algebra homomorphism $\rho\colon \mathcal L\to T_X$, called the anchor,
such that $[v,fw]_{\mathcal L}=f[v,w]_{\mathcal L}+\rho_v(f)w$ for all $v,w\in
\mathcal L$, $f\in \mathcal O_X$.  As shown by Vaintrob \cite{Vaintrob}, Lie
algebroid structures on a locally free $\mathcal L$ are in one-to-one
correspondence with homological vector fields on the graded variety
$V=(X,\Sym_{\mathcal O_X}\mathcal L^*[-1])$, namely vector fields
$Q\in\Gamma(X,T_V)$ of degree 1 such that $[Q,Q]=0$. The vector field $Q$
corresponding to a Lie algebroid structure is the Chevalley--Eilenberg
differential on $\mathcal O_V= \wedge_{\mathcal O_X} \mathcal L^*$ namely the
derivation sending $f\in\mathcal O_X$ to $Q(f)\in\mathcal L^*$ with
$Q(f)(v)=\rho_v(f)$ and $\beta\in \mathcal L^*$ to $Q(\beta)\in\wedge^2\mathcal
L^*\cong\scriptHom_{\mathcal O_X}(\wedge^2\mathcal L,\mathcal O_X)$ given by
\[
u\wedge v\mapsto \rho_u\beta(v)-\rho_v\beta(u)
-\beta([u,v]_{\mathcal L}),\quad u,v\in\mathcal L.
\]
 The corresponding Hamiltonian function $S_Q$ on $M=T^*[-1]V$ of degree zero is
then a solution of the master equation. 

If $S_0\in\Gamma(X,\mathcal O_X)$ is a regular function invariant under the Lie
algebroid, in the sense that $\rho_u(S_0)=0$ for all local sections $u\in \mathcal L$,
then $S=S_0+S_Q$ is a solution
of the master equation.  It has the same local form as \eqref{e-FP} with
$\xi_i=\rho_{\beta_i^*}$ for some local basis $\beta_i^*$ of $\mathcal L$. The
main difference is that the structure constants $c_{jk}^\ell$ are not constants
but functions in $\mathcal O_X$.   

The pair $(M,S)$ is a BV variety if the cohomology of $\mathcal O_M/I_M$ is
trivial in nonzero degree. As in the case of the previous example, this follows
if there is flat morphism $p\colon X\to Y$, such that $\rho\colon \mathcal L\to
T_{X|Y}$ is an isomorphism and $S_0$ is the pull-back of a function on $Y$ with
isolated critical points.  
\end{example}
In the next example we consider the case of the function $0$ on
a (not necessarily affine) variety $X$. We show that 
the $(-1)$-shifted cotangent bundle of the $1$-shifted
tangent bundle of any variety $X$ has a canonical structure
of BV variety with support $(X,0)$. Its BRST cohomology is
the de Rham cohomology with trivial bracket. 
\begin{example}\label{exa-5andahalf}
Let $\Omega^\bullet_X=\Sym_{\mathcal O_X}^\bullet T^*_X[-1]$ 
be the sheaf of differential forms
on a nonsingular variety $X$ and let $T[1]X$ (the 1-shifted tangent
bundle to $X$) denote 
the $\mathbb Z_{\geq 0}$-graded variety $T[1]X=(X,\Omega_X^\bullet)$.
The de Rham differential $d$ is a derivation of $\Omega_X^\bullet$ of degree 1
obeying $[d,d]=0$. Let $M=T^*[-1](T[1]X)$. Let $S$ be $d$, viewed as a 
section of degree 0 
of $\mathrm{Der}(\Omega_X^\bullet)[1]\subset \mathcal O_{M}$.
The function $S\in\Gamma(X,\mathcal O_M)$ 
is a hamiltonian function of $d$ and is a solution of the
classical master equation. 
\begin{prop}
The pair $(M=T^*[-1](T[1]X),S)$ is a BV variety with support $(X,0)$.
The canonical inclusion of scalars
\[
i\colon\Omega^\bullet_X\to \mathcal O_M
=\Sym_{\Omega^\bullet_X}(\mathrm{Der}(\Omega^\bullet_X)[1])
\]
is a quasi-isomorphism of sheaves of differential $P_0$-algebras from the
de Rham algebra $(\Omega^\bullet_X,d)$, viewed as a $P_0$-algebra with trivial
bracket to the BRST complex $(\mathcal O_M,d_S)$.
\end{prop}
The first statement follows from the fact that this example
is a special case of the previous one, as the tangent bundle
with the Lie bracket of sections and identity anchor is a Lie algebroid.
To prove the second statement, notice that the morphism $i$ is
the inclusion of the first summand in
\[
\mathcal O_M=\Omega^\bullet_X\oplus 
\mathrm{Der}(\Omega^\bullet_X)[1]\oplus\Sym^2_{\Omega^\bullet_X}
(\mathrm{Der}(\Omega^\bullet_X)[1])\oplus\cdots,
\footnote{There is no completion here as the degrees are bounded above}
\]
and thus comes with a
projection $p\colon \mathcal O_M\to \Omega^\bullet_X$ such that
$p\circ i=\mathrm{id}$. The left inverse $p$ is a morphism of
differential graded commutative algebras. It is sufficient to show
that the chain map $i\circ p$ is locally
homotopic to the identity. To show this notice that every point has
an open neighborhood $U$ such that 
$\mathrm{Der}(\Omega^\bullet_X)|_U$ is a free $\Omega^\bullet_U$-module generated by
 interior multiplication $\iota_j$ and Lie derivative 
$L_j=d_S(\iota_j)=[d,\iota_j]$ by
vector fields $\partial_j$, $j=1,\dots,\mathrm{dim}\,X$ trivializing $T_X$.
Then the $\Omega^\bullet_X$-linear derivation $h$ of $\mathcal O_M|_U$ sending
$L_j$ to $\iota_j$ and $\iota_j$ to zero obeys
\[
d_S\circ h+h\circ d_S=m\,\mathrm{id} \quad
\text{on $\Sym^m_{\Omega^\bullet_X}(\mathrm{Der}(\Omega^\bullet_X)[1])|_U$,}
\]
 and $H=0\oplus\bigoplus_{m>0}\frac1m h$ is a homotopy between $i\circ p$ and
the identity of $\mathcal O_M$.
\end{example}

\begin{example}\label{exa-6}  
Let $\pi:X\to Y$ be a vector bundle with a nondegenerate symmetric bilinear form
$\langle\ ,\ \rangle$ on the fibers.
The associated quadratic form 
\[
S_0(v)=\frac12 \langle v,v\rangle
\]
is a regular function on $X$ whose critical locus is the zero section.  BV
varieties with support $(X,S_0)$ can be obtained from orthogonal connections
$\nabla$ on $X$. Let $\mathcal E$ be the locally free $\mathcal O_Y$ module of
sections of $\pi$ so that $\mathcal O_X=\Sym_{\mathcal O_Y}\mathcal E^*$.
Suppose that $\nabla\in\mathrm{Hom}(\mathcal E,\mathcal E\otimes_{\mathcal O_Y}\Omega^1_Y)$ 
is a connection such that $d\langle u,v\rangle= \langle\nabla
u,v\rangle+\langle u,\nabla v\rangle$ for all $u,v\in\mathcal
E$. 
Let $V$ be the graded variety $(X,\pi^*\Omega^\bullet_Y)$. The solution $S$ of
the classical master equation is a function on $T^*[-1]V$. 
 Notice first that for an open set $U\subset Y$, 
\begin{equation}\label{e-Gus}
\mathcal O_V(U)=\pi^*\Omega^\bullet_Y(\pi^{-1}(U))
=
\Gamma(U,\Sym_{\mathcal O_Y}\mathcal E^*\otimes_{\mathcal O_Y}\Omega_Y).
\end{equation}
The connection $\nabla$ induces a connection $\nabla^*$ 
on $\mathcal E^*$ and therefore a
derivation of degree $1$, also denoted $\nabla^*$ 
of $\Sym_{\mathcal O_Y}\mathcal E^*\otimes_{\mathcal
O_Y}\Omega^\bullet_Y$ and via \eqref{e-Gus} a vector field of degree $1$ on $V$.
Let $S_\nabla$ be the corresponding hamiltonian: it is just $\nabla^*$ viewed as
a function on $T^*[-1]V$.
 Let 
$F^*\in\Gamma(Y,\Omega_Y^2\otimes_{\mathcal O_Y} \scriptEnd(\mathcal E^*))$ 
be the curvature of $\nabla^*$. Since the bilinear form is 
non-degenerate, it defines an isomorphism $b\colon \mathcal E^*\to\mathcal E$; 
therefore $b\circ F^*$ may be viewed as bilinear form on $\mathcal E^*$ 
which, by the orthogonality of the connection, is skew-symmetric and thus
defines
an element $4S_F$ of 
$\Omega^2_Y\otimes\wedge^2\mathcal E\subset \Sym^2_{\mathcal O_V}
(T_V[1])$.
The claim is that
\[
S=S_0+S_\nabla+S_F
\]
is a solution of the master equation. 
To check it, let us realize $\mathcal O_M$ as the symmetric algebra of
$T_V[1]$ and notice that $S_0$ is a function in $\mathcal O_V$,
so $[S_0,S_0]=0$. The fact that the connection is orthogonal is equivalent
to $[S_\nabla,S_0]=\nabla^* S_0=0$. 
Then $[S_\nabla,S_\nabla]$ is the derivation $[\nabla^*,\nabla^*]$ which is
$2F^*$ on generators in $\mathcal E^*$ and is extended as a derivation
of $\mathcal O_V$. We claim that 
\[
[S_\nabla,S_\nabla]+2[S_0,S_F]=0.
\]
Indeed, if $S_0=\frac12\sum g_{ij}v^iv^j$ for some 
local basis $v^i$ of $\mathcal E^*$ and $F^*(v^i)=\sum_jF^i_jv^j$ for
local sections $F^i_j$ of $\Omega^2_Y$, 
then $S_F=\frac14\sum F^{ij}\frac{\partial}{\partial v^i}\wedge
\frac\partial{\partial v^j}$  where $\sum_j F^{ij}g_{jk}=F^i_k$.
Thus $[S_0,S_F]=\sum F^{ij}g_{jk}v^k\frac{\partial}{\partial v^i}$,
which is the derivation sending $v^i$ to $\sum_j F^i_jv^j$ and is
thus equal to $-F^*$.
Finally $[S_\nabla,S_F]=0$  follows from the Bianchi identity 
$\nabla F=0$ and $[S_F,S_F]$
clearly vanishes. 

Alternatively, one can do an explicit calculation in the local description using
local functions $y^1,\dots,y^m\in\mathcal O_Y(U)$ such that $dy^1,\dots,d y^m$
are a basis of $\Omega_Y(U)$ on some neighborhood $U$ of a point of $Y$ and
local sections $v^i$ as above. Then one obtains a slightly more
general statement: let $A$ be the graded polynomial algebra over $\mathcal O_Y$
in variables $y^*_\mu,v^i,v^*_i,\beta^\mu,\beta^*_\mu$, $\mu=1,\dots,m$,
$i=1,\dots,r$ of degrees $-1,0,-1,1,-2$ and nontrivial Poisson brackets
among generators $[f,y^*_\mu]=\partial_\mu f$, $f\in\mathcal O_Y$, $[\beta^\mu,\beta^*_\nu]=\delta_{\mu\nu}$, $[v^i,v_j^*]=\delta_{ij}$. Suppose  
$(g_{ij})$ is a (not necessarily invertible) matrix
with entries in $\mathcal O_Y$,
$A^i_j=A^i_{j\mu}dy^\mu$ are one-forms on $Y$,
$F^{ij}=\frac12F^{ij}_{\mu\nu}dy^\mu\wedge dy^\nu$ are two-forms with
$F^{ij}=-F^{ji}$, such that
(a) $\partial_\mu g_{ij}+2A^\ell_{i\mu}g_{\ell j}=0$ (orthogonality),
(b) $dA^i_j+A^i_k\wedge A^k_j=F^{ik}g_{kj}$ (structure equation)
and (c) $dF^{ij}+A^i_\ell\wedge
F^{\ell j}-A^j_\ell\wedge F^{\ell i}=0$ (Bianchi identity). Then
\[
S=  
\frac12 g_{ij}v^iv^j
+\beta^\mu(y^*_\mu-A^i_{j\mu}v^jv^*_i)+
\frac14F^{ij}_{\mu\nu}\beta^\mu \beta^\nu v^*_iv^*_j,
\]
(we omit summation signs over repeated indices)
obeys the classical master equation in $A$. 

Let us check that the vanishing axiom (iii) of Definition \ref{def-BV} 
holds in this case by
writing the induced differential on $R_M(\tilde U)$ for some open set 
$\tilde U=\pi^{-1}U$ on which the local description is valid.

The induced differential $\delta$ on 
\[
R_M(\tilde U)=\mathcal O_Y(U)[v^i,v_i^*,y_\mu^*,\beta_\mu^*]
\]
is the $\mathcal O_Y$-linear derivation such that
\[
\delta \beta^*_\mu=y^*_\mu- A^i_{j\mu}v^jv^*_i,\quad
\delta v_i^*= g_{ij}v^j,\quad
\delta v^i=0,\quad
\delta y^*_\mu=\frac12\partial_\mu g_{ij}v^iv^j.
\]
Let $(g^{ij})\in\mathcal O_Y(U)$ the matrix inverse to $(g_{ij})$.
Let $h$ be the $\mathcal O_Y$-linear derivation of degree -1 such that
\[
h v^i=g^{ij}v_j^*,\quad
h y^*_\mu=\beta^*_\mu+ A^{i}_{j\mu}g^{jk}v_k^*v_i^*,\quad
h\beta^*_\mu=0=h v_i^*.
\]
Then it is easy to check that for $x=v^i,v_i^*,y^*_\mu,\beta^*_\mu$, one
has $(h\circ\delta+\delta\circ h)(x)=x$. 
It follows that the cohomology is spanned
by polynomials of degree 0 in these variables,  namely those
in $\mathcal O_Y(U)$. Thus $(R_M,\delta)$ is indeed a resolution of
the Jacobian ring $J(S_0)(\tilde U)\cong\mathcal O_Y(U)$.
\end{example}

The next example illustrates 
the fact that polynomial functions $S_0$ on affine spaces
with critical locus of positive dimension usually lead to complicated solutions
of the master equation. It is also an example with nontrivial BRST cohomology in
positive degree.
\begin{example}\label{exa-7} 
Let $X=\mathbb A^2$ be the affine plane with coordinates $x,y$, and set
\[
S_0=(x^2+y^2-1)^2/4.
\]
The Jacobian ring is
$J(S_0)=\K[x,y]/(xh,yh)$, where $h=x^2+y^2-1$.
The critical locus is $h^{-1}(0)\cup\{0\}$. 
We have
$L(S_0)=\K[x,y]\tau$, $\tau=y\partial_x-x\partial_y$, $L_0(S_0)=\K[x,y]h\tau$.
$L^{\mathrm{eff}}(S_0)$ is the ${\mathcal O}$-module generated by $\tau$ with
relation $h\tau=0$. There does not appear to be a finitely generated Tate
resolution in this case. Here is a BV variety $(M,S)$
up to filtration degree $4$: $V$ has coordinates $x,y$ of degree $0$, $\beta$
of degree $1$, $\gamma$ of degree $2$, $\xi,\eta$ of degree 3,
$\rho,\mu,\nu,\phi$ of degree $4$. The solution of the master equation up to
filtration degree four, computed with the help of Macaulay2 \cite{M2} is 
\begin{eqnarray*} S&\equiv& S_0
+(xy^*-yx^*)\beta
+(h\beta^*-x^*y^*)\gamma
\\
&&+(x\gamma^*-x^*\beta^*)\xi
+(y\gamma^*-y^*\beta^*)\eta
+(xy^*\gamma^*-yx^*\gamma^*+x^*y^*\beta^*)\rho
\\
&&+(h\xi^*-x^*\gamma^*)\mu
+(h\eta^*-y^*\gamma^*)\nu+(\beta^*)^2\gamma^2
-\eta^*\beta\xi
+\xi^*\beta\eta
\\
&&
+(y\xi^*-x\eta^*-(\beta^*)^2/2)\phi\mod F^5.
\end{eqnarray*} It is probably hopeless to compute the BRST cohomology using
this formula. Instead we can use Corollary \ref{c-MV} to write $X=U_0\cup U_1$
with $U_c=\{(x,y)\in \mathbb A^2\,|\,x^2+y^2\neq c\}$.  Then $S_0|_{U_1}$ has an isolated critical
point with one dimensional Jacobian ring, so that
$H^\bullet(U_1,S_0|_{U_1})=\K$, concentrated in degree $0$. The map
$h\colon U_0\to\mathbb A^1$ is a principal $SO(2)$-bundle and $S_0|_{U_0}$ is
the pull-back of $h^2\in \mathcal O(\mathbb A^1)=\K[h]$, so we are
in the setting of Faddeev--Popov as in Example \ref{exa-3}.

The BRST complex is quasi-isomorphic 
to the Chevalley--Eilenberg complex of the one-dimensional Lie
algebra $\mathfrak{so}(2)=\K$ of the rotation group with values in the algebra
$A=\Gamma(U_0,J(S_0))$ of functions on the critical set. By definition
$A=\K[x,y,(x^2+y^2)^{-1}]/I$ where $I$ is the ideal generated by $xh,yh$. But $I$
is also generated by $h$ since $h=(x^2h+y^2h)(x^2+y^2)^{-1}$.  So over an algebraic
closure $\bar \K$, $A\otimes_\K\bar\K=\bar\K[t,t^{-1}]$, $t=x+\sqrt{-1}y$. The
Chevalley--Eilenberg complex is concentrated in degree $0$ and $1$ with
differential $\tau\colon A\to A$.  Since $\tau(t^j)=j\sqrt{-1} t^j$ we see that
$H^0(\mathfrak{so}(2),A)=\K 1$ and $H^1(\mathfrak{so}(2),A)=\K\alpha$, where
$\alpha$ is the class of $1$.
Thus
\[
H^\bullet(U_0,S_0|_{U_0})=\K[\alpha],\qquad \deg\alpha=1,\qquad \alpha^2=[\alpha,\alpha]=0.
\]
By Corollary \ref{c-MV}, the BRST cohomology is three-dimensional and
is the direct sum of algebras
\[
H^\bullet(X,S_0)=\K\oplus\K[\alpha], \qquad \deg\alpha=1,
\]
with trivial bracket.
\end{example}

The next is an example, due to Pavel Etingof, 
with infinite dimensional BRST cohomology.

\begin{example}\label{exa-8}
Let $X=\mathbb A^4$ with coordinates $x,y,z,w$ and
\[
S_0=x^3+y^3+z^3-3wxyz.
\]
This function is of weight $3$ if we assign weight $1$ to $x,y,z$ and $0$ to
$w$.  We claim that the zeroth BRST cohomology $J(S_0)^{L(S_0)}$ is infinite
dimensional. The Jacobian ring $J(S_0)=\K[x,y,z,w]/(x^2-wyz,y^2-wxz,z^2-wxy,xyz)$
is the quotient by an ideal generated by homogeneous elements and is thus
graded with respect to the weight. In weights $0,1,2$ it is a free
$\K[w]$-module generated by $1,x,y,z,xy,xz,yz$. In weight $3$ we have a free
$\K[w]/(w^3-1)\K[w]$-module generated by $x^2y$, $xy^2$. It follows that every
$f\in J(S_0)$ of weight $\geq3$ obeys $(w^3-1)f=0$.  Every vector field in
$L(S_0)$ is a sum of weight homogeneous vector fields. It is clear that there
are no vector fields of weight $-1$ in $L(S_0)$. A vector field of weight $0$
in $L(S_0)$ has the form $\xi=a\partial_x+b\partial_y+c\partial_z+d\partial_w$
with $a,b,c$ of weight $1$ and $d\in\K[w]$. Since $\xi(S_0)\equiv
3(ax^2+by^2+cz^2-dxyz)\mod w$ it follows that $a,b,c,d$ must all be divisible
by $w$. But if $\xi\in L(S_0)$ is divisible by $w$ then also $w^{-1}\xi\in
L(S_0)$.
It follows that $\xi=0$. 
Thus all homogeneous vector fields in
$L(S_0)$ have weight at least $1$. Let $g\in J(S_0)$ be of weight $2$. Then
$h=(w^3-1)^2g$ is annihilated by $L(S_0)$ since $\xi(h)$ is of weight at least
three and is divisible by $w^3-1$. But the space of such $h$ is infinite dimensional. 
\end{example}

\section{BRST cohomology in degree 0 and 1 and Lie--Rinehart cohomology}\label{s-6}
In this Section we show that, up to degree 1, the BRST cohomology of a
BV variety with support $(X,S_0)$, with $X$ an affine variety, coincides with
the de Rham cohomology of a Lie--Rinehart algebra associated with the critical
locus of $S_0$.
We also describe the induced bracket $H^0(M,S)\otimes H^0(M,S)\to
H^1(M,S)$ geometrically.  The degree $1$ cohomology appears as an obstruction
to extend a solution of the classical master equation to a solution of the
quantum master equation. The bracket $H^0\otimes H^0\to
H^1$ appears in deformation theory as the obstruction to extend an
infinitesimal deformation of a solution of the classical master equation to a
solution in formal power series.

\subsection{Lie--Rinehart algebras}
Recall that a Lie--Rinehart algebra over $\K$ is a pair $(A,\mathfrak g)$ where
$A$ is a commutative algebra over $\K$ and $\mathfrak g$ is both a Lie algebra
over $\K$ acting on $A$ by derivations and an $A$-module.
The required compatibility conditions are
$[\tau,f\cdot\sigma]=\tau(f)\cdot\sigma+f\cdot[\tau,\sigma]$ and
$(f\cdot\tau)(g)=f\cdot(\tau(g))$ for all $f,g\in A$ and
$\tau,\sigma\in\mathfrak g$. Here $f\otimes\tau\mapsto f\cdot\tau$ denotes the
module structure and $\tau\otimes f\mapsto \tau(f)$ the action of $\mathfrak g$
on $A$.  The main example is $(A,\mathrm{Der}(A))$ for an algebra $A$.

Lie--Rinehart algebras come with a differential graded algebra, 
the de Rham (or Chevalley--Eilenberg) complex
$\Omega_{\mathrm{dR}}^\bullet(A,\mathfrak{g})=\mathrm{Hom}_{A}(\wedge^\bullet_A\mathfrak{g},A)$ introduced by Rinehart \cite{Rinehart}.
We only consider this complex for $p\leq 2$. For
$g\in\Omega_{\mathrm{dR}}^0(A,\mathfrak{g})=A$ and
$\alpha\in\Omega_{\mathrm{dR}}^1(A,\mathfrak{g})$ we have 
\[
dg(\tau)=\tau(g),\quad d\alpha(\tau,\sigma)=\tau(\alpha(\sigma))-\sigma(\alpha(\tau))-\alpha([\tau,\sigma]),
\qquad
\tau,\sigma\in \mathfrak{g}.
\]
An {\em ideal} of a Lie--Rinehart algebra $(A,\mathfrak g)$ is a pair $(I,\mathfrak h)$ such that
\begin{enumerate}
\item[(i)] $I$ is an ideal of $A$.
\item[(ii)] $\mathfrak h$ is a Lie ideal of $\mathfrak g$.
\item[(iii)] $i\cdot\tau\in \mathfrak h$ and $\tau(i)\in I$ for all $ i\in I$, $\tau\in \mathfrak g$.
 \item[(iv)]  $f\cdot\sigma\in\mathfrak h$ and $\sigma(f)\in I$ for all $f\in A$, $\sigma\in\mathfrak h$.
\end{enumerate}
\begin{lemma}\label{l-n1}
Let $(I,\mathfrak h)$ be an ideal of the Lie--Rinehart algebra $(A,\mathfrak g)$. Then
$(A/I,\mathfrak g/\mathfrak h)$ is naturally a Lie--Rinehart algebra.
\end{lemma}

\begin{proof} The conditions (i)--(iv) guarantee that the structure maps defined on representatives are well-defined
on the quotient.
\end{proof}

\subsection{Cohomology in degree $\leq 1$}
Let $S_0\in{\mathcal O}(X)$ be a regular function
and denote as before $J=J(S_0)={\mathcal O}_X/\mathrm{Im}(dS_0\colon T_X\to {\mathcal O}_X)$ the Jacobian
ring, $L=L(S_0)$ the Lie algebra of vector fields annihilating $S_0$, $L_0=L_0(S_0)$ the ${\mathcal O}_X$-submodule 
of $L$ spanned by $\xi(S_0)\eta-\eta(S_0)\xi$, $\xi,\eta\in T$.

\begin{lemma}\label{l-n2}
$(\mathrm{Im}(dS_0\colon T_X\to {\mathcal O_X}),L_0)$ is an ideal of the Lie--Rinehart algebra $({\mathcal O}_X,L)$.
\end{lemma}

\begin{proof} (i) The subspace $\mathrm{Im}(dS_0)$ consists of functions of the form $\xi(S_0)$, $\xi\in T_X$. 
It is clearly an ideal of ${\mathcal O_X}$.
(ii) Let $\tau\in L$. By exploiting the fact that $\tau(S_0)=0$, we have
\[
[\tau,\xi(S_0)\eta-\eta(S_0)\xi]=[\tau,\xi](S_0)\eta-\eta(S_0)[\tau,\xi]+\xi(S_0)[\tau,\eta]-[\tau,\eta](S_0)\xi,
\]
which lies in $L_0$.
(iii) Let $\tau\in L$, $\xi(S_0)\in I$. Then, again because $\tau(S_0)=0$,
$\xi(S_0)\tau=\xi(S_0)\tau-\tau(S_0)\xi\in L_0$ and $\tau(\xi(S_0))=
[\tau,\xi](S_0)\in I$.
(iv) Let $f\in {\mathcal O_X}$, $\nu=\xi(S_0)\eta-\eta(S_0)\xi\in L_0$. Then clearly $f\nu\in L_0$ and $\nu(f)=
(\eta(f)\xi-\xi(f)\eta)(S_0)\in I$.
\end{proof}

It follows that $(J,\mathfrak g=L/L_0)$ is a Lie--Rinehart algebra. We use the
notation $\mathfrak g=\mathfrak g(S_0)$ to denote the $J$-module $L/L_0$. It is
convenient to distinguish it from $L^{\mathrm{eff}}$ which is $L/L_0$
considered as an ${\mathcal O_X}$-module.
\begin{theorem}\label{t-deRham}
The BRST cohomology is isomorphic to the de Rham cohomology up to degree $1$:
\[
H^p(M,S)\cong H_{\mathrm{dR}}^p(J(S_0),\mathfrak{g}(S_0)) \text{ for } p=0,1.
\]
\end{theorem}

\begin{remark} The de Rham cohomology of a Lie--Rinhart algebra $(A,\mathfrak g)$
is the appropriate cohomology for 
Lie--Rinehart algebras only if the Lie algebra $\mathfrak g$ is a projective module
over the commutative algebra $A$. If it is not one should replace $\mathfrak g$
by a quasi-isomorphic differential graded Lie algebra which is
projective over the commutative algebra, see \cite{CasasLadraPirashvili}.
The resulting cohomology groups $H^p(A,\mathfrak g)$ coincide with the
de Rham groups for $p=0,1$ but are in general different. We conjecture
that $H^p(M,S)\cong H^p(J(S_0),\mathfrak{g}(S_0))$ for all $p\geq0$.
\end{remark} 

We have seen that $H^0(M,S)\cong J^L=J^\mathfrak g$ since $L_0$ acts trivially. This is the
claim for $p=0$.
To prove this theorem for $p=1$ we need to construct a solution of the master equation modulo $F^3$.

Let  $\tau_1,\dots,\tau_r$ be a system of generators of $L/L_0$ as an ${\mathcal O_X}$-module.
Then
\begin{equation}\label{e-00}
[\tau_i,\tau_j]=\sum_{k=1}^rf_{ij}^k\tau_k,
\end{equation}
for some (non-unique) $f_{ij}^k\in{\mathcal O_X}$.

Let $r_{aj}\in{\mathcal O_X}$, $a=1,\dots,s$, $j=1,\dots,r$,  a generating system of relations, namely a set of elements of ${\mathcal O_X}$ such that
\begin{enumerate}
\item[(i)]
$
\sum_{j=1}^rr_{aj}\tau_j=0, \qquad a=1,\dots s
$
\item[(ii)] if $\sum_{j=1}^r r_j\tau_j=0$ with $r_j\in{\mathcal O_X}$
  then $r_j=\sum_{a=1}^s f_ar_{aj}$ for some $f_a\in{\mathcal O_X}$.
\end{enumerate}
In other words we choose a presentation of the ${\mathcal O_X}$-module
$L^{\mathrm{eff}}(S_0)$:
\[
\mathcal O_X^s\to\mathcal O_X^r\to L^{\mathrm{eff}}(S_0)\to 0.
\]
Using this presentation we construct a Tate resolution
$R=(\Sym_{\mathcal O_X}\mathcal W,\delta)$ down to degree $-3$, with
$\mathcal W$ of the form $\mathcal W=T_X[1]\oplus \mathcal E^*[1]$,
see \ref{ss-TateResolutions}, and a solution $S\in
\Gamma(X,\mathcal O_{T^*[-1]V})$ of the classical master equation modulo $F^3$
on the shifted cotangent bundle of $V=(X,\Sym_{\mathcal O_X}\mathcal E)$. We set $\mathcal
W^{-1}=T_X$, $\mathcal W^{-2}=\mathcal O_X^r$ with basis
$\beta_1^*,\dots,\beta_r^*$, $\mathcal W^{-3}=\mathcal O_X^s\oplus E$,
where $\mathcal O_X^s$ has basis $\gamma_1^*,\dots,\gamma_s^*$ and $E$
is a free ${\mathcal O_X}$-module to be determined and whose precise
form will not be important for the computation of the cohomology. The
beginning of the Tate resolution looks like
\[
R^{-3}
\stackrel{\delta}{\to}
R^{-2}
\stackrel{\delta}{\to}
R^{-1}
\stackrel{\delta}{\to}
R^{0},
\]
with
\[
R^{-3}=\wedge^3 T_X\oplus (T_X\otimes \mathcal O_X^r)\oplus \mathcal O_X^s\oplus E,\quad
R^{-2}=\wedge^2 T_X\oplus \mathcal O_X^r,\quad
R^{-1}=T_X,\quad
R^0=\mathcal O_X.
\]
Here $\wedge^\bullet T_X=\wedge^\bullet_{\mathcal O_X}T_X$ is the exterior algebra
of the $\mathcal O_X$-module $T_X$.
The differential $\delta\colon R^{-1}\to R^0$ is $\delta(\xi)=\xi(S_0)$.
The differential on the other components depends on a choice of lifts of $\tau_1,\dots,\tau_r\in L/L_0$ to vector fields $\hat\tau_1,\dots,\hat\tau_r\in L$.
We set
\[
\delta(\beta_j^*)=\hat\tau_j,
\]
so that $\delta^2|_{\mathcal O_X^r}=0$.
The relations for the lifts hold modulo $L_0$:
\begin{equation}\label{e-n1}
\sum_j r_{aj}\hat \tau_j+dS_0\lrcorner v_a=0,
\end{equation}
for some $v_a\in\wedge^2 T_X$. Here $\lrcorner$ is the contraction operator:
$dS_0\lrcorner\xi\wedge\eta=\xi(S_0)\eta-\eta(S_0)\xi\in L_0$.  By the Leibniz
rule we have $\delta(v)=dS_0\lrcorner v$ for $v\in\wedge^2T_X$. Thus
Eq.~\eqref{e-n1} implies that $v_a+\sum_j r_{aj}\beta^*_j$ is a cocycle for
$a=1,\dots,s$.
We set
\[
\delta(\gamma^*_a)=v_a+\sum_jr_{aj}\beta^*_j\in\wedge^2T_X\oplus \mathcal O_X^r.
\]
Thus $\delta^2|_{\mathcal O_X^s}=0$.

Let $R_0^{-2}$ be the ${\mathcal O_X}$-submodule spanned by $\xi(S_0)r$ with
$\xi\in T_X$, $r\in R^{-2}$. The differential on $E$ will be chosen so that
$\delta(E)\subset R_0^{-2}$.

Let us check that we indeed get a resolution down to degree $-3$. The kernel of
$\delta\colon T_X\to{\mathcal O_X}$ is $L$. Since $\delta(\wedge^2T_X)=L_0$ and
$\tau_1,\dots,\tau_r$ span $L/L_0$, the complex is exact at $R^{-1}$. The
kernel of the next differential consists of pairs $v\in\wedge^2T_X$, $\sum
r_i\beta^*_i\in\mathcal O_X^r$ such that $\sum r_i\hat \tau_i+dS_0\lrcorner
v=0$. We need to show that these cocycles are equivalent to cocycles lying in
$R_0^{-2}$, so that they can be killed by a judicious choice of $E$. Reducing
modulo $L_0$, we see that $\sum r_i\tau_i=0$ so that $r_i=\sum f_ar_{ai}$, see
(ii).  Possibly subtracting an element of $\delta(\mathcal O_X^s)$, we may
assume that $r_i=0$, $i=1,\dots,r$. Thus $dS_0\lrcorner v=0$ and therefore $v$
is a sum of bivector fields of the form $\xi\wedge\eta$ where $\eta\in L$,
i.e., $\eta$ is a linear combination of $\hat \tau_i$ and $dS_0\lrcorner w$,
$w\in\wedge^2T_X$, the latter being in $R_0^{-2}$. We can get rid of terms
$\xi\wedge\hat\tau_i$ modulo $R^{-2}_0$ by subtracting a coboundary from
$\delta(T_X\otimes \mathcal O_X^r)$: indeed we have 
\[
\delta(\xi\otimes\beta_i^*)=\xi(S_0)\beta_i^*-\xi\wedge\hat\tau_i\equiv
-\xi\wedge\hat\tau_i\mod R_0^{-2}.  
\] 
Thus any (-2)-cocycle is equivalent
modulo coboundaries to a cocycle lying in $R_0^{-2}$. Exactness at $R^{-2}$ is
achieved by defining $\delta$ to map $E$ onto the space of cocycles in
$R_0^{-2}$.

With this information on the Tate resolution we can construct a solution of the
master equation modulo $F^3$ along the line of the existence proof in
Section \ref{ss-existence} (b).  The first approximation is
\[
S_{\leq 1}\equiv S_0+\sum \delta(\beta_i^*)\beta^i+\sum \delta(\gamma_a^*)\gamma^a+\sum \delta(\rho^*_b)\rho^b\mod F^3,
\]
where $(\rho^*_b)$ is a basis of $E$. By construction, $[S_{\leq 1}, S_{\leq 1}]$ is
a section of $I^{(2)}$, see \ref{ss-existence} (b) (ii). So the first violation of the master
equation appears from the bracket of the second term on the right-hand side
with itself: $[S_{\leq 1},S_{\leq 1}]\equiv \sum
[\hat\tau_i,\hat\tau_j]\beta^i\beta^j\mod F^3$. The relation \eqref{e-00} lifts
to
\[
[\hat\tau_i,\hat\tau_j]=\sum f_{ij}^k\hat\tau_k+dS_0\lrcorner g_{ij},
\]
for some bivector field $g_{ij}\in \wedge^2T_X$. This dictates the form of the corrections to $S$. We obtain
\begin{eqnarray*}
S&\equiv&S_0
+\sum_i\underline{\hat\tau_i\beta^i}
+\sum_{a,j}\underline{r_{aj}\beta^*_j\gamma^a}
+\frac12 \sum_a v_a\gamma^a\\
&&+\sum_b\underline{\delta(\rho^*_b)\rho^b}
+\frac12\sum_{i,j} g_{ij}\beta^i\beta^j-\frac12\sum_{i,j,k}
\underline{f_{ij}^k\beta_k^*\beta^i\beta^j}\mod F^3.
\end{eqnarray*}
The BRST cohomology is isomorphic to the cohomology of the first term
$E^{\bullet,0}_1=J[\beta^i,\gamma^a,\rho^b,\dots]$ of the spectral sequence.
The beginning of the complex $E^{\bullet,0}_1$ is
\[
J\to \oplus_i J\beta^i\to(\oplus_{i<j} J\beta^i\beta^j)\oplus(\oplus_a J\gamma^a)\oplus (\oplus_bJ\rho^b)\to\cdots.
\]
The differential is obtained from the terms in $S$ linear in the generators,
underlined in the formula for $S$ above. We have 
\[
d_1(f)=-\sum_i\tau_i(f)\beta^i,
\quad
d_1(\beta^i)=-\sum_a r_{ai}\gamma^a+\frac12
f_{jk}^i\beta^j\beta^k.
\]
In $\delta(\rho^*_b)\rho^b$ there could be a linear term in $\beta^*_i$ that
would contribute to $d_1(\beta^i)$. But since $\delta(\rho^*_b)$ lies in
$R_0^{-2}$ its image in $E_1$ vanishes.

The Leibniz rule gives the differential of a general 1-cochain
\[
d_1(\sum_j g_j\beta^j)=-\sum_{i,j}\tau_i(g_j)\beta^i\beta^j+\frac12
\sum_{i,j,k}g_if_{jk}^i\beta^j\beta^k
-\sum_{a,i}r_{ai}g_i\gamma^a.
\]
Thus $H^1(M,S)\cong Z^1/B^1$ with
\begin{eqnarray*}
Z^1&=&\biggl\{\sum_i g_i\beta^i\in \oplus_i J\beta^i\,\text{ such that }
\\
&&\text{(a) } \tau_i(g_j)-\tau_j(g_i)-\sum_kf_{ij}^kg_k=0,\forall i,j,\quad
\text{(b) } \sum_k r_{ak}g_k=0,\forall a\biggr\},
\\
B^1&=&\biggl\{\sum_i\tau_i(f)\beta^i\,|\,f\in J\biggr\}.
\end{eqnarray*}
Cochains $\sum_ig_i\beta^i$ obeying (b) are in one-to-one correspondence with
linear functions $\alpha\in \mathrm{Hom}_{\mathcal O_X}(L/L_0,J)$ via $\alpha(\tau_i)=g_i$.
This space is canonically isomorphic to $\mathrm{Hom}_J(J\otimes_{\mathcal O_X}
L/L_0,J)=\Omega_{\mathrm{dR}}^1(J,\mathfrak g)$ and (a) translates to the
cocycle condition for the de Rham differential. Thus
\[
Z^{1}\cong \mathrm{Ker}(d\colon \Omega_{\mathrm{dR}}^1(J,\mathfrak g)\to\Omega_{\mathrm{dR}}^2(J,\mathfrak g)).
\]
Similarly $B^1\cong d(\Omega_{\mathrm{dR}}^0(J,\mathfrak g))$.
This concludes the proof of Theorem \ref{t-deRham}.
\subsection{The $P_0$-algebra structure}
As the filtration is compatible with the product, it follows that the
isomorphism of Theorem \ref{t-deRham} respects the product $H^p\otimes H^q\to
H^{p+q}$, $p+q\leq1$. The situation with the bracket is more tricky and
potentially more interesting.

Let, as above, $\mathfrak g=L(S_0)/L_0(S_0)$ and $f\in
H^0_{dR}(J,\mathfrak{g})=J(S_0)^{L(S_0)}$ be an invariant function on the
critical locus.  Let $\tilde f\in{\mathcal O_X}$ be a representative of $f$ in
${\mathcal O_X}$. Then for each $\tau\in L(S_0)$
\[
\tau(\tilde f)=\xi(S_0),
\]
for some $\xi=\xi(\tau,\tilde f)\in T$ defined modulo $L(S_0)$ depending
${\mathcal O_X}$-linearly on $\tau$ and $\K$-linearly on $\tilde f$.  This
defines a homomorphism of ${\mathcal O_X}$-modules
\begin{equation}\label{e-n}
L(S_0)\to T/L(S_0),\qquad \tau\mapsto \xi(\tau,\tilde f)\mod L(S_0).
\end{equation}
Recall that, by Lemma \ref{l-n2}, $L_0(S_0)$ acts trivially on $J$, so that
the map \eqref{e-n} descends to a map $\mathfrak g\to T/L(S_0)$.  We define a symmetric
$\K$-bilinear map $H^0(J,\mathfrak g)\times H^0(J,\mathfrak g)\to
H^1(J,\mathfrak g)$
\[
B(f,g)=\text{class of }\left(\tau\mapsto \xi(\tau,\tilde f)(\tilde
g)+\xi(\tau,\tilde g)(\tilde f)\right) \in \mathrm{Hom}_J(\mathfrak g,J),
\]
for any choice of lifts $\tilde f,\tilde g\in{\mathcal O_X}$ of $f,g\in J$.
\begin{lemma} 
  The bilinear map $B$ is well-defined.
\end{lemma}
\begin{proof}
Any two lifts $\tilde f$ of $f\in J$ differ by $\zeta(S_0)$ for some vector
field $\zeta\in T$. We need to show that $B(\zeta(S_0),\tilde g)=0$.  We have
$\tau(\zeta(S_0))=[\tau,\zeta](S_0)$ for $\tau\in L$. Thus
\[
\xi(\tau,\zeta(S_0))=[\tau,\zeta],
\]
and therefore
\[
B(\zeta(S_0),\tilde g)(\tau)=[\tau,\zeta](\tilde g)+\xi(\tilde g,\tau)(\zeta(S_0)).
\]
The second term is equal to $[\xi(\tilde g,\tau),\zeta](S_0)+\zeta(\tau(\tilde
g))$ by definition of $\xi$. The bracket evaluated on $S_0$ vanishes in $J$ and
we get 
\[
B(\zeta(S_0),\tilde g)(\tau)=\tau(\zeta(\tilde g)),
\]
which belongs to the zero cohomology class in $H^1(J,\mathfrak g)$.
\end{proof}
\begin{theorem}
The bilinear form $B$ coincides with the induced bracket 
\[
H^{0}(M,S)\times
H^{0}(M,S)\to H^1(M,S)
\] under the identification $H^p(M,S)\cong
H^p(J,\mathfrak g)$ of Theorem \ref{t-deRham}.  
\end{theorem}

\begin{proof}
To compute the induced bracket \[
[\;\,,\;]\colon H^0(M,S)\otimes H^0(M,S)\to
H^1(M,S),
\] using the description of Theorem \ref{t-deRham}, we need to lift
the cocycles of $E_1$ to cocycles of the BRST complex modulo $F^2$. So let
$\tilde f\in{\mathcal O_X}$ represent a cocycle $f$ in $E_1^{0,0}=J$. Then
$\hat\tau_i(\tilde f)=\xi_i(S_0)$ for some vector fields $\xi_i=\xi(\tilde
f,\hat\tau_i)$.
\begin{eqnarray*}
[S,f]&\equiv& [\sum_i\hat\tau_i\beta^i,f]\mod F^2
\\     &\equiv& -\sum_i\hat\tau_i(f)\beta^i \mod F^2
\\     &\equiv& -\sum_i\xi_i(S_0)\beta^i\mod F^2.
\\     &\equiv&[S,\sum_i\xi_i\beta^i]\mod F^2.
\end{eqnarray*}
It follows that a cocycle in the BRST complex corresponding to $f$ is
\[
F\equiv f-\sum_i\xi_i\beta^i \mod F^2.
\]
Let us do the same for a second function $g\in J^L$: $G\equiv g+\sum_i\eta_i\beta^i$.
Then
\begin{eqnarray*}
[F,G]&\equiv& [f-\sum_i\xi_i\beta^i,g-\sum_i\eta_i\beta^i]\mod F^2\\
     &\equiv& (\xi_i(g)+\eta_i(f))\beta^i \mod F^2
\end{eqnarray*}
The right-hand side is the image in $E^{1,0}_1$ of $[F,G]$ and coincides with the claimed
formula.
\end{proof}

\section{The case of quasi-projective varieties}\label{s-7}
Let $X$ be a nonsingular algebraic variety over a field $\K$ of
characteristic $0$, $S_0\in\mathcal O(X)/\K_X$ a function defined
modulo constants or more generally, in the setting of
Sect.~\ref{ss-mBV}, a multivalued function $S_0=\int\lambda$ for a
closed $1$-form $\lambda\in\Gamma(X,\Omega^1_X)$. Then on an open
affine subset $U\subset X$, the restriction of $S_0$ to $U$ gives rise
to a BV variety $(M=T^*[-1]V,S)$ with support $(U,S_0|_U)$ which is
unique up to stable equivalence. In particular the BRST complex
$(\mathcal O_{T^*[-1]V},d_S)$ is uniquely determined, up to
quasi-isomorphism of sheaves of differential $P_0$-algebras on $U$, by
$S_0$. Is it possible to glue these local data to form a sheaf of
differential $P_0$-algebras on $X$, whose restriction to any open
affine subset is quasi-isomorphism to the local BRST complex?  How
unique is the construction?

We cannot answer these questions in general, but we show existence in
the case of quasi-projective varieties $X$. In this case, every
coherent sheaf admits a resolution by locally free sheaves and, as a
consequence, a global Tate resolution of the sheaf $J(S_0)$ of
Jacobian rings exists. Let us fix such a Tate resolution.  Then for
any covering $(U_i)$ by affine open sets, the Tate resolution
restricts on every non-empty intersection $U=U_{i_1}\cap\cdots\cap
U_{i_k}$ to a Tate resolution of $J(S_0)|_{U}$ and we can consider the
class of BV varieties on $U$ associated with this Tate resolution (see
Sect.~\ref{s-existence}).  As the Tate resolution is defined globally,
the corresponding solutions of the master equation are sections
$S\in\Gamma(U,\mathcal O_{T^*[-1]V})$ for some globally defined graded
variety $V$ with support $X$.  We know by Theorem \ref{t-T1} that the
group of gauge equivalances acts transitively on the BV varieties with
support on an affine set and associated with a fixed Tate resolutions.
To glue, we need a stronger form of this result and show that the
corresponding action groupoid extends to a simplicial groupoid whose
sets of morphisms are contractible Kan complexes.  We first introduce
the necessary notions to formulate the result.

Let $U$ be nonsingular affine, $S_0\in\mathcal O(U)$ and $(R,\delta)$ be a Tate
resolution of the Jacobian ring of $S_0$. We realize $R$ as in Section \ref{ss-TateResolutions}
as $\mathcal O_{M}/I_{M}$ where $M=T^*[-1]V$ is the cotangent bundle of a 
$\mathbb Z_{\geq 0}$-graded variety $V$. Let
 $\mathcal B=\mathcal B(M,\delta)$ be the set of BV varieties $(T^*[-1]V,S)$ with support $(U,S_0)$ associated with 
the Tate resolution $(\mathcal O_{M}/I_{M},\delta)$, see Def.~\ref{d-asso}. 
Then by Theorem \ref{t-T1} the group of gauge equivalences $G(M)=\exp\mathrm{ad}\,\mathfrak g(M)\subset \mathrm{Aut}(\mathcal O_M)$, 
see Sect.~\ref{s-gauge}, fixes the Tate resolution and
acts transitively on $\mathcal B$. 
The Lie algebra $\mathfrak g(M)$ is the degree $0$ component of the graded
Lie algebra $\mathfrak g^\bullet(M)=\Gamma(U,I_M^{(2)})[-1]$. We will need the other components
as well.
 
Let $\mathcal G_0$ denote the action groupoid of the action of $G(M)$ on $\mathcal B$: its object set
is $\mathcal B$ and the set of morphisms $S_1\to S_2$ is
\[
\mathcal G_0(S_1,S_2)=\{g\in G(M)\,|\,g\cdot S_1=S_2\}.
\]
We will show that $\mathcal G_0(S_1,S_2)$ is the set of $0$-simplices of 
a contractible Kan
complex, namely a simplicial set $\mathcal G(S_1,S_2)$ such that every
simplicial sphere $\partial\Delta^n\to \mathcal G(S_1,S_2)$ can be filled to a
map of simplicial sets 
$\Delta^n\to \mathcal G(S_1,S_2)$. Here $\Delta^n$
is the standard combinatorial $n$-simplex. It is the simplicial set
$\mathrm{Hom}_\Delta(\;,[n])$ represented by the object $[n]=\{0,\dots,n\}$
of the simplicial category $\Delta$ of finite ordered sets and non-decreasing
maps.

Let $A$ be the $P_0$-algebra $A=\Gamma(U,\mathcal O_M)$.
Let $\Omega^\bullet(|\Delta^n|)$ be the de Rham algebra of polynomial
differential forms on the geometric $n$-simplex $|\Delta^n|$:
\[
\Omega^{\bullet}(|\Delta^n|)=
\K[t_0,\dots,t_n,dt_0,\dots,dt_n]/(\sum_{i=0}^n t_i-1,\sum_{i=0}^n dt_i),
\; \deg(t_i)=0,\deg(dt_i)=1.
\] 
Then the sequence ${\mathfrak g}_n^\bullet=\mathfrak g^\bullet(M)
\otimes\Omega^\bullet(|\Delta^n|)$ is a simplicial differential
graded Lie algebra. The Lie bracket is $[a\otimes\omega,b\otimes \eta]=
(-1)^{\deg\,b\,\deg\,\omega}[a,b]\otimes\omega\eta$. The simplicial algebra
$\Omega^\cdot(|\Delta^\bullet|)$ is the functor from $\Delta$ to differential
graded algebras sending $[n]$ to  $\Omega^\cdot(|\Delta^n)$ and 
$f\colon [n]\to [m]$ to the map of differential graded algebras
such that $t_i\mapsto \sum_{j\in f^{-1}(i)}t_j$.

Let $G_n=\exp(\mathrm{ad}(\mathfrak g_n^0))$. It is a subgroup of the
group
$\mathrm{Aut}_{\Omega^\bullet(|\Delta^n|)}(A\otimes\Omega^\bullet(|\Delta^n|))$
of $\Omega^\bullet(|\Delta^n|)$-linear automorphism of the graded
$P_0$-algebra $A\otimes\Omega^\bullet(|\Delta^n|)$. A solution $S\in
\mathcal B$ of the classical master equation defines a differential
$d_S=[S,\ ]$ on $A$. Let us use the same notation for the differential
on the product of complexes $A\otimes \Omega^\bullet(|\Delta^n|)$ given by
\[
d_S(a\otimes \omega)=[S,a]\otimes\omega+(-1)^{\deg\,a}a\otimes d\omega.
\]
We set
\[
\mathcal G_n(S_1,S_2)=\{g\in G_n\,|\,g \circ d_{S_1}=d_{S_2}\circ g\}.
\]
For $n=0$ this condition means that $[g\cdot S_1,a]=[S_2,a]$ for all $a\in A$.
Since the Poisson center consists of constants and $S_1$ and $S_2$ are both congruent to $S_0$
modulo $I_M$, the definition coincides with the previous definition of $\mathcal G_0$.
The collection $\mathcal G(S_1,S_2)=(\mathcal G_n(S_1,S_2))_{n=0,1,\dots}$ is then naturally a 
simplicial set.

\begin{prop}\label{p-contractible}
Let $S_1,S_2\in \mathcal B$.
The restriction map  
\[
\mathrm{Hom}_{\mathrm{SSet}}(\Delta^n,\mathcal G(S_1,S_2))
\to
\mathrm{Hom}_{\mathrm{SSet}}(\partial \Delta^n,\mathcal G(S_1,S_2))
\]
is surjective.
\end{prop}

\begin{proof}
Since $G_0$ acts transitively on $\mathcal B$, 
post-composing $g$ with a map in $G_0$ mapping $S_2$ to $S_1$ reduces the
problem to the case where $S_1=S_2=S$.

By the Yoneda lemma $\mathrm{Hom}(\Delta^n,\mathcal G(S,S))=\mathcal G_n(S,S)$. 

\begin{lemma}
  $\mathcal G_n(S,S)=\exp(\mathrm{ad}(\mathfrak g_n^0(S,S)))$ where $\mathfrak
  g_n^0(S,S)=\{\omega\in\mathfrak g_n^0\,|\, d\omega+[S,\omega]=0\}$.
\end{lemma}

\begin{proof}
  This is a special case of a general result: let $(L,d_L)$ be a
  differential nilpotent graded Lie algebra and $(M,d_M)$ be a
  differential graded faithful $L$-module on which $L$ acts locally
  nilpotently. This means that $M$ is an $L$-module such that
  \begin{enumerate}
  \item For all $a\in L, m\in M$, $d_M(a\cdot m)=d_L(a)\cdot m+a\cdot
    d_L(m)$
  \item If $a\in L$ and $a\cdot m=0$ for all $m\in M$ then $a=0$.
  \item For all $a\in L, m\in M$, there exists an $n$ such that
    $a^n\cdot m=0$.
  \end{enumerate}
  The claim is then:
  \[
  a\in L^0,\ \exp(a)\circ d_M=d_M\circ \exp(a) \Leftrightarrow
  d_L(a)=0.
  \]
  The proof relies on the identity
  \begin{gather*}
    (\exp(a)\circ d_M \circ\exp(-a)-d_M)m =b\cdot m,\quad m\in M,a\in A,\\
    b=\frac{\mathrm{id}-\exp(\mathrm{ad}(a))}{\mathrm{ad}(a)}(d_L(a))\in
    L.
  \end{gather*}
  Since $b$ is obtained by the action of an invertible operator on
  $d_L(a)$ it follows that $b=0$ iff $d_L(a)=0$.  To check the
  identity, introduce
  \[
  b_t=\exp(t\,a)\circ d_M\circ\exp(-t\,a)-d_M \in\mathrm{End}(M)[t].
  \]
  Differentiating with respect to $t$ yields
  \[
  \dot b_t =\exp(t\,a)\circ (a\circ d_M-d_M\circ
  a)\circ\exp(-t\,a)=-\exp(t\,a)\circ d_L(a)\circ\exp(-t\,a).
  \]
  Thus $\dot b_t$ is the action of $-\exp(t\,\mathrm{ad}(a))d_L(a)\in
  L$ and we conclude that
  \[
  b_{t=1}=-\int_0^1\exp(t\,\mathrm{ad}(a))d_L(a)dt,
  \]
  implying the identity after integration.

  In our case we have $L=\mathfrak g^\bullet_n$ and $M=A\otimes
  \Omega^\bullet(|\Delta^n|)$ with differential $d+d_S$.  The action
  is faithful because the part of the Poisson center in $\mathfrak g^\bullet(M)$
  is trivial: $Z_M\cap\mathfrak g_M^\bullet=0$, see Prop.~\ref{p-center}.  The Lie
  algebra and the action are pronilpotent: they become nilpotent
  modulo $I_M^{(k)}$ for all $k$. The statement of the Lemma is
  equivalent to the statement modulo $I_M^{(k)}$ for all $k$.
\end{proof}
Let us decompose $\omega$ with respect to the de Rham degree
\[
\omega=\omega_0+\omega_1+\cdots+\omega_n, \quad \omega_j\in A^{-1-j}\otimes\Omega^j(|\Delta^n|).
\]
Then the condition for $\exp(\mathrm{ad}(\omega))\in G_n$ to lie in
$\mathcal G_n(S,S)$ is
\begin{equation}\label{e-coc}
  \begin{split}
    [S,\omega_0]&=0,
    \\
    d\omega_0+[S,\omega_1]&=0,
    \\ 
    \cdots
    \\
    d\omega_{n-1}+[S,\omega_n]&=0.
  \end{split}
\end{equation}
We need to show that $\omega\in \mathfrak g^\bullet \otimes
\Omega^\bullet(|\partial\Delta^n|)$ of degree $0$ obeying
\eqref{e-coc} extends to a differential form on the simplex obeying
\eqref{e-coc}.

The basic fact about polynomial differential forms is that every
differential form on the boundary of a simplex extends to the simplex
(see \cite{Sullivan}). Thus the restriction homomorphism
\[
\Omega^\bullet(|\Delta^n|)\to\Omega^\bullet(|\partial\Delta^n|)
\]
is surjective. Let us choose a right inverse $s\colon\omega\mapsto
\tilde\omega$ and use the same notation to denote the right inverse
$\mathrm{id}\otimes s$ of the $A$-linear extension
\[
A\otimes\Omega^\bullet(|\Delta^n|)\to A\otimes\Omega^\bullet(|\partial\Delta^n|).
\]
Let now $\omega=\omega_0+\cdots+\omega_{n-1}\in
A\otimes\Omega^\bullet(|\partial\Delta^n|)$ obey \eqref{e-coc} and
$\tilde\omega_0,\dots,\tilde\omega_{n-1}$ $A$-linear extensions to
$|\Delta^n|$. Clearly $[S,\;]$ commutes with $\mathrm{id}\otimes s$
and thus
\[
[S,\tilde\omega_0]=0.
\]
Let us assume inductively that, possibly after modifying the choice of
extensions $\tilde\omega_j$, 
\[
d\tilde\omega_{j-1}+[S,\tilde\omega_j]=0,\text{\ for $j\leq k$.}
\]
In particular, for $j=k$ this implies
\[
0=d[S,\tilde\omega_k]=[S,d\tilde\omega_k]. 
\]
Thus $\alpha=d\tilde\omega_k+[S,\tilde\omega_{k+1}]$ obeys
$[S,\alpha]=0$, $\alpha|_{|\partial\Delta^n|}=0$, where we set
$\tilde\omega_{k+1}=0$ for $k=n-1$.  Thus $\alpha$ is a cocycle for
$[S,\;]$ in $A^{-k-1}\otimes J^k$ where $J$ is the ideal of
differential forms on $|\Delta^n|$ whose restriction to the boundary
vanishes. Since the BRST cohomology vanishes in negative degree, there
exists a $\beta$ vanishing on the boundary such that
$\alpha=[S,\beta]$. Replacing $\tilde\omega_k$ by
$\tilde\omega_k-\beta$ yields
\[
d\tilde\omega_k+[S,\tilde\omega_{k+1}]=0,
\]
which proves the induction step.
\end{proof}

Thus $\mathcal G_n(S,S')$ is the set of $n$-simplices of a contractible
Kan complex. Moreover, we have composition maps
\[
\mathcal G_n(S,S')\times \mathcal G_n(S',S'')\to \mathcal G_n(S,S''),
\]
induced by the product in $A$ and the wedge product of differential
forms.  The compositions are maps of simplicial sets if we define the
Cartesian product $C\times D$ of simplicial sets $C,D$ as the sequence
of sets $C_n\times D_n$ with face maps
$\partial_i(x,y)=(\partial_ix,\partial_iy)$ and degeneracy maps
$s_i(x,y)=(s_ix,s_iy)$.

Thus there is a {\em simplicial category} $\mathcal G$, namely a
category enriched over simplicial sets, with objects $\mathcal B$
and morphisms $S\to S'$ given by the contractible Kan complex 
$\mathcal G(S,S')$.

\begin{cor}\label{c-quasipro}
Let $X$ be a nonsingular 
quasi-projective variety, $\lambda\in\Omega^1(X)$ a closed $1$-form,
$S_0=\int \lambda$. Then there exists a sheaf $\mathcal P$ of differential $P_0$-algebras
on $X$ which is locally the BRST complex of a local BV variety with support $(X,S_0)$.
More precisely, every point of $X$ has an open affine neighborhood $U$
such that $\mathcal P|_U$ is quasi-isomorphic, as a sheaf of differential
$P_0$-algebras, to $(\mathcal O_{M},d_S)$ for some BV variety $(M,S)$ with
support $(U,S_0|_U)$.
\end{cor} 

The sheaf $\mathcal P$ is constructed by gluing the BV varieties associated with
the restriction of a Tate resolution to the affine open subsets of a covering by using
the contractibility result of Prop.~\ref{p-contractible}. This construction
is described in Appendix \ref{appB}.

\appendix

\section{Stable isomorphism of Tate resolutions}\label{a-2}
Let $A$ be a Noetherian unital commutative ring containing $\mathbb Q$ and
$C_A$ be the monoidal category of non-positively graded differential graded
unital commutative algebras $C=\oplus_{i\leq0}C^i$ over $A$ whose homogeneous
components $C^i$ are finitely generated $A$-modules. Differential have degree
1. An object of $C_A$ is called semi-free if it is isomorphic, as a graded
commutative algebra, to the symmetric algebra $\Sym_A(V)$ of some projective
negatively graded $A$-module $V=\oplus_{i<0} V^i$.

Let $C$ be an object of $C_A$ concentrated in degree $0$. A {\em Tate
resolution} of $C$ is a quasi-isomorphism $E\to C$ in $C_A$ where $E=\Sym_A(V)$
is semi-free. Tate resolutions exist by Tate's recursive construction: the
subalgebra $E_{\geq-d}$ of $E$ generated by $\oplus_{i=-d}^0V^i$ is constructed
out of $E_{\geq-(d-1)}$ as
\begin{equation}\label{equation-1}
  E_{\geq-d}=E_{\geq-(d-1)}[T_1,\dots,T_n],\qquad \deg(T_i)=-d,
\end{equation}
with differential extending the differential on $E_{\geq -(d-1)}$ and such that
\[
d\,T_i=c_i,
\]
for some set of cocycles $c_i$ spanning the cohomology of $E_{\geq -d}$ in
degree ${-d}$ (the Noetherian hypothesis ensures that the cohomology of fixed
degree is a finitely generated $A$-module). One then takes $V^{-d}$ at the free
module with basis $T_1,\dots,T_n$.

A morphism of Tate resolutions $p_E\colon E\to C, p_F\colon F\to C$ is a morphism of
differential graded algebras $f\colon E\to F$ such that $p_F\circ f=p_E$.

Any two Tate resolutions are related by a quasi-isomorphism which is unique up
to homotopy. More precisely:
\begin{lemma}\label{lemma-1}
Let $p_E\colon E\to C$, $p_F\colon F\to C$ be two Tate resolutions.
\begin{enumerate}
\item[(i)] 
Any morphism $f\colon E_{\geq-d}\to F_{\geq-d}$ of differential graded Lie
algebras such that $p_F\circ f=p_E$ extends to a morphism $E\to F$ of Tate
resolutions.
\item[(ii)] 
Any two morphisms $f_0,f_1\colon E\to F$ of Tate resolutions are homotopic,
Namely, there is a morphism $h\colon E\to F\otimes_A A[t,dt]$ with
$\mathrm{id}\otimes\epsilon_j\circ h=f_j$, $j\in\{0,1\}$. Here
$\epsilon_j\colon A[t,dt]\to A$ is the map $t\mapsto j$, $dt\mapsto 0$.
\item[(iii)] 
Every morphism $f\colon E\to F$ of Tate resolution admits an inverse up to
homotopy: there is a morphism $g\colon F\to E$ of Tate resolutions so that
$f\circ g$ and $g\circ f$ are homotopic to the identity. In particular all
morphisms of Tate resolutions of $C$ are quasi-isomorphisms.
\end{enumerate}
\end{lemma}
The proof is standard: the required maps (extension of $f$, homotopy, inverse
up to homotopy) on free algebras are uniquely determined by the images of
generators, that are to obey cohomological conditions. The existence of images
obeying the conditions follows recursively by degree by the acyclicity of the
resolution in the previous degrees.

The next result is a stronger statement relating Tate resolutions of an object
by an isomorphism (rather than a quasi-isomorphism) after tensoring with a
semi-free acyclic algebra.

Let $\Sym_A(V\oplus V[1])$ the acyclic symmetric algebra on a free graded
$A$-module $V$ with differential defined on generators $V\oplus V[1]=\oplus
(V^i\oplus V^{i+1})$ as $d(v\oplus w)=w\oplus 0$. If $E\to C$ is a Tate
resolution, then also $E\otimes_A \Sym_A(V\oplus V[1])\to C$, sending $V\oplus
V[1]$ to zero,  is a Tate resolution and the obvious map $E\to E\otimes_A
\Sym_A(V\oplus V[1])$ is a quasi-isomorphism of Tate resolutions.

\begin{theorem}\label{t-111} 
Let $d\geq 0$, $C\in C_A$.
Let $f\colon E\to F$ be a morphism of Tate resolutions of $C$ which is an
isomorphism in degrees $\geq -d+1$. Then there is an isomorphism in
$C_A$
\[
E\otimes_A \Sym_A(V\oplus V[1])\cong F\otimes_A \Sym_A(W\oplus W[1]),
\]
restricting to $f$ in degrees $\geq -d+1$,
for some free graded $A$-modules $V,W$ concentrated in degree $\leq -d$. 
\end{theorem}

\begin{proof}
We prove inductively the following statement.

\begin{lemma}\label{lemma-2}
Let $f_{d-1}\colon E\to F$ be a morphism of Tate resolutions of $C$, which
restricts to an isomorphism $E_{\geq -(d-1)}\to F_{\geq -(d-1)}$ for some
$d>0$. Then there are graded free $A$-modules $V_d,W_d$ of finite rank
concentrated in degree $-d$ and a morphism of Tate resolutions
\[
f_d\colon
E\otimes_A \Sym_A(V_d\oplus V_d[1])\to F\otimes_A \Sym_A(W_d\oplus W_d[1]),
\]
restricting to an isomorphism in degree $\geq-d$ and coinciding with $f_{d-1}$
on $E_{\geq -(d-1)}$.
\end{lemma}

By using this Lemma and starting from any morphism $f_0\colon E\to F$ of Tate
resolutions, we obtain a morphism $f$ as in the claim of the theorem with
$V=\oplus V_d$ and $W=\oplus W_d$.

To prove the Lemma, we first of all notice that if $E=\Sym_A(W)$ we can
assume without loss of generality that the projective $A$-module $W^{-d}$ is 
free: if it is not, we replace $W$ by $W\oplus U[d]\oplus U[d+1]$, where
$W^{-d}\oplus U$ is free and the differential is extended so that $\delta\colon
U[d+1]\to U[d][1]$ is the identity. Similarly we can assume that the
generators of degree $-d$ of $F$ form a free $A$-module. 

Let $T_1,\dots,T_n$ be a basis of the space of generators of
degree $-d$ of $E$ as in \eqref{equation-1} and $S_1,\dots,S_m$ be a basis of
the space generators of degree $-d$ of $F$. We then take $V_d$ to be a vector
space of dimension $m$ with basis $S_1',\dots, S_m'$. Then
\[
\Sym_A(V_d\oplus V_d[1]) = A[S_1',\dots,S_m',S_1'',\dots,S_m''],
\quad \deg(S_j'')=-d-1,\quad d\,S_j''=S_j'.
\]
Similarly,
\[
\Sym_A(W_d\oplus W_d[1]) = A[T_1',\dots,T_n',T_1'',\dots,T_n''],
\quad\deg(T_i'')=-d-1,\quad d\,T_i''=T_i'.
\]
Let $g_{d-1}\colon F\to E$ be a morphism extending the inverse of the
restriction of $f_{d-1}$ to $E_{\geq-(d-1)}$.
Such an extension exists by Lemma \ref{lemma-1} (i).
Let $f_d$ coincide with $f_{d-1}$ on $E_{\geq-(d-1)}$ and similarly for $g_d$. We set
\begin{equation}\label{equation-2}
f_d(T_i)=f_{d-1}(T_i)+T_i',\qquad g_d(S_j)=g_{d-1}(S_j)+S_j'
\end{equation}
This defines $f_d$ as a morphism of graded algebras $E_{\geq-d}\to F_{\geq-d}$.
Since $d\,T_i'=0$, $f_d$ commutes with differentials on $E_{\geq-d}$. The same holds for $g_d$.
The next definition is devised to make $g_d$ the inverse of $f_d$ in degree $-d$:
\[
f_d(S_j')=S_j-f_d(g_{d-1}(S_j)),\qquad g_d(T_i')=T_i-g_d(f_{d-1}(T_i)).
\]
The right-hand sides are already defined since $g_{d-1}(S_j)\in E_{\geq-d}$ and $f_{d-1}(T_i)\in F_{\geq-d}$. Also,
since $f_d$ coincides with $f_{d-1}=g_{d-1}^{-1}$ in degree $\geq -d+1$, $f_d$, and similarly $g_d$,
commutes with the differential:
\[
d( f_d(S_j'))=dS_j-f_d(g_{d-1}(dS_j))=dS_j-f_{d-1}(g_{d-1}(dS_j))=0=f_d(dS_j').
\]
The inversion property is first checked on $T_i$, $S_j$.
\begin{eqnarray*}
g_d\circ f_d(T_i)&=&g_d(f_{d-1}(T_i))+g_d(T_i')\\
&=&
g_d(f_{d-1}(T_i))+
T_i-g_d(f_{d-1}(T_i))\\
&=&
T_i.
\end{eqnarray*}
This and the induction hypothesis proves that $g_d\circ f_d$ is the identity on
$E_{\geq-d}$.  Similarly, $f_d\circ g_d$ is the identity on $F_{\geq-d}$.  We
use these facts for the next check: 
\begin{eqnarray*}
    g_d\circ f_d(S_j')&=&g_d(S_j)-g_d\circ f_d(g_{d-1}(S_j))\\
  &=&
    S_j'+g_{d-1}(S_j)-g_{d-1}(S_j)\\
  &=&
    S_j'.
\end{eqnarray*}
Similarly, $f_d(g_d(T_i'))=T_i'$.

We have constructed an isomorphism of differential graded algebras
\[
f_d\colon
(E\otimes_A \Sym_A(V_{d}\oplus V_d[1]))_{\geq-d}\to (F\otimes_A \Sym_A(W_d\oplus W_d[1]))_{\geq-d},
\]
with inverse $g_d$. By Lemma \ref{lemma-1}, $f_d$ admits an extension to a
morphism of Tate resolution $E\otimes_A \Sym_A(V_{d}\oplus V_d[1])\to
F\otimes_A \Sym_A(W_d\oplus W_d[1])$.  Any such extension (in particular
defining $f_d(T_i'')$) has the properties required in the claim of the Lemma.
\end{proof}

\section{Gluing sheaves of differential graded algebra}\label{appB}
\centerline{by Tomer M. Schlank}

\bigskip

The aim of this Appendix is to describe how to  glue sheaves of differential $P_0$-algebras.

More precisely, let $X$ be a non-singular quasi-projective variety. 
For an open subvariety $U \subset X$ we denote by $\mathfrak{A}(U)$ the category of sheaves  of differential $P_0$-algebras on $U$.
Let  $U_0,\dots,U_n$ be a finite affine cover of $X$. Further let $A_i \in \mathfrak{A}(U_i) $ be a sheaf of differential $P_0$-algebras on $U_i$, our goal is to glue all the $A_i$'s to a sheaf of differential $P_0$-algebras $\mathbf{A} \in \mathfrak{A}(X)$  such that $\mathbf{A}|_{U_i}$ is quasi-isomorphic $A_i$. To better understand the issues at hand consider first the case where we are trying to glue sheaves given on two open subset $U_0, U_1$. In this case our gluing data consist of object $A_i \in \mathfrak{A}(U_i), \! 0 \leq i \leq 1$, and a quasi-isomorphism $K_{0,1} :A_0|_{U_0 \cap U_1} \to A_1|_{U_0 \cap U_1}$. Here we already see the first difference between classical gluing of sheaves and gluing of sheaves in our case, since we only assume that $K_{0,1}$ is a quasi-isomorphism and not necessarily an isomorphism as in the classical case. The difference from the classical case becomes even more apparent when we consider 3 open sets  $U_0, U_1, U_2.$
In this case in addition to  $A_i \in \mathfrak{A}(U_i), \! 0 \leq i \leq 2$ and the quasi-isomorphism $K_{0,1},K_{0,2},K_{1,2}$ our gluing data will consist of a homotopy $K_{0,1,2}$ between $K_{1,2}\circ K_{0,1}$ and $K_{0,2}$. Note that when gluing classical sheaves one will assume that the two maps $K_{1,2}\circ K_{0,1}$ and $K_{0,2}$ are the same.

To conclude we have several steps to go through so we can glue our sheaves of  differential $P_0$-algebras. The first step will be to define gluing data for  sheaves of differential $P_0$-algebras. The second step will be to show that given such gluing data we indeed can construct the desired $\mathbf{A}$. The last step will be to show that in the situation described in this paper we will indeed have such gluing data.

The datum and argument below are very categorical in nature. So we first have to take good care of the notations.

\subsection{Notation}
Throughout this appendix we fix a non-singular quasi-pro\-ject\-ive variety $X$.
Let $\mathcal{S}$ be the category of simplicial sets.

For an open subvariety $U \subset X$ we denote by $\mathfrak{A}(U)$ the category of sheaves of differential $P_0$-algebras on $U$. The de Rham algebras $\Omega_n^\bullet
=\Omega^\bullet(|\Delta^n|)$ of polynomial differential forms on 
the geometric $n$-simplices form a simplicial differential graded commutative
algebra (DGCA) $\Omega^\bullet$. Let us
denote by the same letter the functor 
\[
\Omega^\bullet\colon \mathcal{S} \to \mathrm{DCGA}, \qquad
K\mapsto \mathcal S(K,\Omega^\bullet)
\]
from the category of simplicial sets to the category of DCGA (so that 
$\Omega^\bullet(\Delta^n)=\Omega^\bullet_n$).   
For a simplex $\Delta^n \in \mathcal{S},A \in \mathfrak{A}(U)$ we denote
$$A^{\Delta^n} := A\otimes \Omega^{\bullet}(\Delta^n)\in \mathfrak{A}(U).$$
Similarly for every simplicial set $K \in \mathcal{S}$ we can define
$$A^{K} := A\otimes \Omega^{\bullet}(K)\in \mathfrak{A}(U).$$
For every $A_0,A_1 \in \mathfrak{A}(U)$ we define the simplicial set
$$\mathfrak{M}_{U}(A_0,A_1) \in \mathcal{S},$$
$$\mathfrak{M}_{U}(A_0,A_1)_n := \mathrm{Hom}_{\mathfrak{A}(U)}(A_0,A_1^{\Delta^n}).$$
Note that we can view this also as 
$\mathrm{Hom}_{\mathfrak{A}(U)\otimes\Omega^\bullet(\Delta^n)}(A_0^{\Delta^n},A_1^{\Delta^n}),$ 
so that compositions of maps in $\mathfrak{M}_U$ are
defined.
Given two open subsets $U \subset V \subset X$,
We denote by
$$\dow^V_U : \mathfrak{A}(V) \to \mathfrak{A}(U)$$
$$ A \mapsto A\dow^V_U,$$
the restriction functor, $\dow^V_U$ has a right adjoint which is the pushforward functor, we denote this functor by
$$\upa_U^V : \mathfrak{A}(U) \to \mathfrak{A}(V)$$
$$ A \mapsto A\upa_U^V.$$

\begin{lemma}\label{l:updown}
The functors $\dow^V_U$ and $\upa_U^V$, have the following properties:
\begin{enumerate}
\item $\dow^V_U$ is the left adjoint of $\upa_U^V$.
\item Let $U \subset V \subset W\subset X$; we have $$\dow^W_V\dow^V_U= \dow^W_U,$$ $$ \upa_U^V\upa_V^W= \upa_U^W. $$
\item Let $U,V \subset W$; we have $$ \upa_{U}^W \dow^{W}_V = \dow^U_{U\cup W} \upa_{U \cup W}^V.$$
\end{enumerate}

\end{lemma}

For $n \in \mathbb{N}$ we denote by $[n]$ the ordered set $0 < 1 <\cdots < n$. For any $ I \subset [n] $ we denote
$P_I$ to be the partially ordered set of subsets of $I$ containing the first and last element.
Now  let $$\{a_0,\dots,a_d\}  = I \subset [n],$$
and let $0<i<d$, We denote:
\begin{align*}
I_{\leq i} &:= \{a_0,\dots,a_i\}, \quad I_{\geq i} := \{a_i,\dots,a_d\},\quad
I_{\hat{i}} := \{a_0,\dots,\hat{a_i},\dots,a_d\},\\
&a_0 = n(I),\qquad
a_d = x(I).
\end{align*}
We introduce the following maps:
 $$T_{I,i}:P_{I_{\leq i}} \times P_{I\geq i} \to P_{I}, \qquad T_{I,i}(K,L) = K \cup L,$$
$$S_{I,i}:P_{I_{\hat{i}}}  \to P_{I},\qquad S_{I,i}(K) = K,$$
$$\hat{S}:P_{I_{\hat{i}}}  \to P_{I},\qquad \hat{S}_{I,i}(K) = K \cup \{i\}.$$
Let $\mathfrak{U} = \{U_0, U_1,\dots,U_n\}$  be a cover of $X$. For every non-empty $I \subset [n]$ We denote
$$U_I:= \bigcap_{i \in I}  U_i.$$
For any partially ordered set $P$ we denote by $N(P) \in \mathcal{S}$ the nerve of $P$.
Note that  for $I = \{a_0,\dots,a_d\} \subset [n], \quad d\geq 1,$ $N(P_I)$ is naturally isomorphic as a simplicial set to the $d-1$-dimensional cube $(\Delta^1)^{d-1}$.
We denote by $\partial N(P_I) \subset N(P_I)$ the boundary of the cube.
Note that $\partial N(P_I)$ is the union of the $2(d-1)$ faces of $N(P_I)$, each of which is a $d-2$-dimensional cube.
It easy to see that each of these faces is one of the images of the maps
$$N(S_{I,i}):N(P_{I_{\hat{i}}})  \to N(P_{I})$$
$$N(\hat{S}_{I,i}):N(P_{I_{\hat{i}}})  \to N(P_{I}),$$
for $0<i<d$.

\subsection{Gluing data}
In this section we shall define the gluing data required in order to glue sheaves of differential $P_0$-algebras.
Let $\mathfrak{U} := \{U_0, U_1,\dots,U_n\}$ be a finite open cover of $X$.
Let $(A_0,\dots,A_n;\{K_I\})$ consist of:
\begin{enumerate}
\item A sheaf of differential $P_0$-algebras, $A_i \in \mathfrak{A}(U_i)$ for every $i = 0,\dots,n$.
\item For every $I \subset [n]$ , $|I| \geq 2$ a map
$$K_I \in {\mathcal{S}}(N(P_I), \mathfrak{M}_{U_I}(A_{n(I)}\dow^{U_{n(I)}}_{U_I},A_{x(I)}\dow^{U_{x(I)}}_{U_I})).$$
\end{enumerate}
Then $(A_0,\dots,A_n;\{K_I\})$ are called \emph{gluing data on} $\mathfrak U$ if they satisfy the following conditions
\begin{enumerate}
\item For every $I \subset [n]$, $|I| = 2$, $K_I$ is a quasi-isomorphism.
\item For every $I = \{a_0,\dots,a_d\}$ and $0<i<d$ the diagram:
$$\xymatrix{
N(P_{I_{\leq i}}) \times N(P_{I_{\geq i}})  \ar[r]^{N(T_{I,i})} \ar[d]^-{K_{I_{\leq i}} \times K_{I_{\geq i}}} &
N(P_I) \ar[d]^{K_I} \\
\mathfrak{M}_{U_{I_{\leq i}}}
\times \mathfrak{M}_{U_{I_{\geq i}}}
\ar[r]^-{C_{I,i}}   & \mathfrak{M}_{U_{I}}(A_{a_0}\dow^{U_{a_0}}_{U_I},A_{a_d}\dow^{U_{a_d}}_{U_I} )
}$$
with $$\mathfrak{M}_{U_{I_{\leq i}}}=\mathfrak{M}_{U_{I_{\leq i}}}(A_{a_0}\dow^{U_{a_0}}_{U_{I_{\leq i}}},A_{a_i}\dow^{U_{a_i}}_{U_{I_{\leq i}}}),
$$ $$
\mathfrak{M}_{U_{I_{\geq i}}}=\mathfrak{M}_{U_{I_{\geq i}}}(A_{a_i}\dow^{U_{a_i}}_{U_{I_{\geq i}}},A_{a_d}\dow^{U_{a_d}}_{U_{I_{\geq i}}} ),$$
commutes.
Here $C_{I,i} = \bigcirc \circ (\dow^{U_{I_{\leq i}}}_{U_I} \times \dow^{U_{I_{\geq i}}}_{U_I})$ and $\bigcirc$ is the composition morphism.

\item For every $I = \{a_0,\dots,a_d\}$ and $0<i<d$ the diagram:
$$\xymatrix{
N(P_{I_{\hat i}}) \ar[r]^{N(S_{I,i})} \ar[d]^-{K_{I_{\hat i}}} &
N(P_I) \ar[d]^{K_I} \\
\mathfrak{M}_{U_{I_{\hat i}}}(A_{a_0}\dow^{U_{a_0}}_{U_{I_{\hat i}}},A_{a_d}\dow^{U_{a_d}}_{U_{I_{\hat i}}} )\ar[r]^{\dow^{U_{I_{ \hat i}}}_{U_I}} &
\mathfrak{M}_{U_{I}}(A_{a_0}\dow^{U_{a_0}}_{U_I},A_{a_d}\dow^{U_{a_d}}_{U_I} )
}$$
commutes.
\end{enumerate}
\begin{remark}
For $I= \{i\} \subset [n]$ it will be convenient to denote:

$$K_I = \mathrm{Id}_{A_i} \in \mathfrak{M}_{U_{i}}(A_{i},A_{i}). $$  
\end{remark}
\begin{remark}
If the $K_I$ are given only for  $I \subset [n]$, $2 \leq |I| \leq d$ for some integer $d$,
we shall call $(A_0,\dots,A_n;\{K_I\})$ \emph{$d$-partial gluing data}.
Note that this definition makes sense since the compatibility conditions for $K_I$ involve only $K_J$ with
$|J| < |I|$.
\end{remark}

We shall define a \emph{morphism} between two gluing data on $\mathfrak{U}$, $(A_0,\dots,A_n;K_I)$, $(B_0,\cdots,B_n;L_I)$
as a collection of morphism
$$ f_i \in \mathfrak{M}_{U_i}(A_i,B_i),$$
such that the following squares commute.
$$
\xymatrix{
A_{n(I)}\dow^{U_{n(I)}}_{U_I} \ar[d]^{f_n(I)\dow^{U_{n(I)}}_{U_I}} \ar[r]^{K_I} &
A^{N(P_I)}_{x(I)}\dow^{U_{x(I)}}_{U_I}\ar[d]^{f_x(I)\dow^{U_{x(I)}}_{U_I}} \\
B_{n(I)}\dow^{U_{n(I)}}_{U_I} \ar[r]^{L_I} & B^{N(P_I)}_{x(I)}\dow^{U_{x(I)}}_{U_I}
}
$$

For $(A_0,\dots,A_n)$ given  sheaves of differential $P_0$-algebras, $A_i \in \mathfrak{A}(U_i)$ we say that $(\{K_I\})$ are \emph{gluing data} for $(A_0,\dots,A_n)$ iff $(A_0,\dots,A_n;\{K_I\})$ are gluing data on $\mathfrak{U} := \{U_0, U_1,\dots,U_n\}$.

\subsection{The space of gluing data}
Let $(A_0,\dots,A_n)$ be sheaves of differential $P_0$-algebras, $A_i \in \mathfrak{A}(U_i)$.
The set of all possible gluing data  $(\{K_I\})$ for  $(A_0,\dots,A_n)$ can be considered as the zero simplices of a naturally defined simplicial set.
\begin{definition}
Let  $(A_0,\dots,A_n)$ be sheaves of differential $P_0$-algebras, $A_i \in \mathfrak{A}(U_i)$.
We define the \emph{space of gluing data for $(A_0,\dots,A_n)$ } to be the simplicial set $\mathbb{G}(A_0,\dots,A_n)$
such that $\mathbb{G}(A_0,\dots,A_n)_m $ consists of all collections of maps
$$ K_I \in {\mathcal{S}}(\Delta^m \times N(P_I), \mathfrak{M}_{U_I}(A_{n(I)}\dow^{U_{n(I)}}_{U_I},A_{x(I)}\dow^{U_{x(I)}}_{U_I})),$$
satisfying compatibly conditions as above.
\end{definition}

Given a finite cover of $X$,  $\mathfrak{U} := \{U_0, U_1,\dots,U_n\}$  we can take the disjoint union of  $\mathbb{G}(A_0,\dots,A_n)$ over all the possible   $(A_0,\dots,A_n)$ , $A_i \in \mathfrak{A}(U_i)$. We denote the resulting space by $\mathbb{G}(\mathfrak{U})$.

Similarly the collection of all morphism between gluing data on a cover $\mathfrak{U}$  can be considered as the zero simplices of a naturally defined simplicial set.

Namely let $(A_0,\dots,A_n),(B_0,\dots,B_n)$  be sheaves of differential $P_0$-algebras, $A_i$, $B_i \in \mathfrak{A}(U_i)$.
We shall denote by $\mathbb{H}((A_0,\dots,A_n),(B_0,\dots,B_n))$ the simplicial set of morphisms between gluing data for  $(A_0,\dots,A_n)$, $(B_0,\dots,B_n)$. Namely the $m$-simplices of $\mathbb{H}((A_0,\dots,A_n)$, $(B_0,\dots,B_n))$ are  collections of maps:
$$ K_I \in {\mathcal{S}}(\Delta^m \times N(P_I), \mathfrak{M}_{U_I}(A_{n(I)}\dow^{U_{n(I)}}_{U_I},A_{x(I)}\dow^{U_{x(I)}}_{U_I})),$$
$$ L_I \in {\mathcal{S}}(\Delta^m \times N(P_I), \mathfrak{M}_{U_I}(B_{n(I)}\dow^{U_{n(I)}}_{U_I},B_{x(I)}\dow^{U_{x(I)}}_{U_I})),$$
$$
f_i \in {\mathcal{S}}(\Delta^m , \mathfrak{M}_{U_i}(A_{i},B_{i})),$$
satisfying the natural computability conditions analogous to those specified in the previous subsection.
Again one takes the disjoint union of all possible such $(A_0,\dots,A_n)$ and $(B_0,\dots,B_n)$ and get a simplicial set $\mathbb{H}(\mathfrak{U})$.
Thus we get that the collection of gluing data on $\mathfrak{U}$ can be best described as a category object $(\mathbb{G}(\mathfrak{U}),\mathbb{H}(\mathfrak{U}))$ in the category of simplicial sets. A category object in a category is by definition a pair of objects $O$ and $M$ with two morphisms
$t,d\colon M\to O$ (target, source), a morphism $O\to M$ (idenity) and a
composition morphism $c\colon M\times_OM\to M$ satisfying compatibility conditions of a category.
We take $\mathbb{N}(\mathfrak{U})$ to be the nerve bi-simplicial set corresponding to this category object.
$$\mathbb{N}(\mathfrak{U})_m :=  \mathbb{H}(\mathfrak{U})\times_{\mathbb{G}(\mathfrak{U})}\cdots\times_{\mathbb{G}(\mathfrak{U})} \mathbb{H}(\mathfrak{U}), $$
where the product is taken over $m$ copies of $\mathbb{H}(\mathfrak{U})$  in an analogous fashion to the classical nerve construction.
The diagonal of the bi-simplicial set $\mathbb{N}(\mathfrak{U})$ is 
a simplicial set. We denote this simplicial set by $\mathbb{N}^{\Delta}(\mathfrak{U})$

\subsection{Gluing}
We shall prove the following statement.
\begin{theorem}
Let $\mathfrak{U} = \{U_0, U_1,\dots,U_n\}$ be a finite open cover of $X$
and
let $(A_0,\dots,A_n;\{K_I\}_I) $  be gluing data on  $\mathfrak{U}$. 
Then there exist a sheaf of differential $P_0$-algebra $\mathbf{A} \in \mathfrak{A}(X)$, such that for all $0\leq i \leq n$, $\mathbf{A} \dow^{X}_{U_i}$ is quasi-isomorphic to $A_i$.
\end{theorem}
\begin{proof}
The proof will be done by induction on the number of open sets in the cover,
i.e., we shall first construct the gluing for $n=1$, then we shall show that given gluing data on $\mathfrak{U} = \{U_0, U_1,\dots,U_n\}$, one can
define gluing data on $\mathfrak{V} = \{U_0, U_1,\dots,U_{n-2},U_{n-1} \cup U_{n} \}$, by gluing $A_{n-1}$ and $A_{n}$.
\begin{definition}

Let $\mathfrak{U} = \{U_0, U_1\}$ be a cover of $X$ (i.e. $X = U_0 \cup U_1$)
and let
$$A_0 \in \mathfrak{A}(U_0),$$
$$A_1\in \mathfrak{A}(U_1),$$
$$K_{0,1}\in \mathrm{Hom}_{\mathfrak{A}(U_{0,1})}(A_0\dow^{U_0}_{U_{0,1}}, A_1\dow^{U_1}_{U_{0,1}}),$$
be gluing data on $\mathfrak{U}$.
We denote by
$A_0\coprod_{K_{0,1}}A_1$
the limit  of the following diagram in $\mathfrak{U}(X)$:
$$\xymatrix{
A_{0}\upa_{U_0}^{X} \ar[dd]_{\upa_{U_{0,1}}^{X}\circ K_{0,1}\circ \dow^{U_0}_{U_{0,1}}} &  A_{1}^{\Delta^1}\dow^{U_1}_{U_{0,1}}\upa_{U_{0,1}}^{X}\ar[ddl]^{\upa_{U_{0,1}}^{X} \circ \beta_0}\ar[ddr]_{\upa_{U_{0,1}}^{X}\circ\beta_1} &
A_1 \ar[dd]^{\upa_{U_{0,1}}^{X}\circ\dow^{U_1}_{U_{0,1}}} \upa_{U_1}^{X} \\
\\
A_1 \dow^{U^1}_{U_{0,1}}\upa_{U_{0,1}}^{X} & \empty & A_1 \dow^{U^1}_{U_{0,1}}\upa_{U_{0,1}}^{X}
}$$
Where $\beta_0,\beta_1$ are the maps corresponding to the two maps $\Delta^0 \to \Delta^1$

\end{definition}
\begin{lemma}
Given data as above, $(A_0\coprod_{K_{0,1}}A_1) \dow ^{X}_{U_i}$ is naturally quasi-iso\-mor\-phic to  $A_i$.
\end{lemma}
\begin{proof}
First we prove this for $i=0$: since the restriction functor  commutes with limits, we have that $(A_0\coprod_{K_{0,1}}A_1)\dow^{X}_{U_0}$
is the limit of the diagram:
$$\xymatrix{
A_{0}\ar[dd]  &  A_{1}^{\Delta^1}\dow^{U_1}_{U_{0,1}}\upa_{U_{0,1}}^{U_0} \ar[ddl]\ar[ddr]&
A_1 \dow^{U_1}_{U_{0,1}}\upa_{U_{0,1}}^{U_0} \ar[dd]^{\cong}\\
\\
A_1 \dow^{U^1}_{U_{0,1}}\upa_{U_{0,1}}^{U_0} & \empty & A_1 \dow^{U^1}_{U_{0,1}}\upa_{U_{0,1}}^{U_0}
}$$
Thus we have  a pullback diagram
$$\xymatrix{
(A_0\coprod_{K_{0,1}}A_1 )\dow^{X}_{U_0} \ar[r]\ar[d]& A_{1}^{\Delta^1}\dow^{U_1}_{U_{0,1}}\upa_{U_{0,1}}^{U_0} \ar[d]\\
A_0 \ar[r] & A_{1}\dow^{U_1}_{U_{0,1}}\upa_{U_{0,1}}^{U_0}
}$$
Now we need to show that  left vertical map is a quasi-isomorphism. Note that this can be checked stalk-wise. Now the right vertical map is stalk-wise surjective and quasi-isomorphism (i.e., with acyclic kernel) and thus so is the left vertical map.

For $i=1$:
since the restriction functor  commutes with limits we have that $(A_0\coprod_{K_{0,1}}A_1)\dow^{X}_{U_1}$
is the limit of the diagram:
$$\xymatrix{
A_{0}\dow^{U_0}_{U_{0,1}}\upa_{U_{0,1}}^{U_1} \ar[dd]^{K_{0,1}}  &  A_{1}^{\Delta^1}\dow^{U_1}_{U_{0,1}}\upa_{U_{0,1}}^{U_1} \ar[ddl]^{\beta_0}\ar[ddr]^{\beta_1}&
A_1 \ar[dd]\\
\\
A_1 \dow^{U^1}_{U_{0,1}}\upa_{U_{0,1}}^{U_1} & \empty & A_1 \dow^{U^1}_{U_{0,1}}\upa_{U_{0,1}}^{U_1}
}$$
Let $P$ be the pullback of left side of the diagram. Since $K_{0,1}$, $\beta_0$ and $\beta_1$ are quasi-isomorphisms and $\beta_0$ is surjective on stalks,
we get that the map induced by $\beta_1$, $b_1:P\to   A_1 \dow^{U^1}_{U_{0,1}}$ is a quasi-isomorphism. Further, it is easy to see  that $b_1$ is also surjective on stalks. Thus as in the case $i=0$ we get that the map $(A_0\coprod_{K_{0,1}}A_1)\dow^{X}_{U_1}  = P\times_{A_1 \dow^{U^1}_{U_{0,1}}} A_1 \to A_1$ is a quasi-isomorphism.
\end{proof}
We shall now take care of the induction step in the proof

\begin{prop}
Let $\mathfrak{U} = \{U_0, U_1,\dots,U_n\}$ be a finite open cover of $X$.
Let $(A_0,\dots,A_n;\{K_I\}_I)$ be gluing data on  $\mathfrak{U}$.
Consider the finite cover of $X$, $$\mathfrak{V} = \{V_0:=U_0, V_1:=U_1,\dots,V_{n-2}:=U_{n-2},V_{n-1}:=U_{n-1} \cup U_n\}.$$
There exist gluing data  $(B_0,\dots,B_{n-1};\{L_I\}_I)$ on $\mathfrak{V}$,
such that $B_{n-1}  = A_0\coprod_{K_{0,1}}A_1$ and $B_i = A_i$ for $0\leq i\leq n-2$,
\end{prop}
\begin{proof}
The proof will be done by an explicit construction. A complete description of the construction will take some effort and place. In order to simplify the exposition and improve readability we will follow the following notations and guidelines.
\begin{enumerate}
\item We shall construct the required data, but leave checking the compatibility conditions to the reader.
\item We shall use repeatedly properties  (1),(2) and (3)  from Lemma ~\ref{l:updown}. When we show that two objects are isomorphic by using one or more of these properties will use subscript next to the equal sign: e.g.\  if $ O \subset U \subset W \subset X$ we write
    $$\mathfrak{M}_{U}(A\dow^W_U,B\dow^W_O \upa_O^U) =_{1,2} \mathfrak{M}_{O}(A\dow^W_O,B\dow^W_O).$$
\item When two mapping spaces are isomorphic by applying natural adjunctions we will abuse notation and treat them as equal.
\end{enumerate}
Since $$B_i := A_i,\quad 0 \leq i \leq n-2,$$
$$B_{n-1}  = A_0\coprod_{K_{0,1}}A_1,$$
we are left with defining $L_I$ for every   $\varnothing \neq I \subset \{0,\dots,n-1\}$, $|I| >1$.
If $n-1 \not \in I$ we will just take $L_I := K_I$.
Now assume $n-1 \in I$; we denote $I_+ := I \cup \{n\}$, $I_-  = I_+ \backslash \{n-1\}$.
Note that we then have $V_I  = U_{I} \cup U_{I_-}$,  $U_{I} \cap U_{I_-} = U_{I_+}$.
For every $I$ such that $n-1 \in I$ we require a map
$$L_{I} \in \mathcal{S}(N(P_I), \mathfrak{M}_{V_I}(B_{n(I)}\dow^{U_{n(I)}}_{V_I},B_{n-1}\dow^{V_{n-1}}_{V_I})). $$
Note that $|I|>1$ so $n(I)< n-1$ and $B_{n(I)} = A_{n(I)}$. Further, since $B_{n-1}  = A_0\coprod_{K_{0,1}}A_1$ is a limit, $L_I$ can be described as three maps
$$L^{0}_{I} \in \mathcal{S}(N(P_I),\mathfrak{M}_{V_I}(A_{n(I)}\dow^{U_{n(I)}}_{V_I},(A_{n-1}\upa_{U_{n-1}}^{V_{n-1}})\dow^{V_{n-1}}_{V_I})), $$
$$L^{+}_{I} \in \mathcal{S}(N(P_I), \mathfrak{M}_{V_I}(A_{n(I)}\dow^{U_{n(I)}}_{V_I},(A_{n}^{I}\dow^{U_{n}}_{U_{n-1,n}}\upa_{U_{n-1,n}}^{V_{n-1}})\dow^{V_{n-1}}_{V_I})), $$
$$L^{-}_{I} \in \mathcal{S}(N(P_I), \mathfrak{M}_{V_I}(A_{n(I)}\dow^{U_{n(I)}}_{V_I},(A_n \upa_{U_n}^{V_{n-1}})\dow^{V_{n-1}}_{V_I}) ),$$
satisfying certain compatibility conditions. Again we will construct $L^0_I,L^{+}_I$ and $L^{-}_I$ and leave
checking the compatibility to the reader.
Let us start with defining $L^{+}_I$.
Note that we have
$$(A_{n}^{I}\dow^{U_n}_{U_{n-1,n}}\upa_{U_{n-1,n}}^{V_{n-1}})\dow^{V_{n-1}}_{V_I} =_{2,3}  $$
$$A_{n}^{I} \dow^{U_n}_{U_{I_+}}\upa_{U_{I_+}}^{V_I}. $$
Thus
$$\mathfrak{M}_{V_{I}}(A_{n(I)}\dow^{U_{n(I)}}_{V_I},(A_{n}^{I}\dow^{U_n}_{U_{n-1,n}}\upa_{U_{n-1,n}}^{V_{n-1}})\dow^{V_{n-1}}_{V_I}) =_{2,3} $$
$$\mathfrak{M}_{V_{I}}(A_{n(I)}\dow^{U_{n(I)}}_{V_I},A_{n}^{I} \dow^{U_n}_{U_{I_+}}\upa_{U_{I_+}}^{V_I} ) =_1 $$
$$\mathfrak{M}_{U_{I_1}}(A_{n(I)}\dow^{U_{n(I)}}_{V_{I_+}},A_{n}^{I} \dow^{U_n}_{U_{I_1}}). $$
Now we have:
$$S(N(P_I), \mathfrak{M}_{V_{I}}(A_{n(I)}\dow^{U_{n(I)}}_{V_I},(A_{n}^{I}\dow^{U_n}_{U_{n-1,n}}\upa_{U_{n-1,n}}^{V_{n-1}})\dow^{V_{n-1}}_{V_I})) = $$
$$S(N(P_I), \mathfrak{M}_{U_{I_+}}(A_{n(I)}\dow^{U_{n(I)}}_{V_{I_+}},A_{n}^{I} \dow^{U_n}_{U_{I_+}} )) = $$
 $$= S(N(P_I)\times I, \mathfrak{M}_{U_{I_+}}(A_{n(I)}\dow^{U_{n(I)}}_{V_{I_+}},A_{n} \dow^{U_n}_{U_{I_+}} ))  =$$
 $$ =S(N(P_{I_+}), \mathfrak{M}_{U_{I_+}}(A_{n(I)}\dow^{U_{n(I)}}_{V_{I_+}},A_{n} \dow^{U_n}_{U_{I_+}} )).  $$
Thus we can take $L^{+}_I = K_{I_+}$.

We shall now define
$$L^{0}_{I} \in \mathcal{S}(N(P_I),\mathfrak{M}_{V_I}(A_{n(I)}\dow^{U_{n(I)}}_{V_I},(A_{n-1}\upa_{U_{n-1}}^{V_{n-1}})\dow^{V_{n-1}}_{V_I}) ). $$
Note that we have
$$(A_{n-1}\upa_{U_{n-1}}^{V_{n-1}})\dow^{V_{n-1}}_{V_I} =_3 A_{n-1}\dow^{U_{n-1}}_{U_I}\upa_{U_I}^{V_I}.$$
Thus
$$\mathfrak{M}_{V_I}(A_{n(I)}\dow^{U_{n(I)}}_{V_I},(A_{n-1}\upa_{U_{n-1}}^{V_{n-1}})\dow^{V_{n-1}}_{V_I})=_3 $$
$$\mathfrak{M}_{V_I}(A_{n(I)}\dow^{U_{n(I)}}_{V_I},A_{n-1}\dow^{U_{n-1}}_{U_I}\upa_{U_I}^{V_I})=_{1,2} $$
$$\mathfrak{M}_{U_I}(A_{n(I)}\dow^{U_{n(I)}}_{U_I},A_{n-1}\dow^{U_{n-1}}_{U_I}).$$
Thus we can take $L^0_{I}= K_I.$

Finally we define
$$L^{-}_{I} \in \mathcal{S}(N(P_I), \mathfrak{M}_{V_I}(A_{n(I)}\dow^{U_{n(I)}}_{V_I},(A_n \upa_{U_n}^{V_{n-1}})\dow^{V_{n-1}}_{V_I})). $$
Note that
$$(A_n \upa_{U_n}^{V_{n-1}})\dow^{V_{n-1}}_{V_I} =_3 A_n \dow^{U_n}_{U_{I_-}} \upa_{U_{I_-}}^{V_{I}}.$$
Thus
$$\mathfrak{M}_{V_I}(A_{n(I)}\dow^{U_{n(I)}}_{V_I},(A_n \upa_{U_n}^{V_{n-1}})\dow^{V_{n-1}}_{V_I}) =_{3}$$
$$\mathfrak{M}_{V_I}(A_{n(I)}\dow^{U_{n(I)}}_{V_I}, A_n \dow^{U_n}_{U_{I_-}} \upa_{U_{I_-}}^{V_{I}}) =_{1}$$
$$\mathfrak{M}_{U_{I_-}}(A_{n(I)}\dow^{U_{n(I)}}_{V_I},A_n \dow^{U_n}_{U_{I_-}}).$$
So we can take $L^{-}_I = K_{I_-}$.

\end{proof}
\end{proof}

Above we discussed the process  of gluing  inductively, now we can also describe the final result:

\begin{lemma}\label{ss:All}
Let $(A_0,\dots,A_n;\{K_I\}_I) $ be gluing data on $\mathfrak{U} = \{U_0, U_1,\dots,U_n\}$.
Then the glued sheaf of differential $P_0$-algebra $\mathbf{A} \in \mathfrak{A}(X)$ , constructed in the previous section in the equalizer of a diagram
$$\prod \limits_{I \subset [n], 1\leq |I| } A_{x(I)}^{(\Delta^1)^{|I|-1}} \dow^{U_{x(I)}}_{U_I} \upa_{U_I}^{X} \rightrightarrows
\prod \limits_{I \subset [n], 1\leq |I| } A_{x(I)}^{\partial((\Delta^1)^{|I|-1})} \dow^{U_{x(I)}}_{U_I} \upa_{U_I}^{X} $$
one of the map is induced by the inclusions $\partial((\Delta^1)^{|I|-1}) \subset (\Delta^1)^{|I|-1}$ and the second one is defined by
 $(A_0,\dots,A_n;\{K_I\}_I) $
 \end{lemma}
 \begin{proof}
By induction on $n$.
 \end{proof}

\subsection{Gluing as a functor}
Given a finite cover $\mathfrak{U} = \{U_0, U_1,\dots,U_n\}$, gluing is a way to construct from gluing data on $\mathfrak{U} = \{U_0, U_1,\dots,U_n\}$  new gluing data on  $\mathfrak{V} = \{U_0, U_1,\dots,U_{n-2},U_{n-1} \cup U_{n} \}$. It is clear from the construction that it is functorial with respect to morphisms of gluing data. However, the collection of  gluing data has an additional structure as a category object in simplicial sets and to understand the behavior of gluing with respect to  this structure one needs additional care. Let us e.g.\ assume that we have  $\mathfrak{U} = \{U_0, U_1\}$.
A path in $(\mathbb{G}(\mathfrak{U}))_1$ is given by sheaves of $P_0$-algebras $A_i \in \mathfrak{A}(U_i)$ and a map
$$K_{0,1}:A_0\dow_{U_{0,1}} \to A_1^{\Delta^1}\dow_{U_{0,1}},$$
by restricting this path to its two endpoints we get two different gluing data,
 $$K^0_{0,1}:A_0\dow_{U_{0,1}} \to A_1\dow_{U_{0,1}},$$
  $$K^1_{0,1}:A_0\dow_{U_{0,1}} \to A_1\dow_{U_{0,1}}.$$
 We shall denote the result of gluing according to  $K^i_{0,1}$ by $B^i \in \mathfrak{A}(U_0 \cup U_1) $.
 There is no simplicial path connecting $B^0$ and $B^1$, but it is possible to connect them  by a zigzag of maps.
 Namely the map $$K_{0,1}:A_0\dow_{U_{0,1}} \to A_1^{\Delta^1}\dow_{U_{0,1}}$$ naturally defines a map
  $$K^{0,1}_{0,1}:A_0^{\Delta^1}\dow_{U_{0,1}} \to A_1^{\Delta^1}\dow_{U_{0,1}}.$$
  We denote the resulting gluing by $B^{0,1} \in \mathfrak{A}(U_0 \cup U_1) $.
  Note that restricting to the endpoints results in a zigzag:
  $$B^0 \leftarrow B^{0,1} \rightarrow B^1.$$
  Similarly a triangle   $$K_{0,1}:A_0\dow_{U_{0,1}} \to A_1^{\Delta^2}\dow_{U_{0,1}}$$
  gives rise to a commutative diagram of the form :
  $$\xymatrix{
  & & B^{1} & & \\
  &B^{0,1} \ar[ldd]\ar[ru]& & B^{0,2}\ar[lu] \ar[rdd]& \\
  & & B^{0,1,2} \ar[uu] \ar[lu] \ar[ru] \ar[rrd] \ar[lld]\ar[d]  & & \\
  B^{0} & & B^{0,2}\ar[rr]\ar[ll] & & B^{2}
   }$$
  To conclude, in general we have a natural map from
   $$\mathbb{N}^{\Delta}(\mathfrak{U})_m \to \mathcal{S}(\Delta_m \times sd \Delta^m, \mathbb{N}^{\Delta}(\mathfrak{V})),$$
   where $sd$ denote the barycenteric subdivision.
   Thus we get a map 
       $$\mathbb{N}^{\Delta}(\mathfrak{U}) \to \mathbb{N}^{\Delta}(\mathfrak{V}),$$
   defined up to homotopy, ``realizing'' the gluing construction.

\subsection{Gluing as homotopy limit}
In this section we shall describe the process of gluing as a certain homotopy limit in the category
$\mathfrak{A}(X)$. Usually when discussing homotopy limits one uses the language of model categories.
However, the definition of a homotopy limit actually depends only on the weak equivalences, and can be defined using only them (see ~\cite{Kan}).  
Although we have not presented a model category structure on $\mathfrak{A}(X)$ we do have a natural notion of weak equivalences as the quasi-isomorphisms in  $\mathfrak{A}(X)$.

Further complication results from the fact that we take a homotopy limit of a simplicial
functor 
$$F:D_n \to \mathfrak{A}(X),$$
 where $D_n$ is a  simplicial diagram.
The notion of a limit in the enriched situation is discussed in ~\cite{Kel82} and some  general results about homotopy limits in this realm are discussed in \cite{Shu06}. In this section we shall use the following definitions. 

\begin{definition}[\cite{Kan}, \S 33]  A \emph{ homotopical category} is a category $\mathfrak{M}$ equipped
with a class of morphisms called weak equivalences that contains all the identities
and satisfies the 2-out-of-6 property i.e.: if $hg$ and $gf$ are weak equivalences, then so
are $f$, $g$, $h$, and $hgf$.
\end{definition}

The category one gets from inverting all the weak equivalences is denoted by $\mathrm{Ho}(\mathfrak{M})$.
Note that there is a natural map:

$$\delta_\mathfrak{M}: \mathfrak{M}\to \mathrm{Ho}(\mathfrak{M})$$

Given a functor $$G:\mathfrak{M}  \to \mathfrak{N}$$ between two homotopical categories we say that
$G$ is \emph{homotopical} if $G$ takes weak equivalences to weak equivalences. Similarly given a functor
$$G:\mathfrak{M}  \to \mathrm{Ho}(\mathfrak{N})$$ we say that
$G$ is \emph{homotopical} if $G$ takes weak equivalences to isomorphisms .
\begin{definition}
Let $\mathfrak{M}$ and $\mathfrak{N}$ be two homotopical category
and let
$$G: \mathfrak{M} \to \mathfrak{N}$$ be a functor.
 A \emph{right derived functor of $G$} is a functor
 $$RG : \mathfrak{M} \to  ？o(\mathfrak{N})$$ equipped
with a comparison map
$$ \delta_{\mathfrak{N}} G \to RG $$ such that $RG$ is homotopical and initial  among
homotopical functors equipped with maps from  $ \delta_{\mathfrak{N}} G$.

\end{definition}
\begin{definition}
Let $\mathfrak{M}$ and $\mathfrak{N}$ be two homotopical category
and let
$$G: \mathfrak{M} \to \mathfrak{N}$$ be a functor.
A \emph{point-set right derived functor of $G$} is a homotopical functor
$$\mathbb{R}G: \mathfrak{M} \to \mathfrak{N}$$ equipped
with a comparison map $G \to \mathbb{R}G$ such that the induced map $ \delta_{\mathfrak{N}} G \to \delta_{\mathfrak{N}}\mathbb{R}G$ makes
$\delta_{\mathfrak{N}}\mathbb{R}G$ into a right derived functor of $G$.
\end{definition}

To present gluing as a homotopy limit we shall first explain how given a gluing data $(A_0,\dots,A_n;K_I)$ one can construct a diagram $F:D_n \to \mathfrak{A}(X)$ were $D_n$ depends on $n$ alone (but $F$ depends on   $(A_0,\dots,A_n;K_I)$). Then, the result of gluing  $(A_0,\dots,A_n;K_I)$ would turn out to be the homotopy limit of this diagram i.e the result of applying a  point-set right derived functor  of the limit functor  on $F$.


Now let $\mathfrak{U} = \{U_0, U_1,\dots,U_n\}$ be a finite cover of
$X$.  Let $(A_0,\dots,A_n;\{K_I\}_I) $ be gluing data on
$\mathfrak{U}$. We shall construct a finite simplicial category $D_n$
and a simplicial functor $F:D_n \to \mathfrak{A}(X)$ where the result
of gluing along $(A_0,\dots,A_n;\{K_I\}_I) $ is a homotopy limit of
$F$.

First we shall describe $D_n$.

\begin{definition}
Let $\varnothing \neq I \subset J \subset [n]$ be two non-empty subsets  we denote
$$H(I,J) :=  \{t \in J |x(I) \leq t  \} \subset J$$
We also denote
$$\mathbb{P}_{I,J}:= P_{H(I,J)} $$
For convenience  for   $\varnothing \neq I , J \subset [n]$ such that $I \not \subset J $ we shall take $P_{I,J} = \varnothing$
to be the empty poset.
\end{definition}

\begin{definition}
for $n\geq0$ We denote by $D_n$ the simplicial category such that
\begin{enumerate}
\item The objects of $D_n$ are non-empty subsets, $I \subset [n]$,
\item Given $I,J \in Ob D_n$, we take $$\mathrm{Map}_{D_n}(I,J):= N(\mathbb{P}_{I,J}).$$
\item For $I \subset J\subset K$, we take  the composition
$$N(\mathbb{P}_{I,J}) \times N(\mathbb{P}_{J,K}) \xrightarrow{c_{I,J,K}} N(\mathbb{P}_{I,K})$$
to be the nerve of the map induced by taking union of sets. (note that in all other cases there is only one possible map) \end{enumerate}
\end{definition}

We are now ready to define our simplicial functor $F$.
$$F:D_n \to \mathfrak{A}(X)$$

First, on objects we define:

$$F(I) := A_{x(I)}\dow^{U_{x(I)}}_{U_{I}} \upa_{U_{I}}^X \in \mathfrak{A}(X)$$

now it is left to define for every $I,J$,

$$F:  N(\mathbb{P}_{I,J})= \mathrm{Map}_{D_n}(I,J) \to \mathfrak{M}_{X}(A_{x(I)}\dow^{U_{x(I)}}_{U_{I}} \upa_{U_{J}}^X,A_{x(J)}\dow^{U_{x(J)}}_{U_{J}} \upa_{U_{J}}^X) $$

Clearly it is enough to consider the case $I \subset J$.
Now since $U_J \subset U_I$, we have
$$\upa_{U_{I}}^X \dow_{U_{J}}^X  = \dow^{U_I}_{U_I \cap U_J} \upa^{U_J}_{U_I \cap U_J} = \dow^{U_I}_{U_J} $$
Thus
$$ \mathfrak{M}_{X}(A_{x(I)}\dow^{U_{x(I)}}_{U_{I}} \upa_{U_{I}}^X,A_{x(J)}\dow^{U_{x(J)}}_{U_{J}} \upa_{U_{J}}^X)  = $$
$$
\mathfrak{M}_{U_{J}}(A_{x(I)}\dow^{U_{x(I)}}_{U_{I}} \upa_{U_{I}}^X \dow_{U_{J}}^X,A_{x(J)}\dow^{U_{x(J)}}_{U_{J}} )  =
$$
$$
\mathfrak{M}_{U_J}(A_{x(I)}\dow^{U_{x(I)}}_{U_{I}} \dow^{U_{I}}_{U_J},A_{x(J)}\dow^{U_{x(J)}}_{U_{J}} )  =
$$
$$
\mathfrak{M}_{U_J}(A_{x(I)}\dow^{U_{x(I)}}_{U_{J}} ,A_{x(J)}\dow^{U_{x(J)}}_{U_{J}} )
$$

We have to define:
$$F_{I,J}:  N(\mathbb{P}_{I,J}) \to \mathfrak{M}_{U_J}(A_{x(I)}\dow^{U_{x(I)}}_{U_{J}} ,A_{x(J)}\dow^{U_{x(J)}}_{U_{J}} )$$
Now recall that for $H = H(I,J)$ we have  map

$$ K_H:N(\mathbb{P}_{I,J}) = N(P_{H}) \to \mathfrak{M}_{U_H}(A_{n(H)}\dow^{U_{n(H)}}_{U_{H}} ,A_{x(H)}\dow^{U_{x(H)}}_{U_{H}} )$$

since $H \subset J$ we have also have a map
$$  K_H \dow^{U_{H}}_{U_J}:N(\mathbb{P}_{H}) \to \mathfrak{M}_{U_J}(A_{n(H)}\dow^{U_{n(H)}}_{U_{J}} ,A_{x(H)}\dow^{U_{x(H)}}_{U_{J}}).$$
Now since $n(H) = x(I),x(H)= x(J)$, We can just take $F_{I,J} = K_H \dow^{U_{H}}_{U_J}$.
We leave it to the reader to verify that the compatibility conditions on $(A_0,\dots,A_n;\{K_I\}_I) $ insures that $F$ respects composition.

Now, along the lines of ~\cite{Shu06} we have a functor
$$\lim:\mathfrak{A}(X)^{D_n} \to \mathfrak{A}(X).$$
We shall construct $\mathbb{R}\lim$ as a special form of weighted limit, i.e., 
we shall define a ``coefficients functor"
$C:D_n \to \mathcal{S}$
such that for any $F':D_n \to \mathfrak{A}(X)$ we have $\mathbb{R}\lim(F') = \lim^C F'$,
We will get the required natural transformation $\lim \to \mathbb{R}\lim$ by taking the unique natural transformation $C \to *$, where
$*$ is the constant functor on the terminal object and by using the identification   $\lim^*F = \lim F$.

We define
$C:D_n \to \mathcal{S}$
on objects to be $C(I) := N(Q_I)$ where $Q_I$ is the poset of subsets of $I$ that contain the last element $x(I)$ .
note that $N(Q_I) \cong (\Delta^1)^{|I|-1}$, Now we need to define a map
$$N(\mathbb{P}_{I,J})\times N(Q_I) \to N(Q_J).$$
Thus it is enough to define a map
$$C_{I,J}:\mathbb{P}_{I,J}\times Q_I \to Q_J.$$
for every $\varnothing \neq I \subset J \subset [n].$
We shall take:
$$C_{I,J}(K,I') = K \cup I'.$$
We leave it to the reader to check compatibly.

\begin{lemma}
Let $(A_0,\dots,A_n;K_I)$ be gluing data and let
$$F:D_n \to \mathfrak{A}(X)$$ be the corresponding functor.
 then the weighted limit $\lim^C F \in \mathfrak{A}(X)$ is isomorphic to result of gluing of according to  $(A_0,\dots,A_n;K_I)$.
\end{lemma}
\begin{proof}
The weighed limit $\lim^C F$ can be computed as an  equalizer of the form

$$\prod_{\varnothing \neq I\subset [n]}A_{x(I)}^{N(Q_I)}\dow^{U_{x(I)}}_{U_I} \rightrightarrows \prod_{\varnothing \neq I\subset [n]}A_{x(I)}^{W_I}\dow^{U_{x(I)}}_{U_I} $$
Where $W_I$ is some simplicial set. We leave it to reader the completely formal but tedious  check that this  equalizer is isomorphic to the one described in Lemma  ~\ref{ss:All}.

\end{proof}

To complete the proof we need to show that $\lim^C$ is indeed a point set right derived functor of $\lim$.
The proof relies on two essential facts:
\begin{enumerate}
\item $D_n$ is a Reedy category and $C:D_n \to \mathcal{S}$ is Reedy cofibrant  --- this can be verified directly by computing latching objects.
\item The category $\mathfrak{A}(X)$ can be given the structure of a category of fibrant objects in a way that is compatible with the simplicial enrichment.
\end{enumerate}
However giving a complete account of this proof will exceed the scope of this appendix, thus it is  omitted.

\subsubsection{A different indexing category}
It is convenient to consider a alternative simplicial category $E_n$ which is more symmetric and comes naturally with a simplicial  functor
$$ W: E_n \to D_n$$ which is a categorical equivalence.
Thus we get that for every functor
$$F:D_n \to \mathfrak{A}(X)$$
we have
$$ \mathbb{R}\lim (F\circ W) \approx \mathbb{R}\lim (F)$$
where the homotopy limit is taken over $E_n$ in the left hand side and on $D_n$ on the right-hand side.
As a conclusion  we will have that gluing can also be considered as homotopy limit over $E_n$.

To define $E_n$ we take:
\begin{enumerate}
\item The objects of $E_n$ are all the non-empty  subsets $\varnothing \neq I \subset [n]$.
\item For $\varnothing \neq I,J \subset [n]$ we take define the poset $\mathbb{E}_{I,J}$ to be the poset of all
chains $$ I = I_0 \subset I_1 \subset \cdots \subset I_n = J$$ ordered by inclusion. We define
$$\mathrm{Map}_{E_n}(I,J) := N(\mathbb{E}_{I,J})$$
\item We take the composition to be the nerve of the map induced by taking the union of chains.
\end{enumerate}
We are now left with defining the functor
$$W:E_n \to D_n$$

On objects we shall just take $W(I) = I$. To define $W$ on morphisms it is enough to give for every $\varnothing \neq I \subset J\subset [n]$ a map

$$W_{I,J}:\mathbb{E}_{I,J}\to \mathbb{P}_{I,J}$$
we take the map:
$$W_{I,J}(I = I_0 \subset I_1 \subset \cdots \subset I_n = J) = x(I) = x(I_0) < x(I_1) < \cdots < x(I_n) = x(J) $$
note that $W$ is an isomorphism on objects and a weak equivalence on morphism spaces and thus a  categorical equivalence.

\subsection{The existence of gluing data}
In this section we shall demonstrate the existence and suitable uniqueness of gluing data in our case.
\begin{theorem}
Let  $X$ be a non-singular quasi-projective variety, $\lambda \in \Omega^1(X)$ a closed 1-form, $S_0 =  \int \lambda$.
Let  $\mathfrak{U} = \{U_0,\dots,U_n\}$ be a finite  affine cover of $X$.
Let $(M_i,T_i)$ be a BV variety with support $(U_i,S_0|_{U_i})$.
Denote $$A_i := (\mathcal{O}_{T*[-1]V_i},d_{T_i}),$$
and consider the sub simplicial sets
 $$\mathcal{G}({U_I})(T_{n(I)},T_{x(I)}) \subset \mathfrak{M}_{U_I}(A_{n(I)}\dow^{U_{n(I)}}_{U_I},A_{x(I)}\dow^{U_{x(I)}}_{U_I}).$$
Then the subspace of $\mathbb{G}(A_0,\dots,A_n)$ consisting of gluing data $\{K_I\}_I$, such that
the image of
$$K_I:N(P_I)\to  \mathfrak{M}_{U_I}(A_{n(I)}\dow^{U_{n(I)}}_{U_I},A_{x(I)}\dow^{U_{x(I)}}_{U_I})$$
lands in
$$\mathcal{G}({U_I})(T_{n(I)},T_{x(I)}) \subset \mathfrak{M}_{U_I}(A_{n(I)}\dow^{U_{n(I)}}_{U_I},A_{x(I)}\dow^{U_{x(I)}}_{U_I}),$$
is contractible. In particular it is non-empty.
\end{theorem}
\begin{proof}
First let us denote by $\mathbb{G}_d$ the space of all $d$-partial gluing data $\{K_I\}$ for $(A_0,\dots,A_n)$ such that
the image of
$$K_I:N(P_I)\to  \mathfrak{M}_{U_I}(A_{n(I)}\dow^{U_{n(I)}}_{U_I},A_{x(I)}\dow^{U_{x(I)}}_{U_I}).$$
will land
$$\mathcal{G}({U_I})(T_{n(I)},T_{x(I)}) \subset \mathfrak{M}_{U_I}(A_{n(I)}\dow^{U_{n(I)}}_{U_I},A_{x(I)}\dow^{U_{x(I)}}_{U_I}).$$
note that $\mathbb{G}_1 = *$ and our goal is to prove that $\mathbb{G}_{n+1}$ is contractible.
For this it is enough to show that  the projection map $\mathbb{G}_{d+1} \to \mathbb{G}_{d}$ is a trivial Kan fibration.
Thus we need to exhibit a lift  for diagrams of the sort:

$$\xymatrix{
\Lambda^n_m \ar[d] \ar[r] & \mathbb{G}_{d+1}\ar[d]\\
\Delta^m \ar[r] \ar@{-->}[ur] & \mathbb{G}_{d}
}$$
unraveling the definitions and compatibility conditions  and using standard adjunctions, specifying such  a lift is the same as specifying a collection of lifts:
$$\xymatrix{
\partial N(P_I) \times \Delta^m \coprod_{\partial N(P_I) \times \Lambda^n_m} N(P_I) \times \Lambda^n_m \ar[d]\ar[r] & \mathcal{G}({U_I})(T_{n(I)},T_{x(I)})\\
N(P_I)\times \Delta^m \ar@{-->}[ur]
}$$
for every $I \subset [n]$ with $|I| = d+1$.
But this all exist since   $\mathcal{G}({U_I})(T_{n(I)},T_{x(I)})$ is Kan contractible.

%
%
%
\end{proof}

\subsection{Permuting the ordering of the cover}
In this appendix we took a  minimalistic   approach to describe gluing data, i.e.,  we took the minimal amount of data required to perform the gluing. This minimality came with a price of some  symmetry breaking (as often happens). Specifically,  note that  gluing data $(A_0,\dots,A_n;K_I)$ is presented in a way which is not symmetric with respect to the ordering of the open subsets $U_0,\dots,U_n$. In this section we shall explain why the resulting glued object will be the same (up to weak equivalence) regardless  of the chosen order.

The essential point lies in the Kan contractible groupoid that we discussed above. Note that the lifting conditions we have are not sensitive to ordering. Thus one could get gluing data for any possible ordering.  Further the contractibility allows us to get homotopies that relate the different ordering. To show how this works we give a partial account of the case of two open sets, and a very rough sketch for the general case.  Let $U_0,U_1$ be a cover of $X$. By considering the groupoid above we get gluing data $(A_0,A_1;K_{0,1})$ as well as $(A_1,A_0;L_{1,0})$ when $K_{0,1}$ and $L_{1,0}$ are maps
$$K_{0,1}: A_0 \dow ^{U_0}_{U_{0,1}} \to A_1 \dow ^{U_1}_{U_{0,1}}$$
$$L_{1,0}: A_1 \dow ^{U_1}_{U_{0,1}} \to A_0 \dow ^{U_0}_{U_{0,1}}$$
Further the contractibly supply us with a homotopy:

$$H: A_0 \dow ^{U_0}_{U_{0,1}} \to A_0^{\Delta^1} \dow ^{U_0}_{U_{0,1}} $$
between the identity map and $L_{1,0}\circ K_{0,1}$ Thus we get a commutative diagram
$$
\xymatrix{
A_0\upa_{U_0}^{X}\ar[d]^{H} \ar[r]^{K_{0,1}}& A_1 \dow^{U_1}_{U_{0,1}}\ar[d]^{L_{1,0}} &\ar[l] A_1 \upa_{U_1}^X\ar[d] \\
A_0^{\Delta^1}\upa_{U_0}^{X} \ar[r]& A_0 \dow^{U_1}_{U_{0,1}} &\ar[l]^{L_{1,0}} A_1 \upa_{U_1}^X \\
A_0\upa_{U_0}^{X} \ar[r]\ar[u] & A_0 \dow^{U_1}_{U_{0,1}} \ar[u]&\ar[l]^{L_{1,0}} A_1 \upa_{U_1}^X \ar[u] \\
}
$$
where all the vertical maps are weak equivalences. 

Since the gluing along  $(A_1,A_0;L_{1,0})$ is the homotopy limit of
the bottom row  and the gluing along  $(A_0,A_1;K_{0,1})$ is the homotopy limit of
the top row, we get that the two results are weakly equivalent. In general in order to prove the contractibly  of space of results for arbitrary $n$ one uses the higher homotopies supplied by the contractible groupoid to show equivalences (or equivalence between equivalence etc.) of corresponding  diagrams $E_n \to \mathfrak{A}(X)$. This is here were  the symmetries of $E_n$ becomes useful.

Now let $\mathfrak{U} := \{U_0, U_1,\dots,U_n\}$,$\mathfrak{V} := \{V_0, V_1,\dots,V_m\}$  be two  affine coverings of $X$.  We would like to show that the result of our construction is independent of  the choice of a covering.
To see this consider the covering
\begin{align*}
\mathfrak{W}&:= \{U_0\cap V_0, U_1\cap V_0,\dots,U_n\cap V_0,
U_0\cap V_1, U_1\cap V_1,\dots,U_n\cap V_1,\dots,
\\
&U_0\cap V_m,U_1\cap V_m,\dots,
U_n \cap V_m\}.
\end{align*}
Indeed we get that the gluing with respect either $\mathfrak{U}$ or $\mathfrak{V}$  can be described by gluing along $\mathfrak{W}$ according to different orderings.

\bibliographystyle{amsplain}
\bibliography{The_Classical_Master_Equation_AMS}
		
\end{document}